\documentclass[twoside, 14pt]{article}
\usepackage{amsmath,amssymb,amsthm,mathrsfs}
\usepackage[margin=2.4cm]{geometry} 
\usepackage{lipsum}
\usepackage{titlesec,hyperref}
\usepackage{fancyhdr}
\usepackage[numbers,sort&compress]{natbib}
\usepackage{color}
\usepackage[titletoc]{appendix}

\fancypagestyle{plain}
{
	\fancyhf{}
	
}

\titleformat{\subsection}{\it}{\thesubsection.\enspace}{1.5pt}{}
\titleformat{\subsubsection}{\it}{\thesubsubsection.\enspace}{1.5pt}{}

\newtheorem{theo}{Theorem}[section]
\newtheorem{lemm}[theo]{Lemma}

\newtheorem{prop}[theo]{Proposition}
\newtheorem{rema}{Remark}[section]
\numberwithin{equation}{section}

\allowdisplaybreaks

\def\beq{\begin{equation}}
	\def\bal{\begin{aligned}}
		\def\dal{\end{aligned}}
	\def\deq{\end{equation}}
\def\beqq{\begin{equation*}}
	\def\deqq{\end{equation*}}

\def\p{\partial}

\def\al{\alpha}
\def \f{\frac}
\def \var{\varepsilon}

\def \i {\int_{\mathbb{R}^3}}

\def \ah{\alpha_h}
\def\me{\mathcal{E}}
\def\md{\mathcal{D}}
\def \wu{\partial_3 u}
\def \wb{\partial_3 b}
\def \wr{\partial_3 \vro}
\def \du{{\rm{div}} u}
\def \hs{\Lambda_h^{-s}}
\def \vro {\varrho}
\def \dl{\delta}
\def \bu{\overline{u}}
\def \bb{\overline{b}}
\def \br{\overline{\vro}}

\begin{document}
	\begin{sloppypar}
		\title{Global uniform regularity for the 3D compressible
			MHD equations near a background magnetic field \hspace{-4mm}}
		\author{$\mbox{Jincheng Gao}^1$ \footnote{Email: gaojch5@mail.sysu.edu.cn}, \quad
        	$\mbox{Xianpeng Hu}^2$ \footnote{Email: xianphu@polyu.edu.hk}, \quad
			$\mbox{Lianyun Peng}^2$ \footnote{Corresponding author. Email: lianyun.peng@polyu.edu.hk}, \quad
			$\mbox{Jiahong Wu}^3$ \footnote{Email: jwu29@nd.edu}, \\
			\quad
			$^1\mbox{School}$ of Mathematics, Sun Yat-sen University,\\
			Guangzhou 510275, China\\
			$^2\mbox{Department}$ of Applied Mathematics, 
			The Hong Kong Polytechnic University,\\
			Hong Kong 999077, China\\
			$^3\mbox{Department}$ of Mathematics,
			University of Notre Dame, \\
			Notre Dame 46556, USA\\
		}
		
		\date{}
		
		\maketitle
		
		\begin{abstract}
			{This paper resolves the global regularity problem for the three-dimensional compressible magnetohydrodynamics (MHD) equations in the three-dimensional whole space, in the presence of a background magnetic field. Motivated by geophysical applications, we consider an anisotropic compressible MHD system with weak dissipation in the $x_2$ and $x_3$ directions and small vertical magnetic diffusion.
				By exploiting the stabilizing effect induced by the background magnetic field and constructing a hierarchy of four energy functionals, we establish global-in-time uniform bounds that are independent of the viscosity in the $x_2$ and $x_3$ directions and the vertical resistivity. A key innovation in our analysis is the development of a two-tier energy method, which couples the boundedness of vertical derivatives with the decay of horizontal derivatives. The analysis of time scale, 
				together with global regularity estimates and sharp decay rates, enable us to rigorously justify the vanishing dissipation limit and derive explicit long-time convergence rates to the compressible MHD system with vanishing dissipation in the $x_2$ and $x_3$ directions and no vertical magnetic diffusion. In the absence of magnetic field and background magnetic field, the global-in-time well-posedness and vanishing viscosity limit for the 3D compressible Navier-Stokes equations with only one direction dissipation remains a challenging open problem. This work reveals the mechanism by which the magnetic field enhances dissipation and stabilizes the fluid dynamics in the global well-posedness and vanishing viscosity limit.}
			
					\vspace*{5pt}
			           
		\end{abstract}

		\tableofcontents

\section{Introduction}

Magnetohydrodynamics (MHD) describes the dynamics of electrically conducting fluids such as plasmas, liquid metals and astrophysical fluids, where the fluid motion interacts with magnetic fields. 
The governing equations couple the compressible Navier--Stokes equations with Maxwell's equations and play a central role in plasma physics, astrophysics and geophysical fluid dynamics. 
We refer to the classical monographs  and the survey articles \cite{D1993, Braginskii1965, D2001, Kulsrud2005} for a general introduction to MHD theory. 
In many physical situations the dissipation is highly anisotropic due to external forces, stratification or strong background magnetic fields. 
Such anisotropic effects arise naturally in geophysical flows, plasma confinement devices and rotating fluids, and have motivated extensive mathematical studies on anisotropic fluid models.

\medskip
In this paper we consider the following anisotropic compressible magnetohydrodynamic system in $\mathbb{R}^3$,
\begin{equation}\label{01}
	\left\{
	\begin{aligned}
		&\partial_t \rho^\varepsilon + \mathrm{div}(\rho^\varepsilon u^\varepsilon)=0,\\
		&\rho^{\varepsilon} \partial_t u^\varepsilon+\rho^\varepsilon u^\varepsilon \cdot \nabla u^\varepsilon
		-\partial_{1}^2u^\varepsilon-\varepsilon\partial_{2}^2u^\varepsilon
		-\varepsilon\partial_{3}^2u^\varepsilon-\nabla \mathrm{div}u^\varepsilon
		+\nabla P(\rho^\varepsilon) \\
		&\qquad\qquad =B^\varepsilon\cdot \nabla B^\varepsilon-\frac{1}{2}\nabla |B^\varepsilon|^2,\\
		&\partial_t B^\varepsilon + u^\varepsilon \cdot \nabla B^\varepsilon
		-\partial_{1}^2 B^\varepsilon-\partial_{2}^2 B^\varepsilon
		-\varepsilon  \partial_3^2 B^\varepsilon
		= B^\varepsilon\cdot \nabla u^\varepsilon-B^\varepsilon \mathrm{div}u^\varepsilon,\\
		&\mathrm{div} B^\varepsilon= 0 .
	\end{aligned}
	\right.
\end{equation}
Here $\rho^\varepsilon$, $u^\varepsilon$ and $B^\varepsilon$ denote the density, velocity field and magnetic field, respectively.
For simplicity we assume the pressure obeys the law
\[
P(\rho^\varepsilon)=\frac{(\rho^\varepsilon)^3}{3}.
\]
The parameter $\varepsilon\ge 0$ is assumed to be small. 
System \eqref{01} admits the physically important steady state
\begin{equation}\label{steady}
	\rho^* \equiv 1,
	\qquad 
	u^* \equiv 0,
	\qquad 
	B^* \equiv e_2 := (0,1,0).
\end{equation}
The purpose of this paper is to investigate several fundamental questions concerning perturbations around the steady state \eqref{steady}. 
Physical experiments and numerical simulations of electrically conducting fluids have revealed a remarkable stabilizing phenomenon: the presence of a strong background magnetic field can significantly suppress turbulence and stabilize the flow. 
This effect has been observed in plasma physics, liquid-metal experiments and astrophysical flows, where magnetic fields tend to damp velocity fluctuations and generate additional dissipation mechanisms. See, for instance, the experimental and numerical studies in \cite{{Alexakis-2011},{A-1942},{Alemany-Moreau-Sulem-Frisch-1979},
	{D1995},{D1997},{D2001},{Gallet-Berhanu-Mordant-2009},{Gallet-Doering-2015}}).  One of the motivations of this work is to provide a rigorous mathematical verification of this stabilizing mechanism for the anisotropic compressible MHD system \eqref{01}. 
	
\medskip
To be more precise, we introduce the perturbation variables
\[
\varrho^\varepsilon=\rho^\varepsilon-1,
\qquad 
b^\varepsilon=B^\varepsilon-e_2 .
\]
Substituting
\[
\rho^\varepsilon=1+\varrho^\varepsilon,
\qquad 
B^\varepsilon=e_2+b^\varepsilon
\]
into \eqref{01}, we obtain the following system governing the perturbation $(\varrho^\varepsilon,u^\varepsilon,b^\varepsilon)$:
\begin{equation}\label{eqr}
	\left\{
	\begin{aligned}
		&\partial_t \varrho^\varepsilon+\mathrm{div}u^\varepsilon=F_1^\varepsilon,\\
		&\partial_t u^\varepsilon-\partial_{1}^2u^\varepsilon
		-\varepsilon \partial_{2}^2u^\varepsilon
		-\varepsilon \partial_{3}^2u^\varepsilon
		-\nabla \mathrm{div}u^\varepsilon
		+\nabla \varrho^\varepsilon+\nabla b_2^\varepsilon-\partial_2 b^\varepsilon=F_2^\varepsilon,\\
		&\partial_t b^\varepsilon-\partial_{1}^2b^\varepsilon
		-\partial_{2}^2b^\varepsilon
		-\varepsilon\partial_{3}^2b^\varepsilon
		+e_2\,\mathrm{div}u^\varepsilon-\partial_2 u^\varepsilon=F_3^\varepsilon,\\
		&\mathrm{div}b^\varepsilon=0,
	\end{aligned}
	\right.
\end{equation}
where $F_i^\varepsilon$ $(i=1,2,3)$ denote the nonlinear terms given by
\[
\left\{
\begin{aligned}
	F_1^\varepsilon&=-u^\varepsilon\cdot \nabla \varrho^\varepsilon-\varrho^\varepsilon \mathrm{div}u^\varepsilon,\\
	F_2^\varepsilon&=-u^\varepsilon \cdot \nabla u^\varepsilon-\varrho^\varepsilon\nabla \varrho^\varepsilon
	-\frac{\varrho^\varepsilon}{1+\varrho^\varepsilon}
	(\partial_{1}^2u^\varepsilon+\varepsilon\partial_{2}^2u^\varepsilon+\varepsilon\partial_{3}^2u^\varepsilon)\\
	&\quad
	-\frac{\varrho^\varepsilon}{1+\varrho^\varepsilon}
	(\nabla \mathrm{div}u^\varepsilon -\nabla b_2^\varepsilon+\partial_2 b^\varepsilon)
	+\frac{1}{1+\varrho^\varepsilon}b^\varepsilon\cdot \nabla b^\varepsilon
	-\frac{1}{2(1+\varrho^\varepsilon)}\nabla |b^\varepsilon|^2,\\
	F_3^\varepsilon&=-u^\varepsilon\cdot \nabla b^\varepsilon-b^\varepsilon\mathrm{div}u^\varepsilon+b^\varepsilon\cdot \nabla u^\varepsilon.
\end{aligned}
\right.
\]
System \eqref{eqr} is supplemented with the initial data
\begin{equation}\label{eqr-i}
	(\varrho^\varepsilon,u^\varepsilon,b^\varepsilon)|_{t=0}
	=
	(\varrho_0,u_0,b_0).
\end{equation}

The goal of this paper is to address three fundamental problems for system \eqref{eqr}:

\begin{itemize}
	\item[(1)] When $\varepsilon>0$, we establish global-in-time Sobolev estimates that are uniform with respect to $\varepsilon$, which lead to global regularity, nonlinear stability of the steady state, and precise large-time decay properties.
	
	\item[(2)] In the limiting case $\varepsilon=0$, we prove the global regularity, nonlinear stability and large-time behavior of the corresponding limiting system.
	
	\item[(3)] We rigorously justify the limit of \eqref{eqr} as $\varepsilon\to0$,
		\begin{equation}\label{eqr0}
		\left\{\begin{array}{*{4}{ll}}
			\p_t \varrho^0+{\rm div}u^0=F_1^0,\\
			\p_t u^0-\p_{1}^2u^0-\nabla {\rm div}u^0
			+\nabla \varrho^0+\nabla b_2^0-\p_2 b^0=F_2^0,\\
			\p_t b^0-\p_{1}^2 b^0-\p_{2}^2 b^0
			+e_2{\rm div}u^0-\p_2 u^0=F_3^0,\\
			{\rm {div}}b^0=0,
		\end{array}\right.
	\end{equation}	
	with the force terms $F_i^0(i=1,2,3)$ defined by
	\begin{equation*}
		\left\{\begin{array}{*{4}{ll}}
			F_1^0=-u^0\cdot \nabla \varrho^0-\varrho^0 {\rm div}u^0,\\
			F_2^0=-u^0 \cdot \nabla u^0-\varrho^0\nabla  \varrho^0-\frac{\varrho^0}{1+\varrho^0}(\p_{1}^2u^0+\nabla {\rm div}u^0 -\nabla b_2^0+\p_2 b^0)
			\\
			\quad \quad
			+\frac{1}{1+\varrho^0}b^0\cdot \nabla b^0
			-\frac{1}{2(1+\varrho^0)}\nabla |b^0|^2,\\
			F_3^0=-u^0\cdot \nabla b^0-b^0{\rm div}u^0+b^0\cdot \nabla u^0.
		\end{array}\right.
	\end{equation*}	
    System \eqref{eqr0} is supplemented with the initial data
\begin{equation}\label{eqr-i0}
	(\varrho^0,u^0,b^0)|_{t=0}=(\varrho_0,u_0,b_0).
\end{equation}
	More precisely, we derive an explicit convergence rate for the solutions as $\varepsilon\to0$, and the convergence rate is uniform in time $t$.
\end{itemize}

\medskip 
As far as the authors are aware, the results obtained in this paper represent some of the first advances on the compressible Navier–Stokes and compressible MHD equations with genuinely anisotropic dissipation. These systems are not only physically relevant but also mathematically challenging. Anisotropy arises naturally in a wide range of fluid models. For instance, in geophysical flows such as oceanic and atmospheric dynamics, the horizontal viscosity is typically several orders of magnitude larger than the vertical viscosity, leading to strongly anisotropic dissipative structures. The vertical viscosity coefficient ranges from $1$ to $10^3 ,{\rm cm}^2/{\rm sec}$ while the horizontal viscosity ranges from $10^5$ to $10^8 ,{\rm cm}^2/{\rm sec}$ (see, e.g., [Chapter 4, \cite{1987Geophysical}]). This type of anisotropy is also reflected at the mathematical level in the study of the primitive equations, for which a large body of work has been devoted to partial dissipation cases; see, for example, the contributions of Cao, Li, and Titi \cite{Cao-Li-Titi-2016, Cao-Li-Titi-2017, Cao-Li-Titi-2020}.
Similarly, in plasma physics and magnetically confined fusion devices, strong background magnetic fields induce directional preferences in both viscous and magnetic diffusion, effectively suppressing transport in certain directions while enhancing it in others (see, e.g., \cite{Braginskii1965,Kulsrud2005}). Another important source of anisotropy comes from scaling limits: when modeling fluids in thin domains, layered media, or rapidly rotating regimes, appropriate rescalings of spatial variables and physical parameters naturally lead to equations with partial or directional dissipation (see, e.g., \cite{Raugel-Sell1993,Temam-Ziane1997,Temam-Ziane-2004,Babin-Mahalov-Nicolaenko1997}).

\medskip 
From a mathematical viewpoint, stability and large-time behavior for compressible Navier--Stokes and MHD systems have been extensively studied since the seminal works of Matsumura and Nishida (\cite{Matsumura-Nishida1980, Matsumura-Nishida1983}), together with the structural theory developed by Shizuta, Kawashima and others \cite{Umeda-Kawashima-Shizuta1984,Shizuta-Kawashima1985,Kawashima1984}, which established global existence and decay for small perturbations of equilibrium states under full dissipation.
However, much less is known when the dissipative structure is weakened or becomes anisotropic. In particular, the loss of dissipation in certain directions creates significant analytical difficulties, as the standard energy method no longer directly yields sufficient control of all derivatives. 
 
\medskip 
This difficulty is already present even for the 3D incompressible Navier--Stokes equations with anisotropic dissipation, for example,
 \begin{equation}\label{02}
 	\partial_t u + u \cdot \nabla u - \partial_1^2 u - \varepsilon \partial_2^2 u - \varepsilon \partial_3^2 u + \nabla p = 0, 
 	\qquad {\rm div} u = 0.
 \end{equation}
 It is not clear whether one can obtain global Sobolev bounds for solutions that are uniform with respect to $\varepsilon$. Indeed, as $\varepsilon \to 0$, equation \eqref{02} formally converges to the partially dissipative system
 \begin{equation}\label{03}
 	\partial_t u + u \cdot \nabla u - \partial_1^2 u + \nabla p = 0, 
 	\qquad {\rm div} u = 0,
 \end{equation}
 for which the global well-posedness problem remains widely open. Consequently, it is difficult to justify whether solutions of \eqref{02} converge globally in time to those of \eqref{03}, or even to establish stability uniformly in $\varepsilon$.
This issue becomes even more pronounced in the compressible setting due to the strong coupling between density and the velocity.

\medskip 
The results presented in this paper constitute a significant advance in this direction. They demonstrate that global stability and precise large-time behavior can still be achieved even in the presence of highly degenerate dissipation, where only one directional dissipation is present in the velocity equation. More importantly, our analysis reveals that the background magnetic field plays a decisive role in compensating for the lack of dissipation.  This paper provides a rigorous mathematical confirmation of the stabilizing phenomenon observed in physical experiments and numerical simulations, and highlights a fundamentally new mechanism for stability in anisotropic compressible fluid models.

\medskip 
To give a precise account of our main results, we introduce the following notations. 
Let $\nabla_h := (\partial_1,\partial_2,0)$ and $\alpha_h := (\alpha_1,\alpha_2,0)$. 
For every $m \in \mathbb{N}$, we define
\begin{equation*}
	\begin{aligned}
		\|f\|_{H^m_{tan}}^2 
		:= \sum_{0 \le |\alpha_h| \le m} \|\partial^{\alpha_h} f\|_{L^2}^2,
		\qquad
		\|f\|_{H^m}^2 
		:= \sum_{0 \le |\alpha| \le m} \|\partial^\alpha f\|_{L^2}^2.
	\end{aligned}
\end{equation*}
We also define the horizontal fractional operator $\Lambda_h^s$, $s \in \mathbb{R}$, by
\begin{equation*}
	\Lambda_h^s f(x_h)
	:= \int_{\mathbb{R}^2} |\xi_h|^s \widehat{f}(\xi_h)\, e^{2\pi i x_h \cdot \xi_h} \, d\xi_h,
\end{equation*}
where $x_h := (x_1,x_2,0)$, $\xi_h := (\xi_1,\xi_2,0)$, and $\widehat{f}$ denotes the Fourier transform of $f$ with respect to the horizontal variables.

For $\zeta \in (0,\tfrac{1}{20})$, we define the energy functional
\begin{equation*}
	\begin{aligned}
		\mathcal{E}(\varrho^{\varepsilon}, u^{\varepsilon}, b^{\varepsilon})(t)
		:=& \|(\varrho^{\varepsilon}, u^{\varepsilon}, b^{\varepsilon})(t)\|_{H^m}^2 \\
		&+ \|\Lambda_h^{-(1-\zeta)}(\varrho^{\varepsilon}, u^{\varepsilon}, b^{\varepsilon})(t)\|_{L^2}^2
		+ \|\Lambda_h^{-(1-\zeta)} \partial_3 (\varrho^{\varepsilon}, u^{\varepsilon}, b^{\varepsilon})(t)\|_{L^2}^2,
	\end{aligned}
\end{equation*}
and the dissipation functional
\begin{equation*}
	\begin{aligned}
		\mathcal{D}(\varrho^{\varepsilon},u^{\varepsilon}, b^{\varepsilon}, \varepsilon)(t)
		:=& \|(\partial_1 u^{\varepsilon}, \mathrm{div}\, u^{\varepsilon}, \nabla_h b^{\varepsilon})(t)\|_{H^m}^2 \\
		&+ \|(\partial_2 u^{\varepsilon}, \nabla_h \varrho^{\varepsilon})(t)\|_{H^{m-1}}^2
		+ \varepsilon \|(\partial_2 u^{\varepsilon}, \partial_3 u^{\varepsilon}, \partial_3 b^{\varepsilon})(t)\|_{H^m}^2.
	\end{aligned}
\end{equation*}

We briefly explain the motivation behind the definition of the energy functional $\mathcal{E}$. 
The momentum equation contains dissipation only in the $x_1$-direction, while the background magnetic field is in the $x_2$-direction. 
Through the coupling between the velocity and magnetic field equations, this background magnetic field induces a smoothing and regularizing effect in the $x_2$-direction. 
However, this effect is weaker than standard dissipation. In particular, in analogy with incompressible MHD systems, the regularization generated by this coupling is typically one derivative lower than that provided by full dissipation.
This weaker regularization in the $x_2$-direction creates significant difficulties in controlling the nonlinear terms in the momentum equation. 
The situation is even more complicated in the compressible setting due to the additional coupling with the density perturbation.

\medskip
Our strategy is to combine stability and decay mechanisms and establish both simultaneously. 
Since there is no direct regularization in the $x_3$-direction, the decay can only be obtained through the horizontal directions. 
This explains the presence of the negative-order horizontal norms in the definition of $\mathcal{E}$, which are designed to capture the decay in the low-frequency regime.
The dissipation functional $\mathcal{D}$ reflects both the genuine dissipation and the effective regularization induced by the magnetic field. 
In particular, the term $\|\partial_1 u^\varepsilon\|_{H^m}$ corresponds to the full dissipation in the $x_1$-direction, while 
$\|\partial_2 u^\varepsilon\|_{H^{m-1}}$ captures the weaker, one-derivative-lower regularization in the $x_2$-direction generated through the coupling with the magnetic field. 
Moreover, although the density equation does not possess an intrinsic dissipative structure, the presence of the term $\|\nabla_h \varrho^\varepsilon\|_{H^{m-1}}$ in the dissipation reflects an effective smoothing mechanism. This regularization is induced through the coupling with the momentum equation, where the pressure term transfers part of the velocity dissipation to the density, yielding control of its horizontal derivatives.
The terms involving $\varepsilon$ represent the small dissipative effects in the remaining directions.

\medskip
Our strategy for controlling these functionals is to separate the estimates of horizontal and vertical derivatives within a bootstrap framework. To this end, we further introduce a horizontal dissipative energy functional:
\begin{equation*}
	\begin{aligned}
		\mathcal{D}_{tan}(\varrho^{\varepsilon}, u^{\varepsilon}, b^{\varepsilon}, \varepsilon)(t)
		:=& \|(\partial_1 u^{\varepsilon}, \mathrm{div}\, u^{\varepsilon}, \nabla_h b^{\varepsilon})(t)\|_{H^{m-1}_{tan}}^2 \\
		&+ \|(\partial_2 u^{\varepsilon}, \nabla_h \varrho^{\varepsilon})(t)\|_{H^{m-2}_{tan}}^2 \\
		&+ \|\partial_3(\partial_1 u^{\varepsilon}, \mathrm{div}\, u^{\varepsilon}, \nabla_h b^{\varepsilon})(t)\|_{H^{m-2}_{tan}}^2 \\
		&+ \|\partial_3(\partial_2 u^{\varepsilon}, \nabla_h \varrho^{\varepsilon})(t)\|_{H^{m-3}_{tan}}^2 \\
		&+ \varepsilon \|(\partial_2 u^{\varepsilon}, \partial_3 u^{\varepsilon}, \partial_3 b^{\varepsilon})(t)\|_{H^{m-1}_{tan}}^2 \\
		&+ \varepsilon \|\partial_3(\partial_2 u^{\varepsilon}, \partial_3 u^{\varepsilon}, \partial_3 b^{\varepsilon})(t)\|_{H^{m-2}_{tan}}^2.
	\end{aligned}
\end{equation*}

Due to the dissipative structure in the $x_1$ and $x_2$ directions in both \eqref{eqr} and \eqref{eqr0}, it is straightforward to verify that the solution $(\varrho^\varepsilon, u^\varepsilon, b^\varepsilon)$ of \eqref{eqr} converges to $(\varrho^0, u^0, b^0)$ locally in time as $\varepsilon \to 0$. 
Our main goal is to establish this convergence globally in time. To this end, uniform-in-$\varepsilon$ global estimates play a decisive role in the vanishing dissipation limit.

\vskip .1in 
We now state our first main result concerning global-in-time uniform regularity and time decay.

\begin{theo}\label{main_result_one}
	Let $m \ge 4$ be an integer and $\zeta \in (0,\tfrac{1}{20})$. 
	Assume the initial data $(\varrho_0, u_0, b_0)$ satisfies ${\rm div} b_0 = 0$ and
	\begin{equation}\label{condition}
		\mathcal{E}(\varrho^{\varepsilon}, u^{\varepsilon}, b^{\varepsilon})(0) \le \delta_0
	\end{equation}
	for some sufficiently small constant $\delta_0 > 0$. 
	Then the system \eqref{eqr}--\eqref{eqr-i} admits a unique global solution $(\varrho^{\varepsilon}, u^{\varepsilon}, b^{\varepsilon})$ satisfying
	\begin{equation}\label{uniform_estimate}
		\mathcal{E}(\varrho^{\varepsilon}, u^{\varepsilon}, b^{\varepsilon})(t)
		+ \int_0^t \mathcal{D}(\varrho^{\varepsilon}, u^{\varepsilon}, b^{\varepsilon}, \varepsilon)(\tau)\, d\tau
		\le C \delta_0,
	\end{equation}
	and the decay estimate
	\begin{equation}\label{decay_estimate}
		\begin{aligned}
			&(1+t)^{1-\zeta}
			\Big( \|(\varrho^{\varepsilon}, u^{\varepsilon}, b^{\varepsilon})(t)\|_{H^{m-1}_{tan}}^2
			+ \|\partial_3 (\varrho^{\varepsilon}, u^{\varepsilon}, b^{\varepsilon})(t)\|_{H^{m-2}_{tan}}^2 \Big) \\
			&\quad + \int_0^t (1+\tau)^{1-2\zeta}
			\mathcal{D}_{tan}(\varrho^{\varepsilon}, u^{\varepsilon}, b^{\varepsilon}, \varepsilon)(\tau)\, d\tau
			\le C \delta_0,
		\end{aligned}
	\end{equation}
	where $C$ is a positive constant independent of $t$ and $\varepsilon$.
\end{theo}
\begin{rema}
	The background magnetic field $B^*=(0,1,0)$ provides enhanced dissipative  structure for the derivative of velocity in the $x_2$ direction in Theorem \ref{main_result_one}.
	This reveals the mechanism by which the magnetic field enhances dissipation and stabilizes the fluid dynamics.
	However, due to the absence of vertical dissipation of magnetic field, we can only obtain the  horizontal dissipation of density by using the
	pressure term in the momentum equation.
	Thus, the estimate of higher order vertical derivative of density
	is the most difficult task as one establishes the uniform estimates
	\eqref{uniform_estimate}, see the analysis in Section \ref{section-approach}. 
\end{rema}

\begin{rema}
	Due to the absence of vertical dissipation structure 
	in equation \eqref{eqr}, we need to establish 
	estimate for one-order vertical derivative of solution
	when constructing the time decay estimate.
	As the dissipation estimate for the derivative of velocity 
	in the $x_2$ direction is induced by the background magnetic field,
	we cannot obtain the decay estimate for the m-th order horizontal derivative of solution in \eqref{decay_estimate}. For the specific reasons, one can refer to Lemma \ref{lemma3.7}.
\end{rema}

\begin{rema}
	In order to establish the estimate for the negative derivative
	of solution, we need to use the  Hardy-Littlewood-Sobolev inequality in Lemma \ref{H-L} to control the nonlinear term.
	This requires the index $\zeta$
	satisfying $\zeta \in (0,1)$.
	Furthermore, we choose the index  satisfying
	$\zeta \in (0, \f{1}{20})$
	to make sure the term $\mathcal{E}_{tan}^{m-1}(\tau)^{\frac{7}{12}}\md_{tan}^{m-1}(\tau)^{\frac14}$ is integral in $L^1$ 
	with respect to time(see \eqref{3807}, here $\sigma:= 1-2\zeta, s:=1-\zeta$).
	Finally, in order to obtain the time integration
	of  $\md_{tan}^{m-1}(t)$, we require the decay rate of  $\me_{tan}^{m-1}(t)$ is faster than  that of $\md_{tan}^{m-1}(t)$(see \eqref{31001}).
	Therefore, we choose the the constant 
	$\zeta$ satisfying $\zeta \in (0, \f{1}{20})$
	in Theorem \ref{main_result_one}.
\end{rema}	

Since the uniform estimates and decay rates in Theorem \ref{main_result_one}
are independent of $\varepsilon$, the same arguments can be applied to the
limit system \eqref{eqr0}, yielding the following result.

\begin{theo}\label{main_result_two}
	Let $m \ge 4$ be an integer and $\zeta \in (0,\tfrac{1}{20})$. 
	Assume the initial data $(\varrho_0, u_0, b_0)$ satisfies ${\rm div} b_0 = 0$ and
	\begin{equation}\label{condition2}
		\mathcal{E}(\varrho^{0}, u^{0}, b^{0})(0) \le \delta_0,
	\end{equation}
	for some sufficiently small constant $\delta_0 > 0$. 
	Then the system \eqref{eqr0}--\eqref{eqr-i0} admits a unique global solution $(\varrho^{0}, u^{0}, b^{0})$ such that
	\begin{equation*}
		\mathcal{E}(\varrho^{0}, u^{0}, b^{0})(t)
		+ \int_0^t \mathcal{D}(\varrho^{0}, u^{0}, b^{0},0)(\tau)\, d\tau
		\le C \delta_0,
	\end{equation*}
	and moreover satisfies the decay estimate
	\begin{equation*}
		\begin{aligned}
			&(1+t)^{1-\zeta}\Big(
			\|(\varrho^{0}, u^{0}, b^{0})(t)\|_{H^{m-1}_{tan}}^2
			+ \|\partial_3 (\varrho^{0}, u^{0}, b^{0})(t)\|_{H^{m-2}_{tan}}^2 \Big) \\
			&\quad + \int_0^t (1+\tau)^{1-2\zeta}
			\mathcal{D}_{tan}(\varrho^{0}, u^{0}, b^{0},0)(\tau)\, d\tau
			\le C \delta_0,
		\end{aligned}
	\end{equation*}
	where $C$ is a positive constant independent of $t$ and $\varepsilon$.
\end{theo}

\begin{rema}
	The global-in-time well-posedness of the anisotropic compressible Navier--Stokes equation 
		\begin{equation*}
		\left\{\begin{array}{*{2}{ll}}
			\p_t \rho+{\rm div}(\rho u)=0,\\
			\rho \p_t u+\rho u \cdot \nabla u-\p_{1}^2u-\nabla {\rm div}u+\nabla P(\rho)=0.
		\end{array}\right.
	\end{equation*}
	 remains largely open. 
	In contrast, Theorem \ref{main_result_two} establishes the global-in-time well-posedness for the anisotropic compressible MHD system \eqref{eqr0}. 
	This highlights the stabilizing role of the magnetic field, which generates enhanced dissipation and ensures global stability.
\end{rema}

The local-in-time solutions $(\varrho^\varepsilon, u^\varepsilon, b^\varepsilon)$ of \eqref{eqr}
and $(\varrho^0, u^0, b^0)$ of \eqref{eqr0} have been extended to global solutions in 
Theorems \ref{main_result_one} and \ref{main_result_two}, respectively.
Our final objective is to justify the global-in-time convergence of 
$(\varrho^\varepsilon, u^\varepsilon, b^\varepsilon)$ to $(\varrho^0, u^0, b^0)$ as $\varepsilon \to 0$.
To this end, we establish an explicit convergence rate between \eqref{eqr} and \eqref{eqr0}, 
which is uniform in time.

\begin{theo}\label{main_result_three}
Let $m \ge 9$ be an integer and $\zeta \in (0,\tfrac{1}{20})$, assume the initial data $(\varrho_0,u_0,b_0)$ satisfies 
    ${\rm div} b_0 = 0$ and the small initial  conditions
    \eqref{condition} and \eqref{condition2}.
	Then the global solution $(\varrho^\varepsilon, u^\varepsilon, b^\varepsilon)$ of \eqref{eqr}
	converges to $(\varrho^0, u^0, b^0)$ of \eqref{eqr0} with the rate
	\begin{equation}\label{deo}
		\|(\varrho^\varepsilon-\varrho^0, u^\varepsilon-u^0, b^\varepsilon-b^0)(t)\|_{H^1}^2
		\le C \varepsilon^{\,1-\frac{1+(m-2)\zeta}{m(1-\zeta)}},
	\end{equation}
	where $C$ is a positive constant independent of $\varepsilon$ and time $t$.
\end{theo}

\begin{rema}
	The convergence rate in Theorem \ref{main_result_three} depends on the regularity index $m$. 
	In particular, for sufficiently large $m$, the rate in \eqref{deo} approaches the optimal one.
\end{rema}

\vskip .1in
The proof of Theorem 1.1 is based on a two-tier energy method that combines uniform energy estimates with time decay, tailored to the anisotropic and weakly dissipative structure of the system. Our strategy is to separate the analysis of horizontal and vertical derivatives within a bootstrap framework.
In the first step, we establish estimates for horizontal derivatives and derive enhanced dissipation effects. Although the momentum equation contains dissipation only in the $x_1$-direction, the coupling with the magnetic field yields additional regularization in the $x_2$-direction, allowing us to control $\partial_2 u^\varepsilon$ at one lower derivative level. In addition, the enhanced dissipation for the horizontal derivatives of the density perturbation is obtained through its coupling with the velocity equation, which provides control of $\nabla_h \varrho^\varepsilon$ at the same lower order. By incorporating negative-order horizontal Sobolev norms, we further capture the decay of low-frequency components and derive time decay estimates for the tangential energy.
The analysis is carried out through a sequence of refined estimates. We first obtain bounds for horizontal (tangential) derivatives, followed by estimates for purely vertical derivatives, and then treat mixed derivatives involving vertical and horizontal interactions. A key ingredient is the derivation of enhanced dissipation for the density in the horizontal directions, which relies on the special coupling structure between the velocity and density equations. We also establish dissipative estimates for vertical derivatives of both the density and the velocity, reflecting the indirect regularization mechanisms in the system.

\medskip 
In the second step, we focus on vertical derivative estimates. Due to the absence of direct dissipation in the $x_3$-direction and the weak dissipative structure of the density, this part requires a more delicate analysis. We exploit the structure of the equations to rewrite critical terms, perform integration by parts to reveal cancellations, and use the decay of horizontal derivatives obtained in the first step to control nonlinear interactions. This approach effectively couples the boundedness of vertical derivatives with the decay of horizontal ones. Finally, we close the argument through a bootstrap scheme. Assuming suitable a priori bounds on the energy and dissipation functionals, we combine the horizontal and vertical estimates to improve these bounds, which yields global-in-time uniform estimates together with time decay. This establishes the global regularity and decay properties stated in Theorem 
\ref{main_result_one}.  A more technical outline will be provided in the subsequent section.  

\medskip 
To further clarify the role of the magnetic field and to highlight the difficulty of the compressible setting, particularly due to the smallness of the parameter $\varepsilon$ in \eqref{01}, which makes the global stability problem highly challenging, we briefly recall several recent developments on the stability and decay of anisotropic incompressible MHD systems. In the incompressible case, the stabilizing effect of a strong background magnetic field has been rigorously confirmed in a broad range of works. For the ideal MHD equations, nonlinear stability results have been established in \cite{Longtime1988,HXY2018,PZZ2018,WZ2017}. Moreover, substantial progress has been made for MHD systems with partial dissipation or damping mechanisms, where well-posedness and stability have been extensively studied (cf.\cite{DZ2018,HW2010,PZZ2018,RWXZ2014,TW2018,WW2017,WWX2015,CWY2014,DLW2019,Fefferman2014,HL2013,
	Wu2021Advance,new-Abidi-Zhang2017,new-Boardman-Lin-Wu2020,
	new-Cao-Regmi-Wu2013,new-Cao-Wu2011,new-Chen-Zhang-Zhou2022,
	new-Du-Zhou2015,new-Fefferman-McCormick-Robinson-Rodrigo2017,
	new-Jiang-Jiang-Zhao2022,
	new-Lai-Wu-Zhang2021,new-Lai-Wu-Zhang2022,
	new-Li-Yang-2024,new-Lin-Ji-Wu2020,
	new-Lin-Zhang2014,new-Lin-Zhang2015,
	new-Sermange-Temam1983,new-Xie-Jiu-Liu2024,
	new-Xu-Zhang2015,new-Zhang2016,new-Hu-Lin-arxiv, Zhang-2009}).
In particular, it has been shown that small perturbations of a constant background magnetic field lead to global-in-time well-posedness and decay, even when the dissipative structure is highly degenerate. Notably, in \cite{Gao2025, Lai-Wu-Zhang-Zhao2026, Lin-Wu-Zhu-2025}, the global well-posedness of the three-dimensional anisotropic incompressible MHD equations has been established in the regime where only one directional dissipation is present in the momentum equation. These results reveal that the coupling between the velocity and magnetic field generates an effective dissipation mechanism, compensating for the lack of full viscosity.

\medskip 
We now turn to the anisotropic compressible MHD system \eqref{01}. The situation in the compressible framework is substantially more delicate than in the incompressible case. The presence of density introduces additional nonlinear couplings, and more importantly, the density equation itself lacks any intrinsic dissipative structure. As a result, the enhanced dissipation induced by the magnetic field must simultaneously control both the velocity and the density fluctuations, creating new analytical difficulties that are absent in the incompressible setting. In particular, the techniques developed for incompressible MHD systems cannot be directly adapted, and new ideas are required to capture the intricate interplay between compressibility and anisotropic dissipation.

\medskip 
To further illustrate the challenges arising from weak dissipation in \eqref{01}, we briefly recall several classical and recent results for compressible fluid models. For the compressible Navier–Stokes equations, Matsumura and Nishida \cite{Matsumura-Nishida1980, Matsumura-Nishida1983} established the global existence and uniqueness of classical solutions in three dimensions for small perturbations of a constant state in the $H^3$ framework. Later, Huang, Li, and Xin \cite{Huang-Li-Xin2012} proved global strong solutions with small initial data allowing large oscillations. More recently, Wen \cite{Wen-2025} obtained global well-posedness under scaling-invariant smallness conditions, even in the presence of far-field vacuum. 

\medskip 
We also mention recent progress on compressible Navier–Stokes equations with partial dissipation. For instance, the system
\begin{equation*}
	\left\{
		\begin{array}{ll}
			\partial_t \rho + \mathrm{div}(\rho u) = 0,\\
			\rho \partial_t u + \rho u \cdot \nabla u - \partial_1^2 u - \partial_2^2 u - \nabla \mathrm{div} u + \nabla P(\rho) = 0,
		\end{array}
		\right.
\end{equation*}
	which features dissipation in the horizontal two directions, has been recently studied in \cite{FHWW2025}, where global well-posedness and optimal decay rates were established for small perturbations of equilibrium states in the three-dimensional whole space.

\medskip 
For compressible MHD systems with full dissipation in the menentum equation, global small-data well-posedness and large-time decay have also been established (cf.\cite{{WW2017},{Wu-Zhu-2022},{Wu-Zhai-2023}}). However, these results rely heavily on the full dissipative structure in the momentum equation. In contrast, the anisotropic compressible MHD system \eqref{01} exhibits only partial dissipation in the $x_2$ and $x_3$ directions, together with weak magnetic diffusion in the vertical direction, making the problem significantly more challenging and still largely open.

\medskip 
Finally, we note that the vanishing viscosity limit for the Navier–Stokes equations remains one of the most fundamental problems in fluid mechanics, and its convergence and rates have been extensively studied (see, for instance, \cite{Constantin1986,Constantin1988,Kato1972,Masmoudi2007CMP, ConstantinWu1995,ConstantinWu1996,BeiraodaVeiga2010,BeiraodaVeiga2011,Xiao2007,
	Iftimie2011ARMA,Masmoudi2012ARMA,Xiao2013,Wang-2016,Wang-Xin-Yan2015}). Despite these substantial difficulties, our work shows that the background magnetic field still provides sufficient stabilization to yield global-in-time uniform estimates and decay, thereby extending key aspects of the incompressible theory to the compressible regime.

\medskip 
		Throughout this paper, we use symbol $A \lesssim B$ for $ A\le C B$ where $C>0$ is a constant which may change from line to line and independent of
		time $t$ and  $\varepsilon$ for $0<\varepsilon <1$.
		
\medskip 
		The rest of the paper is organized as follows.
		In Section \ref{section-approach}, we explain the difficulties and our approach to establish
		the global uniform estimate, time decay rate estimate
		for the system  \eqref{eqr} and \eqref{eqr0} respectively and the convergence rate between the two systems.
		In Section  \ref{global-estimate}, we apply the four layers of energy functionals
		to establish the global uniform estimate and decay rate estimate
		for system  \eqref{eqr} under the condition of small initial data. 
		In Section  \ref{asymptotic-behavior}, with the help of the decay rate, we obtain the specific convergence rate for the solutions
		between \eqref{eqr} and \eqref{eqr0} as the parameter $\var$ tends to zero.
		Finally, we introduce some useful inequalities in Appendix \ref{usefull-inequality}.

		\section{Difficulties and outline of our approach}\label{section-approach}
		In this section, we will explain the main difficulties of proving from  Theorem
		\ref{main_result_one} to  Theorem \ref{main_result_three} as well as our strategies for overcoming them.
		Due to the lower dissipative structure of velocity in the $x_2$ direction and density in the horizontal direction, 
		we need to build some suitable decay rate for the horizontal derivative of solution to control the weak dissipative
		structure as we deal with the nonlinear term.
		Thus, our proof will divide into the following three steps.
		
		\textbf{Step 1: Estimate of horizontal derivative and enhanced dissipation.}
		First of all, we establish the estimates for the horizontal derivative
		of density, velocity and magnetic field as well as control the nonlinear terms by energy norm that includes the vertical derivative of solution
		and the horizontal derivative of dissipation norm
		with the help of anisotropic Sobolev inequality.
		Secondly, we establish estimate for the vertical derivative of solution
		itself and its horizontal derivative.
		Thirdly, the enhanced dissipation estimate for the derivative of velocity in the $x_2$ direction and the density in the horizontal directions 
		can be obtained if one applies the momentum and  magnetic field equations respectively.
		Indeed, due to pressure term and the good effect of a background magnetic field,
		we employ the good terms  $\nabla \vro^{\var}$ in the equation $\eqref{eqr}_2$ and $\p_2 u^{\var}$ in the equation $\eqref{eqr}_3$ 
		to establish the dissipation estimate $\|(\nabla_h \vro^{\var},\p_2 u^{\var})\|_{H^{m-1}}^2$.
		Then, we can obtain the following differential inequality
		\beq\label{201}
		\frac{d}{dt}\widehat{\mathcal{E}}^{m-1}_{tan}(t)
		+\kappa \md_{tan}^{m-1}(t)\le 0.
		\deq
		Finally, in order to obtain the suitable decay estimate, we need to  establish the uniform estimate of solution
		and itself one-order derivative in the negative Sobolev space.  
		Due to the lack of vertical dissipation of the magnetic field and the lower dissipative structure of density in the horizontal directions,
		it is difficult for us to deal with the term
		$\i \hs ( \f{\vro^{\var}}{1+\vro^{\var}} \p_3 b^{\var}_2) \cdot \hs u^{\var}_3 \ dx$(see $VII_{4,5}$).
		Our method here is to substitute the quantity $\p_3 b^{\var}_2$ in $\eqref{eqr}_2$
		\beq\label{202}
		\p_3 b^{\var}_2 = -\p_t u_3^\var-
		\underset{\rm Some~good~dissipative~structure~terms}{\underbrace{\p_3 \varrho^\var
				+\p_{1}^2 u_3^\var + \var\p_{2}^2 u_3^\var + \var \p_{3}^2 u_3^\var 
				+\p_3 {\rm div}u^\var+\p_2 b_3^\var+(F_2^\var)_3}}
		\deq
		into this term.
		Then, we transfer the time derivative from $\p_t u_3^\var$
		into $\p_t \vro^{\var}$ and $\hs \p_t u^{\var}_3$ and using the 
		$\eqref{eqr}_1$ and $\eqref{eqr}_2$ again to
		deal with the first term on the right hand side of 
		\eqref{202}.
		Then, due to these four layers of energy functionals,
		we rewrite the differential inequality \eqref{201}
		as the following form
		\beqq
		\frac{d}{dt}\widehat{\mathcal{E}}^{m-1}_{tan}(t)
		+\kappa C_0^{-\frac{1}{1-\zeta}}
		\widehat{\mathcal{E}}^{m-1}_{tan}(t)^{1+\frac{1}{1-\zeta}}\le 0,
		\deqq
		which yields directly the decay estimate
		\beqq
		\widehat{\mathcal{E}}^{m-1}_{tan}(t)
		\lesssim C_0(1+t)^{-(1-\zeta)}.
		\deqq
		This estimate and the differential inequality \eqref{201}
		will help us establish the time weighted estimate
		\beq\label{203}
		(1+t)^{1-\zeta}{\mathcal{E}}^{m-1}_{tan}(t)
		+\kappa \int_0^t (1+\tau)^{1-2\zeta} \md_{tan}^{m-1}(\tau)d\tau
		\lesssim C_0.
		\deq
		This estimate helps us give the control of  nonlinear terms as one establishes  estimate for the vertical derivative.
		
		\textbf{Step 2: Estimate of vertical derivative.}
		In order to close the energy estimate, we need to establish the estimates for
		the vertical derivative of density, velocity and magnetic field.
		Due to the only horizontal dissipation of velocity and magnetic field, we can not apply the anisotropic
		Sobolev inequalities to estimate the terms 
		$$
		\i b^{\var} \cdot \nabla \p_3^m u^{\var} \, \p_3^m b^{\var} \ dx
		~{\rm and}~ \i \f{1}{1+\vro^{\var}} b^{\var}  \cdot \nabla \p_3^{m} b^{\var} \, \p_3^m u^{\var} \ dx.
		$$
		Then, we can overcome this difficulty by using 
        cancellation mechanism. To be precise, we have
		\beqq
		\i \f{1}{1+\vro^{\var}} b^{\var} \cdot \nabla \p_3^{m} u^{\var} \, \p_3^m b^{\var} \ dx +  \i \f{1}{1+\vro^{\var}} b^{\var} \cdot \nabla \p_3^{m} b^{\var} \, \p_3^m u \ dx \\ = - \i {\rm{div}}( \f{1}{1+\vro^{\var}} b^{\var}) \, \p_3^{m} u^{\var} \, \p_3^m b^{\var} \ dx.
		\deqq
		Thus, it is necessary to  add the weight $\f{1}{1+\vro^{\var}}$ to magnetic field as we establish the energy estimate, see Lemma \ref{lemma32} in detail.
		In this process, due to the lower dissipative structure of density in the horizontal directions,
		we will encounter the most challenging two terms $II_{1,1}$ and $II_{1,2}$ as follows:
		$$
		\i |\p_3^m \vro^{\var}|^2 \du^{\var} \ dx  ~{\rm{and }} ~  \i |\p_3^m \vro^{\var}|^2 \nabla_h \cdot u^{\var}_h \ dx.
		$$
		Our method is to apply the quantity $\du^{\var}$ in the density equation $\eqref{eqr}_1$
		\beqq
		\du^{\var} = -(\p_t \vro^{\var} +  u^{\var} \cdot \nabla \vro^{\var} + \vro^{\var} \, \du^{\var}),
		\deqq
		and  quantity $\p_3 \vro^{\var}$ in the velocity equation $\eqref{eqr}_2$
		\beqq \bal
		\p_3 \vro^{\var} =  - \p_t u^{\var}_3 + (F_2^{\var})_3 +  \p_1^2 u^{\var}_3 + \var (\p_2^2 u^{\var}_3 + \p_3^2 u^{\var}_3) + \p_3 \du^{\var} - \p_3 b^{\var}_2 + \p_2 b^{\var}_3,
		\dal \deqq
		into these two difficult terms respectively.  
		This procedure generates some difficult terms that only have
		weak dissipative structure, such as $ \i \vro^{\var}  \p_3^m \vro^{\var} \, \p_3 u^{\var}_3 \, \p_3^m  \vro^{\var}  \ dx$ (see \eqref{ieq-sec2-1}). Then, using the
		decay rate for the horizontal derivative \eqref{203} and the anisotropic Sobolev inequality, all of these terms can ultimately be controlled, despite the complexity and length of the process. To be precise, we handle with this term as follows
		\beqq \bal
		&\;  \int_{0}^{t} \i \vro^{\var}  \p_3^m \vro^{\var} \, \p_3 u^{\var}_3 \, \p_3^m  \vro^{\var}  \ dx d\tau\\
		\lesssim &\;\int_{0}^{t} \| \vro^{\var}\|_{L^{\infty}} \| \p_3^m \vro^{\var}\|_{L^2}   \| \p_3^m \vro^{\var}\|_{L^2} \| \p_3 u^{\var}_3\|_{L^{\infty}} d\tau\\
		\lesssim &\;  \int_{0}^{t}\|(\vro^{\var}, \p_3 \vro^{\var}, \p_3^{m} \vro^{\var})\|_{L^2}^{\f94} \|(\nabla_h \vro^{\var},\nabla_h \p_3 \vro^{\var}, \p_3 u^{\var}_3,\p_3^2 u^{\var}_3)\|_{H_{tan}^1}^{\f74} d\tau\\
		\lesssim &\; \left\{\underset{0\le \tau \le t}{\sup} \|(\vro^{\var}, \p_3 \vro^{\var}, \p_3^{m} \vro^{\var})\|_{L^2} \right\}^{\f94}
		\left\{\int_{0}^{t} \|(\nabla_h \vro^{\var},\nabla_h \p_3 \vro^{\var}, \p_3 u^{\var}_3,\p_3^2 u^{\var}_3)\|_{H_{tan}^1}^{2} (1+\tau)^{1-2\zeta} d \tau \right\}^{\f78} \\
		&\; \times \left\{\int_{0}^{t} (1+\tau)^{-7(1-2\zeta)} d \tau \right\}^{\f18}, 
		\dal \deqq
		where $\zeta \in (0, \f{1}{20})$.
		At this time, the decay in time estimate of horizontal derivative \eqref{203}
		will return to balance out the growth of derivative estimate.
		In other words, 
		our method here constructs a two-tier energy method that couples the boundedness
		of derivative to the decay of horizontal derivative, the latter of which is necessary to balance out the growth of derivative.
		
		\textbf{Step 3: Convergence rate}. Let us focus on the convergence rate of solutions between
		\eqref{eqr} and \eqref{eqr0}.
		Let $(\vro^{\var}, u^\var, b^\var)$ and $(\vro^0, u^0, b^0)$ be the global solutions of systems \eqref{eqr}
		and \eqref{eqr0} with the same initial data assumption. By energy method, one can establish the following estimate
		\beqq
		\begin{aligned}
			&\; \|(\vro^\var-\vro^0, u^\var-u^0, b^\var-b^0)(t)\|_{H^1}^2 \\
			\le
			&\; C \var \int_0^t \|(\p_2^2u^\var, \p_3^2 u^\var, \p_3^2b^\var)(\tau)\|_{H^1} \|(\vro^\var-\vro^0, u^\var-u^0, b^\var-b^0)(\tau)\|_{H^1} d\tau \exp\left\{\int_0^t (\mathcal{{A}}(\tau)+\mathcal{{B}}(\tau)
			) d\tau\right\}.
		\end{aligned}
		\deqq
		In order to obtain the time integrability of $\mathcal{{A}}(t)$ and $\mathcal{{B}}(t)$, we require the regularity index $m \ge 9$.
		Define $T^{*}:=\var^{-\f{m-1}{m(1-\zeta)}}$,
		we establish the energy estimate in two time scales $(0, T^*]$ and $[T^*, + \infty)$
		motivated by \cite{WZ2023, BGM2017}. 
		For $ t \in (0, T^*]$,   which means that $1+t \lesssim \var^{-\f{m-1}{m(1-\zeta)}}$, we transform the grow-in-time rate into the coefficient $\var$. To be more precise, due to the time decay estimate built before, we have 
		\beqq \bal
		\|(\vro^\var-\vro^0, u^\var-u^0, b^\var-b^0)(t)\|_{H^1}^2 
		\lesssim  \var (1+t)^{\f{1+(m-2)\zeta}{m-1}} \lesssim \var^{1-\f{1+(m-2)\zeta}{m (1-\zeta)} }.
		\dal \deqq
		For $ t \in [T^*, + \infty)$, applying the above estimate and the relation $(1+T^*)^{-1} \lesssim \var^{\f{m-1}{m(1-\zeta)}} $, we  have
		\beqq \bal 
		&\; \|(\vro^\var-\vro^0, u^\var-u^0, b^\var-b^0)(t)\|_{H^1}^2     \\  \lesssim 
		&\; \left( \mathcal{\overline{E}}_q(T^*) + \var \int_{T^*}^{t} \|(\p_2^2u^\var, \p_3^2 u^\var, \p_3^2 b^\var)(\tau)\|_{H^1} \|(\vro^\var-\vro^0, u^\var-u^0, b^\var-b^0)(\tau)\|_{H^1} d\tau \right) \exp\left\{\int_{T^*}^{t} (\mathcal{{A}}(\tau)+\mathcal{{B}}(\tau)
		) d\tau\right\} \\ 
		\lesssim &\;  \var^{1-\f{1+(m-2)\zeta}{m (1-\zeta)} } + \var^{1-\f{m-3}{2(m-2)}} (1+ T^*)^{ -\f{m-4 -(3m-8)\zeta}{2(m-2)}}  \lesssim \var^{1-\f{1+(m-2)\zeta}{m (1-\zeta)} },
		\dal \deqq 
		where we transform the decay-in-time rate into the coefficient $\var$.
		Therefore, combining the decay rate in these two time scales, we establish the specific convergence rate depending on the parameter $\var$.
		
		\section{Global in time uniform regularity}\label{global-estimate}
		
		In this section, we will establish global-in-time well-posedness for the equations \eqref{eqr}
		under the small initial data \eqref{condition}.
		First of all, 
		one can establish the uniform (with respect to $\var$)
		local-in-time existence and uniqueness for system \eqref{eqr}.
		Then, we will extend the local-in-time solution to be global one.
		Thus, our target in this section is to establish the global-in-time
		uniform regularity under the condition of small initial data  \eqref{condition}.
		For notational convenience, we drop the superscript $\var$ throughout this section.
		Let us define the energy norms
		\beqq
		\begin{aligned}
			\me_{tan}^{k}(t)
			:=&\|(\vro,u,b)(t)\|_{H^{k}_{tan}}^2+\|\p_3 (\vro,u,b)(t)\|_{H^{k-1}_{tan}}^2, \\
			\mathcal{E}^{k}(t)
			:= &\|(\vro, u,b)(t)\|_{H^{k}}^2 
            +\|\Lambda_h^{-(1-\zeta)}( \vro, u, b)(t)\|_{L^2}^2
            +\|\Lambda_h^{-(1-\zeta)}\p_3 ( \vro, u, b)(t)\|_{L^2}^2,
		\end{aligned}
		\deqq
		and the dissipation norms
		\beqq
		\begin{aligned}
			\mathcal{D}_{tan}^{k}(t)
			:=&\|(\p_1 u, {\rm{div}} u, \nabla_h b)(t)\|_{H^k_{tan}}^2 + \var \|(\p_2 u, \p_3 u, \p_3 b)(t)\|_{H^k_{tan}}^2
			+\|\p_3(\p_1 u, \du, \nabla_h b)(t)\|_{H^{k-1}_{tan}}^2\\
			&\;+\var \|\p_3(\p_2 u, \p_3 u,
			\p_3 b)(t)\|_{H^{k-1}_{tan}}^2 +\|(\p_2 u,\nabla_h \vro)(t)\|_{H^{k-1}_{tan}}^2+ \|\p_3(\p_2 u,\nabla_h \vro)(t)\|_{H^{k-2}_{tan}}^2,\\
			\md^{k}(t)
			:=&\|(\p_1 u, {\rm{div}} u, \nabla_h b)(t)\|_{H^k}^2 +  \var \|(\p_2 u, \p_3 u, \p_3 b)(t)\|_{H^k}^2 + \|(\p_2 u,\nabla_h \vro)(t)\|_{H^{k-1}}^2.
		\end{aligned}
		\deqq
		Now, let us state the global uniform estimate as follows. 
		\begin{prop}\label{main_pro}
			Let $m \ge 4$ be an integer and $0<\zeta<\frac{1}{20}$, assume the initial data $(\vro_0, u_0, b_0)$ satisfying
			${\rm div} b_0=0$.
			For the solution $(\vro,u, b)$ of equation \eqref{eqr}
			defined on $ [0,T] \times \mathbb{R}^3$, assume there exists
			a small positive constant $\delta$ such that for any $t \in (0, T]$
			\begin{equation}\label{assumption}
				\underset{0\le \tau \le t}{\sup}\mathcal{E}^m(\tau)
				+\underset{0\le \tau \le t}{\sup}[(1+\tau)^{s} \mathcal{E}_{tan}^{m-1}(\tau)]
				+\int_0^t (1+\tau)^{\sigma} \md_{tan}^{m-1}(\tau)d\tau
				+\int_0^t \md^{m}(\tau)d\tau \le \delta,
			\end{equation}
			then this solution will satisfy the estimate
			\begin{equation}\label{close_assumption}
				\underset{0\le \tau \le t}{\sup}\mathcal{E}^m(\tau)
				+\underset{0\le \tau \le t}{\sup}[(1+\tau)^{s} \mathcal{E}_{tan}^{m-1}(\tau)]
				+\int_0^t (1+\tau)^{\sigma} \md_{tan}^{m-1}(\tau)d\tau
				+\int_0^t \md^{m}(\tau)d\tau \le \frac{\delta}{2},
			\end{equation}
			where  the small positive constant 
            $$\delta:=8C(\|(\vro, u,  b)(0)\|_{H^m}^2
			+\|(\hs( \vro, u, b), \hs \p_3 ( \vro, u, b))(0)\|_{L^2}^2+ \|\hs u(0)\|_{L^2}\|(\vro, u)(0)\|_{H^m}^2),$$
            and $s:=1-\zeta$, $\sigma:=1-2\zeta$.
			Here $C$ is a positive constant independent of time $t$ and
			parameter $\var$.
		\end{prop}
		
		\subsection{Horizontal and vertical derivative estimate}\label{sec-tan-normal}
		It is well-known that the norm $ \|f\|_{H^m}$ is equivalent to the norms $\|f\|_{H_{tan}^m}$ and $\|  \p_3^m f\|_{L^2}$.
		In this subsection, we will establish the estimate for the horizontal derivative of density, velocity and magnetic field.
		\begin{lemm}\label{lemma32}
			Under the assumption \eqref{assumption}, for any smooth solution $(\vro ,u, b)$ of equation \eqref{eqr},
			it holds for the positive integer $k = m $ or $ k = m-1$
			\beq\label{3101} \bal
			&\; \f12 \frac{d}{dt}(\| (\vro, u, \f{1}{\sqrt{1+\vro}} b)\|_{L^2}^2 + \sum\limits_{|\alpha_h| = k}\| (\p^{\ah}\vro, \p^{\ah} u, \f{1}{\sqrt{1+\vro}}\p^{\al_h} b)\|_{L^2}^2 ) 
			\\
			&\; +  \|(\p_1 u, \du,\nabla_h b)\|_{H_{tan}^k}^2 + \var \|(\p_2 u, \p_3 u,\p_3 b)\|_{H_{tan}^k}^2
			\lesssim  \sqrt{\me^m(t)}\md_{tan}^{k}(t).
			\dal \deq
		\end{lemm}
		\begin{proof}
			For any $|\al_h| \le m$ , applying the $\p^{\al_h}$-operator to the equation $\eqref{eqr}_3$ and multiplying by  $\f{1}{1+ \vro}$,  we have
			\beqq \bal
			&\; \f{1}{1+ \vro}  \p_t \p^{\al_h} b +  e_2 \p^{\al_h} \du-\Delta_h \p^{\al_h} b -\var \p_3^2 \p^{\al_h} b
			- \p_2 \p^{\al_h} u  \\
			= &\; \f{1}{1+ \vro} \p^{\al_h}( b \cdot \nabla u-b \, \du - u \cdot \nabla b) + \f{\vro}{1+\vro} (e_2 \p^{\al_h} \du-\Delta_h \p^{\al_h} b 
			- \var \p_3^2 \p^{\al_h} b + \p_2 \p^{\al_h} u ).
			\dal \deqq	
			Using the density equation $\eqref{eqr}_1$, it is easy to check that 
			\beqq \bal
			\i  \f{1}{1+ \vro} \p_t \p^{\al_h} b \cdot \p^{\al_h} b\ dx  = &\; \f{d}{dt} \f12 \i  \f{1}{1+ \vro} |\p^{\al_h} b|^2 \ dx -\i  \p_t(\f{1}{1+ \vro}) |\p^{\al_h} b|^2 \ dx \\
			= &\; \f{d}{dt} \f12 \i  \f{1}{1+ \vro} |\p^{\al_h} b|^2 \ dx - \i \f{1}{2(1+ \vro)^2}|\p^{\al_h} b|^2  (\du +u \cdot \nabla \vro + \vro \, \du)\ dx. 
			\dal \deqq
			Thus we have
			\beqq
			\begin{aligned}
				&\f{d}{dt} \f12 \i  \f{1}{1+ \vro} |\p^{\al_h} b|^2 \ dx  +\| \nabla_h \p^{\al_h} b \|_{L^2}^2 + \var \| \p_3 \p^{\al_h} b \|_{L^2}^2
				+\i( e_2 \p^{\al_h} \du- \p_2 \p^{\al_h} u  )  \p^{\al_h} b \ dx\\
				=& \i \f{1}{2(1+ \vro)^2} |\p^{\al_h} b|^2  (\du +u \cdot \nabla \vro + \vro \, \du)\ dx+\i \f{1}{1+ \vro} \p^{\ah}(b \cdot \nabla u - u \cdot \nabla b -b \, \du )\cdot \p^{\ah} b \ dx\\
				&
				+\i \f{\vro}{1+\vro} (e_2 \p^{\al_h} \du-\Delta_h \p^{\al_h} b 
				- \var \p_3^2 \p^{\al_h} b - \p_2 \p^{\al_h} u )\cdot \p^{\ah} b\ dx,
			\end{aligned}
			\deqq
			which, together with the density equation $\eqref{eqr}_1$ and velocity equation   $\eqref{eqr}_2$, yields that
			\begin{equation}\label{3102}
				\begin{aligned}
					&\frac{d}{dt}\frac{1}{2}\i(|\p^{\ah}u|^2+|\p^{\ah} \vro|^2+\f{1}{1+ \vro} |\p^{\al_h} b|^2 )dx
					+\| \p_1 \p^{\ah}u \|_{L^2}^2 +\|\p^{\ah} \du\|_{L^2}^2+\|\nabla_h \p^{\ah}b\|_{L^2}^2\\
					&\; + \var ( \| \p_2 \p^{\al_h} u \|_{L^2}^2 + \| \p_3 \p^{\al_h} u \|_{L^2}^2 + \| \p_3 \p^{\al_h} b \|_{L^2}^2)\\
					=&-\i \p^{\ah} (u\cdot \nabla \vro +\vro \, \du )\cdot \p^{\ah}\vro \ dx
					-\i \p^{\ah}( u \cdot \nabla u+ \vro \nabla \vro) \cdot \p^{\ah}  u\ dx\\
					& -\i \p^{\ah}( \f{\vro}{1+\vro}(\p_1^2 u+\nabla \du-\nabla b_2 +\p_2 b)) \cdot \p^{\ah}  u\ dx  +\i \p^{\ah} ( \f{1}{1+\vro} b \cdot \nabla b) \cdot \p^{\ah}  u\ dx  \\
					&-\f12 \i \p^{\ah} ( \f{1}{1+\vro} \nabla (|b|^2)) \cdot \p^{\ah}  u\ dx  
					+\i \f{1}{2(1+ \vro)^2} |\p^{\ah}  b|^2  (\du +u \cdot \nabla \vro + \vro \, \du)\ dx \\
					&-\i \f{1}{1+ \vro} \p^{\ah} (b \, \du + u \cdot \nabla b)\cdot \p^{\ah}  b \ dx+ \i \f{1}{1+ \vro} \p^{\ah} (b \cdot \nabla u ) \cdot \p^{\ah}  b \ dx
					\\
					&+\i \f{\vro}{1+\vro} \p^{\ah}(e_2   \du-\Delta_h   b - \var \p_3^2  b - \p_2  u )\cdot \p^{\ah}  b\ dx  - \var \i \p^{\ah}( \f{\vro}{1+\vro}(\p_2^2 u +  \p_3^2 u)) \cdot \p^{\ah}  u\ dx \\
					:= & \sum_{i=1}^{10} I_i,
				\end{aligned}
			\end{equation}
			where we have used the basic fact
			\beqq \bal
			& \i \p^{\ah} \nabla \vro \cdot \p^{\ah} u \ dx
			+\i \p^{\ah} \du \cdot \p^{\ah} \vro \ dx=0,\\
			& \i \p^{\ah} (\nabla b_2 - \p_2 b )\cdot \p^{\ah} u \ dx
			+\i \p^{\ah} (e_2 \du - \p_2 u )\cdot \p^{\ah} b \ dx=0.
			\dal \deqq
			\textbf{First of all, we deal with the case $\alpha_h = 0$. }  
			Integrating by parts and using the anisotropic type inequality $\eqref{ie:Sobolev}$, we have
			\begin{align*}
				I_1 = &\;  - \f12 \i \vro^2 \du \ dx \lesssim \| \du\|_{L^2}^{\f12}\| \p_3 \du\|_{L^2}^{\f12} \| \vro \|_{L^2}  \| \nabla_h \vro \|_{L^2},\\
				I_2 = &\; \f12 \i (|u|^2 + \vro^2 ) \du \ dx  \lesssim \| \du\|_{L^2}^{\f12}\| \p_3 \du\|_{L^2}^{\f12} \| (u,\vro )\|_{L^2}  \| \nabla_h(u, \vro) \|_{L^2},\\
				I_4 + I_8 = &\; \i (u \cdot b) \, b_h \cdot \nabla_h (\f{1}{1+ \vro})  \ dx + \i (u \cdot b) \, b_3 \cdot \p_3 (\f{1}{1+ \vro})  \ dx \\
				\lesssim &\; \| b\|_{L^{\infty}} \|u\|_{L^{2}}^{\f12} \|\p_1 u\|_{L^{2}}^{\f12} \|b_h\|_{L^{2}}^{\f12} \|\p_2 b_h \|_{L^{2}}^{\f12}  \|\nabla_h  (\f{1}{1+ \vro})\|_{L^{2}}^{\f12} \|\nabla_h \p_{3} (\f{1}{1+ \vro})\|_{L^{2}}^{\f12}  \\
				&\; + \| b\|_{L^{\infty}} \|u\|_{L^{2}}^{\f12} \|\p_1 u\|_{L^{2}}^{\f12} \|b_3\|_{L^{2}}^{\f12} \|\p_3 b_3\|_{L^{2}}^{\f12}  \|\p_3  (\f{1}{1+ \vro})\|_{L^{2}}^{\f12} \| \p_{23} (\f{1}{1+ \vro})\|_{L^{2}}^{\f12} \\
				\lesssim &\; \sqrt{\me^m(t)} \md_{tan}^{m-1} +\|b\|_{L^{\infty}} \|(u, b_3, \p_3 \vro)\|_{L^2}^{\f32} \|(\p_1 u, \p_3  b_3, \p_{23}  (\f{1}{1+ \vro}))\|_{L^2}^{\f32},\\
				\lesssim &\; \sqrt{\me^m(t)} \md_{tan}^{m-1}(t),
			\end{align*}
			where we have used the
			the anisotropic type inequality $\eqref{ie:Sobolev}_1$ to obtain that
			\beqq
			\| b\|_{L^{\infty}} \lesssim \| (b, \p_3 b)\|_{L^{2}}^{\f14} \|\nabla_h (b, \p_3 b)\|_{H_{tan}^1}^{\f34}.
			\deqq
			Now we deal with the term $I_3$.
			\beqq \bal
			I_3 =&\; \i  \Big(\p_1 u \, \p_1( \f{\vro}{1+\vro} u) + \du \, {\rm div} ( \f{\vro}{1+\vro} u) + \f{\vro}{1+\vro} (u_h \cdot \p_h b_2 + u \cdot \p_2 b ) \Big)\ dx \\
			&\; + \i  \f{\vro}{1+\vro}  u_3 \, \p_3 b_2 \ dx \\
			\lesssim &\; \|\vro\|_{L^{\infty}} \| (\p_1 u, \du)\|_{L^2}^2 
			+ \|u\|_{L^{2}}^{\f12} \|\p_1 u\|_{L^{2}}^{\f12} \|\p_1 (\f{\vro}{1+ \vro})\|_{L^{2}}^{\f12} \|\p_{12} (\f{\vro}{1+ \vro})\|_{L^{2}}^{\f12} \|\p_1 u\|_{L^{2}}^{\f12} \|\p_{13} u\|_{L^{2}}^{\f12} \\
			&\; + \|u\|_{L^{2}}^{\f12} \|\p_1 u\|_{L^{2}}^{\f12} \|\nabla  (\f{\vro}{1+ \vro})\|_{L^{2}}^{\f12} \|\nabla \p_{2} (\f{\vro}{1+ \vro})\|_{L^{2}}^{\f12} \|\du\|_{L^{2}}^{\f12} \|\p_3 \du\|_{L^{2}}^{\f12} 
			\\&\;+ \|\vro\|_{L^{2}}^{\f12} \|\p_2 \vro\|_{L^{2}}^{\f12} \|u\|_{L^{2}}^{\f12} \|\p_1 u\|_{L^{2}}^{\f12} \|\nabla_h b\|_{L^{2}}^{\f12} \|\p_3 \nabla_h b\|_{L^{2}}^{\f12} \\
			&\; + \|\vro\|_{L^{2}}^{\f12} \|\p_2 \vro\|_{L^{2}}^{\f12} \|u_3\|_{L^{2}}^{\f12} \|\p_3 u_3\|_{L^{2}}^{\f12} \|\p_3 b_2\|_{L^{2}}^{\f12} \|\p_{13} b_2\|_{L^{2}}^{\f12} \\
			\lesssim &\; \sqrt{\me^m(t)} \md_{tan}^{m-1} +\|(\vro, u_3, \p_3 b_2)\|_{L^2}^{\f32} \|(\p_2 \vro, \p_3  u_3, \p_{13} b_2)\|_{L^2}^{\f32}.
			\dal \deqq
			Due to the interpolation inequality 
			\beqq
			\|(\vro, u_3, \p_3 b_2)\|_{L^2}
			\lesssim \|\hs(\vro, u_3, \p_3 b_2)\|_{L^2}^{\frac{1}{1+s}}
			\|\nabla_h(\vro, u_3, \p_3 b_2)\|_{L^2}^{\frac{s}{1+s}},
			\deqq
			and $s > \f{9}{10}$, it is easy to check that
			\beqq \bal
			&\;  \|(\vro, u_3, \p_3 b_2)\|_{L^2}^{\f32} \|(\p_2 \vro, \p_3  u_3, \p_{13} b_2)\|_{L^2}^{\f32} \\
			\lesssim &\; \|\hs(\vro, u_3, \p_3 b_2)\|_{L^2}^{\frac{3}{2(1+s)}}
			\|\nabla_h(\vro, u_3, \p_3 b_2)\|_{L^2}^{\frac{3s}{2(1+s)}} \|(\p_2 \vro, \p_3  u_3, \p_{13} b_2)\|_{L^2}^{\f32} \lesssim \sqrt{\me^m(t)} \md_{tan}^{m-1}(t),\\
			\dal \deqq
			which yields that
			\beqq
			I_3  \lesssim \sqrt{\me^m(t)} \md_{tan}^{m-1}(t).
			\deqq
			Similarly, under the assumption \eqref{assumption}, we can check that
			\beqq \bal
			I_5 + I_6 +I_7 +I_9 +I_{10} \lesssim \sqrt{\me^m(t)}  \md_{tan}^{m-1}(t).
			\dal  \deqq
			Under the assumption \eqref{assumption},  then we have
			\begin{equation}\label{3103} \bal
				&\; \frac{d}{dt}\frac{1}{2}\i(|u|^2+\f{|b|^2}{1+\vro} + |\vro|^2)dx
				+\| (\p_1 u, \du, \nabla_h b)\|_{L^2}^2 + \var \| (\p_2  u ,\p_3 u, \p_3 b) \|_{L^2}^2
				\lesssim   \sqrt{\me^m(t)} \md_{tan}^{m-1}(t).  
				\dal \end{equation}
			\textbf{Now let us deal with the case $|\alpha_h| = k$($k=m$ or $k =m-1$).} 
			Obviously, it holds
			\beqq
			\i \p^{\ah}(-u\cdot \nabla \vro)\cdot \p^{\ah} \vro\ dx
			=-\sum_{0\le \beta_h \le \alpha_h}C^{\beta_h}_{\alpha_h}
			\i (\p^{\beta_h} u \cdot \p^{\al_h-\beta_h}\nabla  \vro)\cdot \p^{\al_h} \vro \ dx.
			\deqq
			If $\beta_h=0$, integrating by parts,
			we conclude
			\beqq
			-\i(u \cdot \p^{\al_h}\nabla \vro)\cdot \p^{\al_h} \vro\ dx
			=\frac{1}{2}\i |\p^{\al_h} \vro|^2  \du \ dx \lesssim \|\p^{\al_h} \vro \|_{L^2}^2 \| \du\|_{L^{\infty}}.
			\deqq
			If $1<|\beta_h|\le [\f{|\alpha_h|}{2}]$, the anisotropic type inequality \eqref{ie:Sobolev} yields directly
			\beqq
			\begin{aligned}
				\i (\p^{\beta_h} u \cdot \p^{\al_h-\beta_h}\nabla \vro)\cdot \p^{\al_h} \vro \ dx
				\lesssim
				\|\p^{\beta_h} u\|_{L^{\infty}}
				\|\p^{\al_h-\beta_h}\nabla \vro\|_{L^2}
				\|\p^{\al_h} \vro\|_{L^2}
				\lesssim
				\sqrt{\me^m(t)}\md_{tan}^{k}(t).
			\end{aligned}
			\deqq
			If $[\f{|\alpha_h|}{2}]<|\beta_h|\le|\alpha_h|$, we apply  the anisotropic
			type inequality \eqref{ie:Sobolev} to obtain
			\beqq
			\begin{aligned}
				&\i (\p^{\beta_h} u \cdot \p^{\al_h-\beta_h}\nabla \vro)\cdot \p^{\al_h} \vro \ dx\\
				\lesssim
				&\|\p^{\al_h} \vro\|_{L^2} \|\p^{\beta_h} u\|_{L^2}^{\frac12}\|\p_1 \p^{\beta_h} u\|_{L^2}^{\frac12}
				\|\p^{\al_h-\beta_h}\nabla \vro\|_{L^2}^{\frac14}
				\|\p_3 \p^{\al_h-\beta_h}\nabla \vro\|_{L^2}^{\frac14}
				\|\p_2 \p^{\al_h-\beta_h}\nabla \vro\|_{L^2}^{\frac14}
				\|\p_{23} \p^{\al_h-\beta_h}\nabla \vro\|_{L^2}^{\frac14}
				\\
				\lesssim
				&\sqrt{\me^m(t)}\md_{tan}^{k}(t).
			\end{aligned}
			\deqq
			Thus, the combination of above estimates yields directly
			\beqq
			\i \p^{\ah}(-u\cdot \nabla \vro)\cdot \p^{\ah} \vro\ dx 
			\lesssim \sqrt{\me^m(t)}\md_{tan}^{k}(t).
			\deqq
			It is easy to check that
			\beqq \bal
			&\;\i \p^{\ah}(- \vro \, \du) \cdot \p^{\ah} \vro\ dx \\
			\lesssim &\; \|\p^{\al_h} \vro\|_{L^2} ( \sum_{0 \le \beta_h \le [\f{\alpha_h}{2}] }\|\p^{\beta_h} \vro\|_{L^{\infty}}
			\|\p^{\al_h-\beta_h} \du\|_{L^2}
			+ \sum_{ [\f{\alpha_h}{2}] < \beta_h \le  \ah}  \|\p^{\beta_h} \vro \|_{L^2}
			\|\p^{\al_h-\beta_h} \du\|_{L^{\infty}})\\
			\lesssim &\; \sqrt{\me^m(t)}\md_{tan}^{k}(t),
			\dal \deqq
			which yields that 
			\beq \label{3104}
			I_1 \lesssim \sqrt{\me^m(t)}\md_{tan}^{k}(t).
			\deq
			Similarly, 
			integrating by parts and using the anisotropic type inequality \eqref{ie:Sobolev}, we have
			\beq \label{3105}  \bal
			I_2  = &\; \f12 \i |\p^{\al_h} u|^2  \du \ dx  - \sum_{0 < \beta_h \le \alpha_h}C^{\beta_h}_{\alpha_h}
			\i (\p^{\beta_h} u \cdot \p^{\al_h-\beta_h}\nabla  u)\cdot \p^{\al_h} u \ dx \\
			&\; + \f12 \i \p^{\ah}(\vro^2) \cdot \p^{\ah} \du \ dx \\ 
			\lesssim &\; \|\p^{\al_h} u\|_{L^2}^2 \| \du\|_{L^{\infty}} + (\sum_{0 < \beta_h \le [\f{\alpha_h}{2}] } \|\p^{\beta_h} u\|_{L^2}^{\f14}
			\|\p_3 \p^{\beta_h} u\|_{L^2}^{\f14}
			\|\p_2 \p^{\beta_h} u\|_{L^2}^{\f14}
			\|\p_{23} \p^{\beta_h} u\|_{L^2}^{\f14}
			\|\p^{\al_h-\beta_h}\nabla u\|_{L^2}
			\\ &\; + \sum_{ [\f{\alpha_h}{2}] < \beta_h \le  \ah}   \|\p^{\beta_h} u\|_{L^2}
			\|\p^{\al_h-\beta_h}\nabla u\|_{L^2}^{\f14}
			\|\p_3 \p^{\al_h-\beta_h} \nabla u\|_{L^2}^{\f14}
			\|\p_2 \p^{\al_h-\beta_h}\nabla  u\|_{L^2}^{\f14}
			\|\p_{23} \p^{\al_h-\beta_h} \nabla  u\|_{L^2}^{\f14})
			\\
			&\; \times 
			\|\p^{\ah} u\|_{L^2}^{\frac12}\|\p_1 \p^{\ah} u\|_{L^2}^{\frac12}   + \|\p^{\al_h} \du\|_{L^2} ( \sum_{0 \le \beta_h \le [\f{\alpha_h}{2}] }\|\p^{\beta_h} \vro\|_{L^{\infty}}
			\|\p^{\al_h-\beta_h} \vro\|_{L^2}\\
			&\;
			+ \sum_{ [\f{\alpha_h}{2}] < \beta_h \le  \ah}  \|\p^{\beta_h} \vro \|_{L^2}
			\|\p^{\al_h-\beta_h} \vro\|_{L^{\infty}}) \\
			\lesssim &\; \sqrt{\me^m(t)}\md_{tan}^{k}(t).
			\dal \deq
			Now we estimate the term $I_3$. 
			First, we split the term $I_3$  into three terms:
			\beqq \bal
			I_{3}=  &\;- \i \p^{\ah}( \f{\vro}{1+\vro}(\p_1^2 u+\nabla \du)) \cdot \p^{\ah}  u\ dx +\i \p^{\ah}( \f{\vro}{1+\vro}(\nabla b_2 - \p_2 b)) \cdot \p^{\ah}  u\ dx
			\\
			=  &\; -\i \f{\vro}{1+\vro}\p^{\ah} (\p_1^2 u+\nabla \du) \cdot \p^{\ah}  u\ dx - \sum_{0< \beta_h \le \alpha_h}C^{\beta_h}_{\alpha_h}
			\i \p^{\beta_h}(\f{\vro}{1+\vro}) \p^{\ah-\beta_h} (\p_1^2 u+\nabla \du) \cdot \p^{\ah}  u\ dx \\
			&\; +\i \p^{\ah}( \f{\vro}{1+\vro}(\nabla b_2 - \p_2 b)) \cdot \p^{\ah}  u\ dx\\
			:= &\;  I_{3,1} +I_{3,2} +I_{3,3}.
			\dal \deqq
			Integrating by parts, we have
			\beqq \bal
			I_{3,1}
			= &\;\i \p^{\ah} \p_1 u \cdot \p_1 (\f{\vro}{1+\vro} \p^{\ah} u) \ dx  +\i \p^{\ah} \du \cdot{\rm div} (\f{\vro}{1+\vro} \p^{\ah} u) \ dx\\
			\lesssim &\; \| \p^{\ah} \p_1 u \|_{L^2}( \| \p^{\ah} \p_1 u \|_{L^2} \| \vro\|_{L^{\infty}} + \| \p^{\ah} u \|_{L^2} \|  \p_1 \vro\|_{L^{\infty}} )\\
			&\; +  \| \p^{\ah} \du \|_{L^2}( \| \p^{\ah} \du \|_{L^2} \| \vro\|_{L^{\infty}} + \| \p^{\ah} u \|_{L^2} \|  \nabla(\f{\vro}{1+\vro}) \|_{L^{\infty}} )\\
			\lesssim &\; \sqrt{\me^m(t)}\md_{tan}^{k}(t).
			\dal \deqq 
			Under the assumption \eqref{assumption}, using the anisotropic type inequality \eqref{ie:Sobolev}, we can obtain
			\beqq \bal
			I_{3,2}
			\lesssim &\; \| \p^{\ah} u \|_{L^2}^{\f12} \| \p_1 \p^{\ah} u \|_{L^2}^{\f12} \Big( \sum_{0 < \beta_h \le [\f{\alpha_h}{2}] } \|\p^{\beta_h}(\f{\vro}{1+\vro})\|_{H^2}
			(\|\p^{\al_h-\beta_h} \nabla \du \|_{L^2} +\|\p^{\al_h-\beta_h} \p_1^2 u\|_{L^2}) \\
			&\; + \sum_{ [\f{\alpha_h}{2}] < \beta_h <   \ah}  \|\p^{\beta_h} (\f{\vro}{1+\vro}) \|_{L^2}^{\f12} \|\p_2 \p^{\beta_h} (\f{\vro}{1+\vro}) \|_{L^2}^{\f12}
			\|\p^{\al_h-\beta_h} \nabla \du\|_{L^2}^{\f12}
			\|\p_3 \p^{\al_h-\beta_h} \nabla \du\|_{L^2}^{\f12}  \\
			&\; + \sum_{ [\f{\alpha_h}{2}] < \beta_h <   \ah}  \|\p^{\beta_h} (\f{\vro}{1+\vro}) \|_{L^2}^{\f12} \|\p_2 \p^{\beta_h} (\f{\vro}{1+\vro}) \|_{L^2}^{\f12}
			\|\p^{\al_h-\beta_h} \p_1^2 u\|_{L^2}^{\f12}
			\|\p_3 \p^{\al_h-\beta_h} \p_1^2 u\|_{L^2}^{\f12}  \\
			&\;  + \|\p^{\ah} (\f{\vro}{1+\vro}) \|_{L^2} (\| \p_1^2 u\|_{L^{\infty}} + \| \nabla  \du\|_{L^{\infty}}) \Big)\\
			\lesssim &\; \sqrt{\me^m(t)}\md_{tan}^{k}(t).
			\dal \deqq
			Similarly, we have
			\beqq
			I_{3,3} = \i \p^{\ah}( \f{\vro}{1+\vro}(\nabla b_2 - \p_2 b)) \cdot \p^{\ah}  u\ dx \lesssim \sqrt{\me^m(t)}\md_{tan}^{k}(t).
			\deqq
			Thus, we have
			\beq \label{3106}
			I_{3} \lesssim \sqrt{\me^m(t)}\md_{tan}^{k}(t).
			\deq
			Similar to the estimate $I_{3,1}$ and $I_{3,2}$, integrating by parts, we have
			\beq\label{3107} \bal
			I_{10} =  &\;- \var \i \f{\vro}{1+\vro}\p^{\ah} (\p_2^2 u+ \p_3^2 u) \cdot \p^{\ah}  u\ dx -  \var \sum_{0< \beta_h \le \alpha_h}C^{\beta_h}_{\alpha_h}
			\i \p^{\beta_h}(\f{\vro}{1+\vro}) \p^{\ah-\beta_h} (\p_2^2 u+ \p_3^2 u) \cdot \p^{\ah}  u\ dx \\
			=  &\; \var \i \p^{\ah} \p_2 u \cdot \p_2 (\f{\vro}{1+\vro} \p^{\ah} u) \ dx +\var \i \p^{\ah} \p_3 u \cdot \p_3 (\f{\vro}{1+\vro} \p^{\ah} u) \ dx  \\
			&\;-  \var \sum_{0< \beta_h \le \alpha_h}C^{\beta_h}_{\alpha_h}
			\i \p^{\beta_h}(\f{\vro}{1+\vro}) \p^{\ah-\beta_h} (\p_2^2 u+ \p_3^2 u) \cdot \p^{\ah}  u\ dx \\
			\lesssim &\; \sqrt{\me^m(t)}\md_{tan}^{k}(t).
			\dal \deq
			Next, we estimate the terms $I_4$ and $I_8$ together. Integrating by parts and using the condition ${\rm div} b=0$ yield directly
			\beqq
			\i \f{1}{1+\vro}(b\cdot \nabla \p^{\ah}b) \cdot \p^{\ah}u\ dx
			+\i\f{1}{1+\vro} (b \cdot \nabla \p^{\ah}u) \cdot \p^{\ah}b\ dx= -\i b \cdot \nabla (\f{1}{1+\vro}) \p^{\ah}u \cdot \p^{\ah}b \ dx,
			\deqq
			then we achieve
			\begin{align}
				I_4 +I_8
				=&\;-\i b \cdot \nabla (\f{1}{1+\vro}) \p^{\ah}u \cdot \p^{\ah}b \ dx + 
				\sum_{0< \beta_h \le \alpha_h}C^{\beta_h}_{\alpha_h}
				\i \p^{\beta_h}(\f{1}{1+\vro}) \p^{\ah-\beta_h} (b\cdot \nabla b)
				\cdot \p^{\ah}u\ dx \notag\\
				&\;+\sum_{0< \beta_h \le \alpha_h}C^{\beta_h}_{\alpha_h}
				\i \f{1}{1+\vro} (\p^{\beta_h}b \cdot \nabla \p^{\ah-\beta_h}u
				\cdot \p^{\ah}b  + \p^{\beta_h}b \cdot \nabla \p^{\ah-\beta_h}b
				\cdot \p^{\ah}u)\ dx. \notag
			\end{align}
			Under the assumption \eqref{assumption} 
			and using the anisotropic type inequality \eqref{ie:Sobolev}, it holds
			\begin{align}\label{3108}
				I_{4}+I_{8}
				\lesssim
				&\; \|\p^{\ah}b\|_{L^2}^{\frac12}
				\|\p_3 \p^{\ah}b\|_{L^2}^{\frac12} \|\p^{\ah} u\|_{L^2}^{\frac12}
				\|\p_1 \p^{\ah} u\|_{L^2}^{\frac12} \|\nabla (\f{1}{1+\vro})\|_{L^2}^{\frac12}
				\|\p_2 \nabla (\f{1}{1+\vro})\|_{L^2}^{\frac12} \| b\|_{L^{\infty}} \notag\\
				&\; + \|\p^{\ah} u\|_{L^2}^{\frac12}
				\|\p_1 \p^{\ah} u\|_{L^2}^{\frac12}
				\Big( \sum_{0 < \beta_h \le [\f{\alpha_h}{2}] } \|\p^{\beta_h}(\f{1}{1+\vro})\|_{H_{tan}^1}^{\f12} \|\p_3 \p^{\beta_h}  (\f{1}{1+\vro})\|_{H_{tan}^1}^{\f12}
				\|\p^{\al_h-\beta_h} ( b\cdot \nabla b) \|_{L^2} \notag\\
				&\; + \sum_{ [\f{\alpha_h}{2}] < \beta_h \le  \ah}  \|\p^{\beta_h} (\f{1}{1+\vro}) \|_{L^2}
				\|\p^{\al_h-\beta_h}   ( b\cdot \nabla b) \|_{H_{tan}^1}^{\f12}
				\|\p_3 \p^{\al_h-\beta_h}   ( b\cdot \nabla b) \|_{H_{tan}^1}^{\f12} \\
				&\; +\sum_{0< \beta_h \le \alpha_h}
				\|\p^{\beta_h}b\|_{L^2}^{\frac12}
				\|\p_3 \p^{\beta_h}b\|_{L^2}^{\frac12}
				\|\nabla \p^{\ah-\beta_h}b\|_{L^2}^{\frac12}
				\|\p_2\nabla \p^{\ah-\beta_h}b\|_{L^2}^{\frac12} \Big) \notag\\
				&\;+\sum_{0< \beta_h \le \alpha_h}
				\|\p^{\beta_h}b\|_{L^2}^{\frac12}
				\|\p_2 \p^{\beta_h}b\|_{L^2}^{\frac12}
				\|\nabla \p^{\ah-\beta_h}u\|_{L^2}^{\frac12}
				\|\p_1\nabla \p^{\ah-\beta_h}u\|_{L^2}^{\frac12}
				\|\p^{\ah}b\|_{L^2}^{\frac12}
				\|\p_3 \p^{\ah}b\|_{L^2}^{\frac12} \notag\\
				\lesssim
				&\sqrt{\me^m(t)}\md_{tan}^{k}(t),\notag
			\end{align}
			and
			\beq\label{3109}
			\begin{aligned}
				I_7
				\lesssim
				&\;\|\p^{\ah}b\|_{L^2}^{\frac12}
				\|\p_3 \p^{\ah}b\|_{L^2}^{\frac12} \Big(\sum_{0 \le \beta_h \le [\f{\alpha_h}{2}] } (
				\|\p^{\beta_h}u\|_{H_{tan}^2}
				\|\nabla \p^{\ah-\beta_h}b\|_{L^2} + \|\p^{\beta_h}b\|_{H_{tan}^2}
				\|\p^{\ah-\beta_h} \du\|_{L^2}) 
				\\ 
				&\; + \sum_{[\f{\alpha_h}{2}]< \beta_h \le  \ah}
				(
				\|\p^{\beta_h}u\|_{L^2}
				\|\nabla \p^{\ah-\beta_h}b\|_{H_{tan}^2} + \|\p^{\beta_h}b\|_{L^2}
				\|\p^{\ah-\beta_h} \du\|_{H_{tan}^2}
				)\Big)\\
				\lesssim
				&\;\sqrt{\me^m(t)}\md_{tan}^{k}(t).
			\end{aligned}
			\deq
			Similarly, it is easy to check that
			\beq \label{3110}
			I_5 + I_6 + I_9 \lesssim
			\sqrt{\me^m(t)}\md_{tan}^{k}(t).
			\deq
			Thus, substituting the estimates \eqref{3104}-\eqref{3110} into \eqref{3102}, together with the estimate \eqref{3103}, 
			we conclude that
			\beqq \bal
			&\; \f12 \frac{d}{dt}(\| (\vro, u, \f{1}{\sqrt{1+\vro}} b)\|_{L^2}^2 + \sum\limits_{|\alpha_h| = k}\| (\p^{\ah}\vro ,\p^{\ah} u, \f{1}{\sqrt{1+\vro}}\p^{\al_h} b)\|_{L^2}^2 )  +  \|(\p_1 u, \du,\nabla_h b)\|_{L^2}^2 \\
			&\;
			+ \var \|(\p_2 u, \p_3 u,\p_3 b)\|_{L^2}^2 
			+  \sum\limits_{|\alpha_h| = k}(\|\p^{\ah} (\p_1 u, \du,\nabla_h b)\|_{L^2}^2+ \var \|\p^{\ah} (\p_2 u, \p_3 u,\p_3 b)\|_{L^2}^2 )\\
			\lesssim &\; \sqrt{\me^m(t)}\md_{tan}^{k}(t).
			\dal \deqq
			Therefore, we complete the proof of this lemma.
		\end{proof}
		Next,  we establish the estimate for the vertical derivative
		of density,  velocity and magnetic field.
		\begin{lemm}
			Under the assumption \eqref{assumption},
			the smooth solution $(\vro ,u, b)$ of equation \eqref{eqr}, it holds
			\beq\label{3201}
			\begin{aligned}
				&\; \i (|\p_3^m u|^2+|\p_3^m \vro|^2+\f{1}{1+ \vro} |\p_3^m  b|^2 + m\nabla_h \cdot u_h \,  \p_3^m \vro \cdot \p_3^{m-1}   u_3) (t) \ dx \\
				&\; + 
				\int_0^t \|\p_3^{m} (\p_1 u, \du, \nabla_h b)\|_{L^2}^2 d\tau + \var \int_0^t \|\p_3^{m} (\p_2 u, \p_3 u, \p_3 b)\|_{L^2}^2 d\tau\\
				\lesssim &\; \i (|\p_3^m u|^2+|\p_3^m \vro|^2+\f{1}{1+ \vro} |\p_3^m  b|^2 + m\nabla_h \cdot u_h \, \p_3^m \vro \cdot \p_3^{m-1} u_3) (0) \ dx \\
				&\; + \int_0^t \sqrt{\me^m (\tau)}\md^m (\tau) d\tau + \dl^{\f32}.
			\end{aligned}
			\deq
		\end{lemm}
		\begin{proof}
			Similar to the equation \eqref{3102}, we can obtain 
			\begin{align*}
				&\frac{d}{dt}\frac{1}{2}\i(|\p_3^m u|^2+|\p_3^m \vro|^2+\f{1}{1+ \vro} |\p_3^m  b|^2 )dx
				+\| \p_1 \p_3^mu \|_{L^2}^2 +\|\p_3^m \du\|_{L^2}^2+\|\nabla_h \p_3^mb\|_{L^2}^2\\
				&\; + \var ( \| \p_2 \p_3^m u \|_{L^2}^2 + \| \p_3^{m+1} u \|_{L^2}^2 + \| \p_3^{m+1}b \|_{L^2}^2)\\
				=&-\i \p_3^m (u\cdot \nabla \vro +\vro \, \du )\cdot \p_3^m \vro \ dx
				-\i \p_3^m ( u \cdot \nabla u+ \vro \nabla \vro) \cdot \p_3^m  u\ dx\\
				& -\i \p_3^m ( \f{\vro}{1+\vro}(\p_1^2 u+\nabla \du-\nabla b_2 +\p_2 b)) \cdot \p_3^m  u\ dx  +\i \p_3^m ( \f{1}{1+\vro} b \cdot \nabla b) \cdot \p_3^m  u\ dx  \\
				&-\f12 \i \p_3^m ( \f{1}{1+\vro} \nabla (|b|^2)) \cdot \p_3^m  u\ dx  
				+\i \f{1}{2(1+ \vro)^2} |\p_3^m  b|^2  (\du +u \cdot \nabla \vro + \vro \, \du)\ dx \\
				&-\i \f{1}{1+ \vro} \p_3^m (b \, \du + u \cdot \nabla b)\cdot \p_3^m  b \ dx+ \i \f{1}{1+ \vro} \p_3^m (b \cdot \nabla u ) \cdot \p_3^m  b \ dx
				\\
				&+\i \f{\vro}{1+\vro} (e_2 \p_3^m  \du-\Delta_h \p_3^m  b - \var \p_3^{m+2} b - \p_2 \p_3^m  u )\cdot \p_3^m  b\ dx 
				- \var \i \p_3^m ( \f{\vro}{1+\vro}(\p_2^2 u+\p_3^2 u)) \cdot \p_3^m  u\ dx\\
				:= & \sum_{i=1}^{10} II_i.
			\end{align*}
			{\bf{Estimate of the term $II_1$.}}
			Due to the lower dissipative structure of density in the horizontal directions, the estimate of term $II_1$ is somewhat complicated. Thus, we decompose this term into five components as follows.
			Integrating by parts, we can write
			\begin{align*}
				II_1 = &\; -\f12 \i |\p_3^m \vro|^2 \du \ dx - \sum_{1 \le l\le m} \i C_{m}^{l} \p_3^l u  \cdot \nabla \p_3^{m-l} \vro \cdot \p_3^m \vro \ dx\\
				&\; -   \sum_{0 \le l\le m-1}\i C_{m}^{l} \p_3^l \vro \cdot \p_3^{m-l} \du \cdot \p_3^m \vro  \ dx\\
				=&\;  -\f12 \i |\p_3^m \vro|^2 \du \ dx -m \i | \p_3^{m} \vro |^2 \p_3 u_3 \ dx  -m \i \p_3 u_h  \cdot \nabla_h \p_3^{m-1} \vro \cdot \p_3^m \vro  \ dx  \\
				&\; - \sum_{2 \le l\le m} \i C_{m}^{l} \p_3^l u  \cdot \nabla \p_3^{m-l} \vro \cdot \p_3^m \vro \ dx -   \sum_{0 \le l\le m-1} \i C_{m}^{l} \p_3^l \vro \cdot \p_3^{m-l} \du \cdot \p_3^m \vro  \ dx\\
				=&\; - (\f12+m) \i |\p_3^m \vro|^2 \du \ dx + m \i  | \p_3^{m} \vro |^2 \nabla_h \cdot u_h\ dx  -m \i \p_3 u_h  \cdot \nabla_h \p_3^{m-1} \vro \cdot \p_3^m \vro  \ dx  \\
				&\; - \sum_{2 \le l\le m} \i C_{m}^{l} \p_3^l u  \cdot \nabla \p_3^{m-l} \vro \cdot \p_3^m \vro \ dx -   \sum_{0 \le l\le m-1} \i C_{m}^{l} \p_3^l \vro \cdot \p_3^{m-l} \du \cdot \p_3^m \vro  \ dx\\
				:= &\; \sum_{i=1}^{5} II_{1,i}.
			\end{align*}
			Estimate of the term $II_{1,1}$. Using the density equation $\eqref{eqr}_1$ and integrating by parts, we can obtain
			\begin{align*}
				&\;  - \i  | \p_3^m \vro|^2 \du \ dx 
				=\i (\p_t \vro +  u \cdot \nabla \vro + \vro \, \du )| \p_3^m \vro|^2 \ dx\\
				=&\;  \f{d}{dt}  \i \vro | \p_3^m \vro|^2 \ dx - 2 \i \vro \, \p_3^m \vro \, \p_3^m \p_t \vro \ dx +
				\i(  u \cdot \nabla \vro + \vro \, \du )| \p_3^m \vro|^2 \ dx\\
				= &\;    \f{d}{dt} \i \vro  | \p_3^m \vro|^2 \ dx + 2 \i \vro\,  \p_3^m \vro \, \p_3^m ( u \cdot \nabla \vro) \ dx  \\
				&\; + 2 \i \vro\,  \p_3^m \vro \, \p_3^m (\du + \vro \, \du) \ dx +  \i (  u \cdot \nabla \vro + \vro \, \du )| \p_3^m \vro|^2 \ dx \\
				=  &\;  \f{d}{dt} \i \vro  | \p_3^m \vro|^2 \ dx  + 2 \sum_{ 1\le l\le m} C_{m}^l \i \vro \cdot \p_3^m \vro (\p_3^l u_h \cdot \nabla_h \p_3^{m-l} \vro + \p_3^l u_3 \, \p_3^{m-l+1} \vro ) \ dx \\
				&\; + 2 \i \vro\,  \p_3^m \vro \, \p_3^m (\du + \vro \, \du) \ dx.
			\end{align*}
			We now estimate the second term on the right-hand side.
			When $l=1$, we can check that
			\beq\label{ieq-sec2-1} \bal
			&\;  \i \vro \cdot \p_3^m \vro (\p_3 u_h \cdot \nabla_h \p_3^{m-1} \vro + \p_3 u_3 \, \p_3^m  \vro ) \ dx \\
			\lesssim &\; \| \vro\|_{L^{\infty}} \| \p_3^m \vro\|_{L^2} ( \|\p_3^{m-1} \nabla_h \vro\|_{L^2} \| \p_3 u_h\|_{L^{\infty}}+   \| \p_3^m \vro\|_{L^2} \| \p_3 u_3\|_{L^{\infty}} )\\
			\lesssim &\;  \me^m(t) \md^{m}(t) +  \|(\vro, \p_3 \vro, \p_3^{m} \vro)\|_{L^2}^{\f94} \|(\nabla_h \vro,\nabla_h \p_3 \vro, \p_3 u_3,\p_3^2 u_3)\|_{H_{tan}^1}^{\f74},
			\dal \deq
			where we have used the estimate
			\beqq \bal
			\|(\vro, \p_3 u_h)\|_{L^{\infty}} \lesssim &\; \|(\vro, \p_3 \vro, \p_3 u_h, \p_3^2 u_h)\|_{L^2}^{\f14} \|\nabla_h (\vro, \p_3 u_h)\|_{H_{tan}^1}^{\f38} \| \nabla_h \p_3 (\vro, \p_3 u_h)\|_{H_{tan}^{1}}^{\f38}, \\
			\|\p_3 u_3\|_{L^{\infty}} \lesssim &\; \|\p_3 u_3\|_{H_{tan}^{2}}^{\f12} \|\p_3^2 u_3\|_{H_{tan}^{2}}^{\f12},
			\dal \deqq
			by the inequality \eqref{ie:Sobolev}. 
			Using the anisotropic type inequality \eqref{ie:Sobolev}, we have
			\beqq \bal
			&\; \sum_{ 2\le l\le m} C_{m}^l \i \vro \cdot \p_3^m \vro (\p_3^l u_h \cdot \nabla_h \p_3^{m-l} \vro + \p_3^l u_3 \, \p_3^{m-l+1} \vro ) \ dx \\
			\lesssim &\; \| \vro\|_{L^{\infty}} \| \p_3^m \vro\|_{L^2} (  \|\p_3^2 u_h \|_{L^2}^{\frac14}
			\|\p_1 \p_3^2 u_h \|_{L^2}^{\frac14}
			\|\p_2 \p_3^2 u_h \|_{L^2}^{\frac14}
			\|\p_{12} \p_3^2 u_h \|_{L^2}^{\frac14}
			\|\nabla_h  \p_3^{m-2} \vro\|_{L^2}^{\f12}  \|\nabla_h  \p_3^{m-1} \vro\|_{L^2}^{\f12} \\
			&\; + \|\p_3^2 u_3 \|_{L^2}^{\frac14}
			\|\p_1 \p_3^2 u_3 \|_{L^2}^{\frac14}
			\|\p_3^3 u_3 \|_{L^2}^{\frac14}
			\|\p_{1} \p_3^3 u_3 \|_{L^2}^{\frac14}
			\| \p_3^{m-1} \vro\|_{L^2}^{\f12}  \|\p_2  \p_3^{m-1} \vro\|_{L^2}^{\f12}\\
			&\; + \sum_{3\le l \le m} \|\p_3^l u_h \|_{L^2}^{\frac12}
			\|\p_1 \p_3^l u_h \|_{L^2}^{\frac12}
			\|\nabla_h  \p_3^{m-l} \vro\|_{L^2}^{\frac14}
			\|\nabla_h  \p_3^{m-l+1} \vro\|_{L^2}^{\frac14}
			\|\nabla_h \p_2  \p_3^{m-l} \vro\|_{L^2}^{\f14}  \|\nabla_h \p_2  \p_3^{m-l+1} \vro\|_{L^2}^{\f14} \\
			&\;  + \sum_{3\le l \le m} \|\p_3^l u_3 \|_{L^2}^{\frac12}
			\| \p_3^{l+1} u_3 \|_{L^2}^{\frac12}
			\| \p_3^{m-l+1} \vro\|_{L^2}^{\frac14}
			\|\p_1  \p_3^{m-l+1} \vro\|_{L^2}^{\frac14}
			\|\p_2  \p_3^{m-l+1} \vro\|_{L^2}^{\f14}  \|\p_{12} \p_3^{m-l+1} \vro\|_{L^2}^{\f14} )\\
			\lesssim &\; \me^m(t) \md^{m}(t).
			\dal \deqq
			Similarly, we can obtain that
			\beqq\bal
			&\;  \i \vro\,  \p_3^m \vro \, \p_3^m (\du + \vro \, \du) \ dx \\
			\lesssim &\;\|\vro\|_{L^{\infty}} \|\p_3^m \vro\|_{L^2} \|\p_3^m \du\|_{L^2} + \| \p_3^m \vro\|_{L^2}^2 \|\du \|_{L^{\infty}} \|\vro\|_{L^{\infty} } \\ 
			\lesssim &\; \|(\vro, \p_3\vro, \p_3^m \vro)\|_{L^2}^{\f54} \|\p_3^m \du\|_{L^2} \|(\nabla_h \vro, \nabla_h \p_3 \vro)\|_{H_{tan}^1}^{\f34}+  \|(\vro, \p_3 \vro, \p_3^{m} \vro)\|_{L^2}^{\f94} \|(\nabla_h \vro, \nabla_h \p_3 \vro, \du,\p_3 \du)\|_{H_{tan}^1}^{\f74}.
			\dal \deqq
			Thus, we have
			\beqq
			II_{1,1} \le (\f12 +m) \f{d}{dt} \i \vro  | \p_3^m \vro|^2 \ dx +C \me^m(t) \md^{m}(t) + C \me^m(t)^{\f58}   \md^{m} (t)^{\f12}  \md_{tan}^{m-1} (t)^{\f38}+ C \me^m(t)^{\f98}  \md_{tan}^{m-1} (t)^{\f78}.
			\deqq
			Estimate of the term $II_{1,2}$. Using the velocity equation $\eqref{eqr}_2$, we have
			\beq\label{ieq-nablavro} \bal
			\nabla \vro 
			=  - \p_t u + F_2 +  \p_1^2 u + \var (\p_2^2 u + \p_3^2 u) + \nabla \du - \nabla b_2 + \p_2 b.
			\dal \deq
			Substituting above equation of $\p_3 \vro$ into the term $II_{1,2}$, we have
			\begin{align*}
				II_{1,2} = &\; - m \i \nabla_h \cdot u_h \p_3^m \vro \cdot \p_3^{m-1} \p_t u_3\ dx + m\i \nabla_h \cdot u_h \p_3^m \vro \cdot \p_3^{m-1} (F_2)_3 \ dx \\
				&\;+ m\i \nabla_h \cdot u_h \p_3^m \vro \cdot \p_3^{m-1} (\p_1^2 u_3 + \var (\p_2^2 u_3 + \p_3^2 u_3) + \p_3 \du - \p_3 b_2 + \p_2 b_3)\ dx \\
				= &\; - m \f{d}{dt} \i \nabla_h \cdot u_h \p_3^m \vro \cdot \p_3^{m-1} u_3 \ dx + m\i \nabla_h \cdot \p_t u_h \p_3^m \vro \cdot \p_3^{m-1} u_3 \ dx \\
				&\; + m \i \nabla_h \cdot u_h \p_3^m \vro_t \cdot \p_3^{m-1}u_3  \ dx  + m\i \nabla_h \cdot u_h \p_3^m \vro \cdot \p_3^{m-1} (F_2)_3 \ dx \\
				&\;+ m\i \nabla_h \cdot u_h \p_3^m \vro \cdot \p_3^{m-1} (\p_1^2 u_3 + \var (\p_2^2 u_3 + \p_3^2 u_3) + \p_3 \du - \p_3 b_2 + \p_2 b_3)\ dx \\
				:=&\; -m\f{d}{dt} \i \nabla_h \cdot u_h \p_3^m \vro \cdot \p_3^{m-1}  u_3\ dx + m \sum_{i=1}^{4} II_{1,2,i}.
			\end{align*}
			Estimate of the term $II_{1,2,1}$. Under the assumption \eqref{assumption}, using the equation $\eqref{ieq-nablavro}$ and the anisotropic type inequality \eqref{ie:Sobolev}, we can obtain that for $m \ge 4$,
			\beqq \bal
			II_{1,2,1} = &\;\i \nabla_h \cdot \left(  (F_2)_h +  \p_1^2 u_h + \var (\p_2^2 u_h + \p_3^2 u_h) + \nabla_h \du - \nabla_h b_2 + \p_2 b_h -\nabla_h \vro \right) \p_3^m \vro \cdot \p_3^{m-1} u_3\ dx\\
			\lesssim &\; \| \p_3^m \vro \|_{L^2} \|\p_3^{m-1} u_3\|_{L^2}^{\f12}  \|\p_3^{m} u_3 \|_{L^2}^{\f12} \big(\|(F_2)_h \|_{ H_{tan}^3} + \|(\p_1^2 u, \nabla_h \du, \nabla_h b, \nabla_h \vro )\|_{ H_{tan}^3} + \var\|(\p_2^2 u, \p_3^2 u )\|_{ H_{tan}^3}  \big)\\
			\lesssim &\;  \sqrt{\me^m (t)}\md^m (t).
			\dal \deqq
			Using the density equation $\eqref{eqr}_1$, integrating by parts and under the assumption \eqref{assumption}, it is easy to check that
			\beqq \bal
			II_{1,2,2} =&\; -\i \nabla_h \cdot u_h \p_3^m (\du+ u \cdot \nabla\vro + \vro \, \du) \cdot \p_3^{m-1} u_3\ dx  \\
			= &\; -\i \nabla_h \cdot u_h \p_3^m (\du + \vro \, \du) \cdot \p_3^{m-1} u_3  \ dx   +  \i \p_3^{m-1} ( u \cdot \nabla\vro ) \, \p_3 (\nabla_h \cdot u_h \, \p_3^{m-1} u_3 ) \ dx \\
			\lesssim &\; \sqrt{\me^m (t)}\md^m (t).
			\dal \deqq
			Similarly, we can estimate the terms $II_{1,2,3}$ to $II_{1,2,4}$ .
			\beqq \bal
			II_{1,2,3} = &\; m\i \nabla_h \cdot u_h \, \p_3^m \vro \cdot \p_3^{m-1} \left(- u \cdot \nabla u_3+ \f{1}{1+\vro} b \cdot \nabla b_3   \right) \ dx\\
			&\; - m\i \nabla_h \cdot u_h \, \p_3^m \vro \cdot \p_3^{m-1} \left( \f{\vro}{1+\vro}( \p_1^2 u_3 + \var \p_2^2 u_3 + \var \p_3^2 u_3 + \p_3  \du +\p_2 b_3)\right) \ dx
			\\
			&\; - m \i \nabla_h \cdot u_h \, \p_3^m \vro \cdot \left(\f12  \sum_{1\le l\le m-1} C_{m}^l \p_3^{l}\vro \, \p_{3}^{m-l} \vro + \vro \, \p_3^{m} \vro \right)\ dx \\
			&\; - m \i \nabla_h \cdot u_h \, \p_3^m \vro \cdot \left( \f12 \sum_{1\le l\le m-1} C_{m-1}^l \p_3^{l}(\f{1}{1+\vro}) \p_{3}^{m-l} (|b|^2)  +\f{1}{2(1+\vro)}\p_3^{m} (|b|^2) \right)\ dx \\
			&\; + m \i \nabla_h \cdot u_h \, \p_3^m \vro \cdot \left(  \sum_{1\le l\le m-1} C_{m-1}^l \p_3^{l}(\f{\vro}{1+\vro}) \p_{3}^{m-l} b_2  +\f{\vro}{1+\vro} \p_3^{m} b_2 \right)\ dx \\
			\lesssim &\;  \sqrt{\me^m (t)}\md^m (t) +  \me^m (t)^{\f34}  \md^m (t)^{\f14} \md_{tan}^{m-1}(t)^{\f34} + \me^m (t)^{\f98}  \md^m (t)^{\f12} \md_{tan}^{m-1}(t)^{\f38},
			\dal \deqq
			and
			\beqq \bal
			II_{1,2,4} \lesssim &\;\|\p_3^m \vro\|_{L^2}  \| \nabla_h \cdot u_h\|_{L^{\infty}} (\|( \p_3^{m-1} \p_1^2 u_3 ,\p_3^{m} \du, \p_3^{m-1}\p_2 b_3 )\|_{L^2}+ \var \|( \p_3^{m-1} \p_2^2 u_3 ,\p_3^{m+1} u_3)\|_{L^2})\\
			&\; +\|\p_3^m \vro\|_{L^2}  \|\p_3^m b\|_{L^2}^{\f12} \|\nabla_h \p_3^m b\|_{L^2}^{\f12}  \| \nabla_h \cdot u_h\|_{H_{tan}^1}^{\f12} \| \nabla_h \cdot \p_3 u_h\|_{H_{tan}^1}^{\f12}\\
			\lesssim &\;  \sqrt{\me^m (t)}\md^m (t) +  \me^m (t)^{\f34}  \md^m (t)^{\f14} \md_{tan}^{m-1}(t)^{\f12}.
			\dal \deqq
			Combining the estimates of $II_{1,2,1}$ to $II_{1,2,4}$, we have
			\beqq \bal
			II_{1,2} \le &\;- m \f{d}{dt} \i \nabla_h \cdot u_h \p_3^m \vro \cdot \p_3^{m-1} u_3\ dx + C \sqrt{\me^m (t)}\md^m (t)  +C  \me^m (t)^{\f34}  \md^m (t)^{\f14} \md_{tan}^{m-1}(t)^{\f12} \\
			&\; + C \me^m (t)^{\f98}  \md^m (t)^{\f12} \md_{tan}^{m-1}(t)^{\f38}.
			\dal \deqq
			Next, we deal with the term $II_{1,4}$. Integrating by parts, we have
			\begin{align*}
				II_{1,4} = &\;  - C_{m}^2 \i \p_3^2 u_h \cdot  \nabla_h \p_3^{m-2} \vro \cdot \p_3^{m} \vro \ dx -  \f12  C_{m}^2  \i \p_3(|\p_3^{m-1} \vro|^2)\, \p_3^2 u_3 \ dx \\
				&\;- \sum_{3 \le l\le m} \i C_{m}^{l} \p_3^l u_3  \,  \p_3^{m-l+1} \vro \cdot \p_3^m \vro \ dx - \sum_{3 \le l\le m-1} \i C_{m}^{l} \p_3^l u_h  \cdot \nabla_h \p_3^{m-l} \vro \cdot \p_3^{m} \vro \ dx  \\
				&\; -  \i \p_3^m u_h  \cdot \nabla_h \vro \cdot \p_3^{m} \vro \ dx\\
				= &\;  C_{m}^2 \i \p_{3} (\p_3^2 u_h\cdot \nabla_h \p_3^{m-2} \vro) \, \p_3^{m-1} \vro \ dx+  \f12  C_{m}^2  \i |\p_3^{m-1} \vro|^2 \, \p_3^3 u_3 \ dx \\
				&\;+ \sum_{3 \le l\le m} \i C_{m}^{l} \p_3^{m-1} \vro \, (\p_3^{l+1} u_3  \, \p_3^{m-l+1} \vro + \p_3^l u_3  \, \p_3^{m-l+2} \vro) \ dx \\
				&\; + \sum_{3 \le l\le m-1} \i C_{m}^{l} \p_3^{m-1} \vro\, (\p_3^{l+1} u_h  \cdot \nabla_h \p_3^{m-l} \vro + \p_3^l u_h  \cdot \nabla_{h} \p_3^{m-l+1} \vro) \ dx\\
				&\; -  \i \p_3^m u_h  \cdot \nabla_h \vro \cdot \p_3^{m} \vro \ dx.
			\end{align*}
			Using the  anisotropic type inequality \eqref{ie:Sobolev}, we can obtain that
			\beqq \bal
			II_{1,4} 
			\lesssim &\;  \| \p_3^{m-1} \vro\|_{L^2}^{\f12}  \| \p_2 \p_3^{m-1} \vro\|_{L^2}^{\f12}\Big(\|\p_3^2 u_h \|_{L^2}^{\frac14}
			\|\p_1 \p_3^2 u_h \|_{L^2}^{\frac14}
			\|\p_3^3 u_h \|_{L^2}^{\frac14}
			\|\p_{1} \p_3^3 u_h \|_{L^2}^{\frac14}
			\|\nabla_h  \p_3^{m-1} \vro\|_{L^2}\\
			&\; + \|\p_3^3 u_h \|_{L^2}^{\frac12}
			\|\p_{1} \p_3^3 u_h \|_{L^2}^{\frac12}\|\nabla_h  \p_3^{m-2} \vro\|_{L^2}^{\f12} \|\nabla_h  \p_3^{m-1} \vro\|_{L^2}^{\f12} \\
			&\;+ \|\p_3^3 u_3 \|_{L^2}^{\frac12}
			\|\p_3^4 u_3 \|_{L^2}^{\frac12}\| \p_3^{m-1} \vro\|_{L^2}^{\f12} \|\p_1  \p_3^{m-1} \vro\|_{L^2}^{\f12} \\
			&\; + \sum_{3 \le l\le m} (
			\|\p_3^{l+1} u_3 \|_{L^2}
			\|\p_3^{m-l+1} \vro\|_{L^2}^{\frac14}
			\| \p_3^{m-l+2} \vro \|_{L^2}^{\frac14}
			\|\nabla_h \p_3^{m-l+1} \vro  \|_{L^2}^{\frac14}
			\|\nabla_h \p_3^{m-l+2} \vro \|_{L^2}^{\frac14} \\
			&\; + \|\p_3^{l} u_3 \|_{L^2}^{\frac12}
			\|\p_1 \p_3^{l} u_3 \|_{L^2}^{\frac12}
			\|\p_3^{m-l +2} \vro \|_{L^2}^{\frac12}
			\|\nabla_h \p_3^{m-l+2} \vro  \|_{L^2}^{\frac12}) \\
			&\; + \sum_{3 \le l\le m-1} (
			\|\p_3^{l+1} u_h \|_{L^2}^{\frac12}
			\|\p_1 \p_3^{l+1} u_h \|_{L^2}^{\frac12}
			\|\nabla_h \p_3^{m-l} \vro \|_{L^2}^{\frac12}
			\|\nabla_h \p_3^{m-l+1} \vro  \|_{L^2}^{\frac12} \\
			&\; + \|\p_3^{l} u_h \|_{L^2}^{\frac12}
			\|\p_1 \p_3^{l} u_h \|_{L^2}^{\frac12}
			\|\nabla_h \p_3^{m-l +1} \vro \|_{L^2}^{\frac12}
			\|\nabla_h \p_3^{m-l+2} \vro  \|_{L^2}^{\frac12})
			\Big) \\
			&\; + \| \p_3^{m} \vro\|_{L^2} \|\p_3^{m} u\|_{L^2}^{\frac12}
			\| \p_1 \p_3^{m} u \|_{L^2}^{\frac12}   \|\nabla_h \vro\|_{L^2}^{\frac14}
			\| \p_3 \nabla_h  \vro \|_{L^2}^{\frac14}
			\|\nabla_h^2 \vro  \|_{L^2}^{\frac14}
			\|\p_3 \nabla_h^2 \vro \|_{L^2}^{\frac14} \\
			\lesssim &\;  \sqrt{\me^m(t)} \md^{m}(t) + \me^m(t)^{\f34}  \md^m(t)^{\f14}  \md_{tan}^{m-1} (t)^{\f12}.
			\dal \deqq
			Similarly, we have
			\beqq \bal
			II_{1,3} \lesssim &\; \| \p_3^m \vro\|_{L^2}  \| \nabla_h \p_3^{m-1} \vro\|_{L^2}  \| \p_3 u_h\|_{L^{\infty}} \\ \lesssim &\; \| \p_3^m \vro\|_{L^2}  \| \nabla_h \p_3^{m-1} \vro\|_{L^2}  \| (\p_3 u_h, \p_3^2 u_h)\|_{L^{2}}^{\f14}   \| \nabla_h \p_3 u_h\|_{L^{2}}^{\f38}    \|\nabla_h \p_3^2 u_h\|_{L^{2}}^{\f38}  \\
			\lesssim &\;  \me^m(t)^{\f58}  \md^m(t)^{\f{11}{16}}  \md_{tan}^{m-1} (t)^{\f{3}{16}},\\
			II_{1,5} 
			\lesssim &\; \sqrt{\me^m(t)} \md^{m}(t) + \|\p_3^{m} \vro\|_{L^2} \| \p_3^{m}\du\|_{L^2}  \| \vro\|_{L^{\infty}} \\
			\lesssim &\;  \sqrt{\me^m(t)} \md^{m}(t) +  \me^m(t)^{\f58}  \md^m(t)^{\f12}  \md_{tan}^{m-1} (t)^{\f38}. 
			\dal \deqq
			The combination of estimates in terms  $II_{1,i}(i=1,\cdots,5)$  yields that
			\beq\label{3202} \bal
			II_1 \le &\; (\f12+m)  \f{d}{dt} \i \vro  | \p_3^m \vro|^2 \ dx   - m\f{d}{dt} \i \nabla_h \cdot u_h \p_3^m \vro \cdot \p_3^{m-1}   u_3\ dx \\
			&\; + C\sqrt{\me^m (t)}\md^m (t) + C\me^m (t)^{\f34}  \md^m (t)^{\f14} \md_{tan}^{m-1}(t)^{\f12} \\
			&\; + C \me^m(t)^{\f58}  \md^m(t)^{\f{11}{16}}  \md_{tan}^{m-1} (t)^{\f{3}{16}}
			+C\me^m(t)^{\f58}  \md^m(t)^{\f12}  \md_{tan}^{m-1} (t)^{\f38}. 
			\dal \deq
			{\bf{Estimate of the term $II_2$.}} 
			Integrating by part, we have
			\begin{align*}
				II_2 = &\;  \f12 \i |\p_3^m u|^2 \du \ dx - \sum_{ 1\le l \le m } C_{m}^{l} \i \p_3^l u \cdot \nabla \p_3^{m-l} u \cdot \p_3^m u \ dx+  \f12 \i \p_3^m (|\vro|^2) \cdot \p_3^m \du\ dx \\
				= &\;  \f12 \i |\p_3^m u|^2 \du \ dx +\i \vro \, \p_3^m \vro \cdot \p_3^m \du\ dx  - \sum_{ 1\le l \le m } C_{m}^{l} \i \p_3^l u \cdot \nabla \p_3^{m-l} u \cdot \p_3^m u \ dx \\
				&\; +\f12 \sum_{ 1\le l \le m-1 } C_{m}^{l} \i  \p_3^l \vro \cdot  \p_3^{m-l} \vro \cdot \p_3^m \du \ dx
				:= \sum_{i=1}^{4} II_{2,i}.
			\end{align*}
			Using the  anisotropic type inequality \eqref{ie:Sobolev}, we can obtain that
			\beqq \bal
			&\; II_{2,1} + II_{2,2} \\
			\lesssim &\; \| \du\|_{L^2}^{\f14}  \| \p_2 \du\|_{L^2}^{\f14}  \| \p_3 \du\|_{L^2}^{\f14}  \|\p_{23} \du\|_{L^2}^{\f14}  \|\p_3^m u\|_{L^2}^{\f32} \|\p_1 \p_3^m u\|_{L^2}^{\f12} +  \|\p_3^m \du\|_{L^2} \|\p_3^m \vro\|_{L^2} \| \vro\|_{L^{\infty}} \\
			\lesssim &\; \me^m(t)^{\f34}  \md^m(t)^{\f14}  \md_{tan}^{m-1} (t)^{\f12} + \me^m(t)^{\f58}  \md^m(t)^{\f12}  \md_{tan}^{m-1} (t)^{\f38}.
			\dal \deqq
			Estimate of the term $II_{2,3}$. It is easy to check that
			\beqq \bal
			II_{2,3}= &\; -   \i (\p_3^m u_h  \cdot \nabla_h u +  \p_3^m u_3  \cdot \p_3 u) \cdot \p_3^m u \ dx -\i   \p_3 u_3  \cdot |\p_3^m u|^2 \ dx \\
			&\; - \sum_{ 1\le l \le m-1 } C_{m}^{l} \i \p_3^l u_h \cdot \nabla_h \p_3^{m-l} u \cdot \p_3^m u \ dx - \sum_{ 2\le l \le m-1 } C_{m}^{l} \i \p_3^l u_3 \cdot \p_3^{m-l+1} u \cdot \p_3^m u \ dx \\
			\lesssim  &\; \|\p_3^m u\|_{L^2}^{\f32} \|\p_1 \p_3^m u\|_{L^2}^{\f12} (\|\nabla_h u\|_{H_{tan}^1}^{\f12}  \| \p_3 \nabla_h u\|_{H_{tan}^1}^{\f12}  + \|\p_3 u_3\|_{H_{tan}^1}^{\f12}   \| \p_3^2 u_3\|_{H_{tan}^1}^{\f12} ) \\ &\; +
			\|\p_3^m u_3\|_{L^2} \|\p_3^m u\|_{L^2} \|\p_1 \p_3^m u\|_{L^2}^{\f12} \|\p_3 u\|_{H_{tan}^1}^{\f12}  \| \p_3^2 u\|_{H_{tan}^1}^{\f12}  + \|\p_3^m u\|_{L^2}^{\f12} \|\p_1 \p_3^m u\|_{L^2}^{\f12} \\
			&\;\times ( \sum_{1 \le l \le [\f{m-1}{2}]} \| \p_3^{l} u_h \|_{L^2}^{\f14}  \| \p_2 \p_3^{l} u_h \|_{L^2}^{\f14}  \| \p_3^{l+1} u_h  \|_{L^2}^{\f14}  \|\p_{2} \p_3^{l+1} u_h \|_{L^2}^{\f14} \|\nabla_h \p_3^{m-l} u\|_{L^2} \\
			&\; +\sum_{[\f{m-1}{2}] < l \le m-1} \| \p_3^{l} u_h \|_{L^2}^{\f12}  \| \p_2 \p_3^{l} u_h \|_{L^2}^{\f12}  \|\nabla_h \p_3^{m-l} u\|_{L^2} ^{\f12} \|\p_2 \nabla_h \p_3^{m-l} u\|_{L^2} ^{\f12} \\
			&\; +\sum_{2 \le l \le m-1} \| \p_3^{l} u_3 \|_{L^2}^{\f12}  \| \p_3^{l+1} u_3 \|_{L^2}^{\f12}  \| \p_3^{m-l+1} u\|_{L^2} ^{\f12} \|\p_2 \p_3^{m-l+1} u\|_{L^2} ^{\f12}) \\
			\lesssim &\; \sqrt{\me^m(t)} \md^{m}(t) + \me^m(t)^{\f34}  \md^m(t)^{\f14}  \md_{tan}^{m-1} (t)^{\f12}.
			\dal \deqq
			Similarly, we have
			\beqq \bal 
			II_{2,4} \lesssim &\; \sqrt{\me^m(t)} \md^{m}(t).
			\dal \deqq
			Thus we have
			\beq \label{3203} \bal
			II_{2} \lesssim  \sqrt{\me^m(t)} \md^{m}(t) + \me^m(t)^{\f34}  \md^m(t)^{\f14}  \md_{tan}^{m-1} (t)^{\f12} + \me^m(t)^{\f58}  \md^m(t)^{\f12}  \md_{tan}^{m-1} (t)^{\f38}.
			\dal \deq
			Similarly, we can obtain that
			\beq \label{3204} \bal
			II_3 \lesssim &\;  \sqrt{\me^m(t)} \md^{m}(t) + \me^m(t)^{\f34}  \md^m(t)^{\f14}  \md_{tan}^{m-1} (t)^{\f12},\\
			II_5+  II_6+ II_7 + II_9+II_{10}  \lesssim &\; \sqrt{\me^m(t)}  \md^{m}(t).
			\dal \deq
			Finally, we estimate the terms $II_4$ and $II_8$. Integrating by parts, we have
			\begin{align*}
				II_4 + II_8 
				= &\; \i \f{1}{1+\vro} (  b \cdot \nabla \p_3^m u \cdot \p_3^m  b + b \cdot \nabla \p_3^m b \cdot \p_3^m  u)\ dx  \\
				&\; + \sum_{1 \le l \le m} C_{m}^l \i\p_3^l (\f{1}{1+\vro} )  \cdot   \p_3^{m-l} (b \cdot \nabla b) \cdot \p_3^m  u \ dx \\
				&\; +\sum_{1 \le l \le m} C_{m}^l \i \f{1}{1+\vro}  (\p_3^l b  \cdot  \nabla \p_3^{m-l} u \cdot \p_3^m  b +  \p_3^l b  \cdot  \nabla \p_3^{m-l} b \cdot \p_3^m  u )\ dx \\
				= &\; - \i \p_3^m u \cdot \p_3^m   b \cdot  b \cdot \nabla( \f{1}{1+\vro}) \ dx + \sum_{1 \le l \le m} C_{m}^l \i\p_3^l (\f{1}{1+\vro} )  \cdot   \p_3^{m-l} (b \cdot \nabla b) \cdot \p_3^m  u \ dx\\
				&\;+\sum_{1 \le l \le m} C_{m}^l \i \f{1}{1+\vro}  (\p_3^l b  \cdot  \nabla \p_3^{m-l} u \cdot \p_3^m  b +  \p_3^l b  \cdot  \nabla \p_3^{m-l} b \cdot \p_3^m  u )\ dx. 
			\end{align*}
			Using the  anisotropic type inequality \eqref{ie:Sobolev}, we can obtain that
			\beqq \bal
			&\; \i \p_3^m u \cdot \p_3^m   b \cdot  b \cdot \nabla( \f{1}{1+\vro}) \ dx   \\
			\lesssim &\; \| \p_3^m b\|_{L^2}^{\f12} \| \p_2 \p_3^m b\|_{L^2}^{\f12} \| \p_3^m u\|_{L^2}^{\f12}   \| \p_1 \p_3^m u\|_{L^2}^{\f12} (\| b_h\|_{L^{\infty}} \| \nabla_h ( \f{1}{1+\vro}) \|_{L^2}^{\f12} \| \p_3 \nabla_h ( \f{1}{1+\vro}) \|_{L^2}^{\f12} \\
			&\;+ \| \p_3( \f{1}{1+\vro}) \|_{L^{\infty}} \| b_3 \|_{L^2}^{\f12} \| \p_3 b_3 \|_{L^2}^{\f12}) \\
			\lesssim &\; \me^m(t) \md^{m}(t),
			\dal \deqq
			and
			\beqq \bal
			&\; \sum_{1 \le l \le m} C_{m}^l \i \f{1}{1+\vro}  (\p_3^l b  \cdot  \nabla \p_3^{m-l} u \cdot \p_3^m  b +  \p_3^l b  \cdot  \nabla \p_3^{m-l} b \cdot \p_3^m  u )\ dx\\
			\lesssim &\; \|\p_3^m b\|_{L^2}^{\f12} \|\p_2 \p_3^m b\|_{L^2}^{\f12} ( \sum_{1 \le l \le [\f{m}{2}]} \| \p_3^{l} b \|_{L^2}^{\f14}  \| \p_2 \p_3^{l} b \|_{L^2}^{\f14}  \| \p_3^{l+1} b \|_{L^2}^{\f14}  \|\p_{2} \p_3^{l+1} b\|_{L^2}^{\f14} \|\nabla_h \p_3^{m-l} u\|_{L^2} \\
			&\; +\sum_{[\f{m}{2}] < l \le m-1} \| \p_3^{l} b \|_{L^2}^{\f12}  \| \p_1 \p_3^{l}b \|_{L^2}^{\f12}  \|\nabla_h \p_3^{m-l} u\|_{L^2} ^{\f12} \| \nabla_h \p_3^{m-l+1} u\|_{L^2} ^{\f12} \\
			&\; +\sum_{1 \le l \le m-1} \| \p_3^{l} b_3 \|_{L^2}^{\f12}  \| \p_3^{l+1} b_3 \|_{L^2}^{\f12}  \| \p_3^{m-l+1} u\|_{L^2} ^{\f12} \|\p_1 \p_3^{m-l+1} u\|_{L^2} ^{\f12}) \\
			&\; + \|\p_3^m u\|_{L^2}^{\f12} \|\p_1 \p_3^m u\|_{L^2}^{\f12} \sum_{1 \le l \le m-1} ( \| \p_3^{l} b_h \|_{L^2}^{\f12}  \| \p_3^{l} \nabla_h b_h \|_{L^2}^{\f12}  \| \nabla_h \p_3^{m-l} b\|_{L^2} ^{\f12} \|\nabla_h \p_3^{m-l+1} b\|_{L^2} ^{\f12} \\
			&\; +\| \p_3^{l} b_3 \|_{L^2}^{\f12}  \| \p_3^{l+1} b_3 \|_{L^2}^{\f12}  \| \p_3^{m-l+1} b\|_{L^2} ^{\f12} \|\p_2 \p_3^{m-l+1} b\|_{L^2} ^{\f12})\\
			\lesssim &\; \sqrt{\me^m(t)} \md^{m}(t).
			\dal \deqq
			Similarly,  under the assumption \eqref{assumption}, we have
			\beqq \bal
			\sum_{1 \le l \le m} C_{m}^l \i\p_3^l (\f{1}{1+\vro} )  \cdot   \p_3^{m-l} (b \cdot \nabla b) \cdot \p_3^m  u \ dx 
			\lesssim  \sqrt{\me^m(t)} \md^{m}(t).
			\dal \deqq 
			Thus we have the following estimates of $II_4$ and $II_8$:
			\beqq
			II_4 + II_8  \lesssim  \sqrt{\me^m(t)} \md^{m}(t),
			\deqq
			which, together with the estimates \eqref{3202}-\eqref{3204}, yields that
			\beqq \bal
			&\; \frac{d}{dt}\frac{1}{2}\i(|\p_3^m u|^2+|\p_3^m \vro|^2+\f{1}{1+ \vro} |\p_3^m  b|^2 -(\f12+m) \vro  | \p_3^m \vro|^2 + m\nabla_h \cdot u_h \, \p_3^m \vro \cdot \p_3^{m-1} u_3   )dx \\
			&\;+ 
			\|\p_3^{m} (\p_1 u, \du, \nabla_h b)\|_{L^2}^2 + \var \|\p_3^{m} (\p_2 u, \p_3 u, \p_3 b)\|_{L^2}^2\\
			\lesssim 
			&\;  \sqrt{\me^m (t)}\md^m (t) + \me^m (t)^{\f34}  \md^m (t)^{\f14} \md_{tan}^{m-1}(t)^{\f12} + \me^m(t)^{\f58}  \md^m(t)^{\f{11}{16}}  \md_{tan}^{m-1} (t)^{\f{3}{16}}.
			\dal \deqq
			Integrating over $[0,t]$, under the assumption \eqref{assumption}, we can obtain
			\beqq \bal 
			&\; \int_{0}^{t}  \me^m (\tau)^{\f34}  \md^m (\tau)^{\f14} \md_{tan}^{m-1}(\tau)^{\f12} \ d\tau\\ 
			\lesssim &\; \left\{\underset{0\le \tau \le t}{\sup} \me^m(\tau)\right\}^{\f34}
			\left\{\int_{0}^{t} \md_{tan}^{m-1} (\tau) (1+\tau)^{\sigma} d \tau \right\}^{\f12} \left\{\int_{0}^{t} \md^{m} (\tau)  d \tau\right\}^{\f14}  \left\{\int_{0}^{t} (1+\tau)^{-2\sigma} d \tau \right\}^{\f14}  \\
			\lesssim &\;  \dl^2,
			\dal \deqq
			and similarly,
			\beqq \bal
			\int_{0}^{t} \me^m(\tau)^{\f58}  \md^m(\tau)^{\f{11}{16}}  \md_{tan}^{m-1} (\tau)^{\f{3}{16}}\ d\tau \lesssim \dl^{\f32},
			\dal\deqq 
			where we have used the condition $\f{9}{10} < \sigma < s < 1$. Thus, we complete the proof of this lemma.
		\end{proof}

		\subsection{One-order vertical derivative estimate}
		In order to obtain the decay estimate of solution, we will establish the estimate for the vertical derivative
		of density, velocity and magnetic field.
		\begin{lemm}
			Under the assumption \eqref{assumption}, for any smooth solution $(\vro ,u, b)$ of equation \eqref{eqr},
			it holds 
			\beq\label{3301} \bal
			&\; \frac{d}{dt}\Big(\|(\p_3 \vro, \p_3 u,\f{1}{\sqrt{1+\vro}} \p_3 b)\|_{L^2}^2 
			+ \sum\limits_{ |\alpha_h|=m-2}\|(\p_3 \p^{\al_h} \vro, \p_3 \p^{\al_h} u,  \f{1}{\sqrt{1+\vro}}\p^{\al_h} \p_3 b)\|_{L^2}^2 \Big)\\
			&\; 
			+\|\p_3 (\p_{1} u,  \du , \nabla_h b)\|_{H_{tan}^{m-2}}^2  
			+\var \|\p_3 (\p_{2} u, \p_3 u , \p_3 b )\|_{H_{tan}^{m-2}}^2\\
			\lesssim &\; \sqrt{\me^m(t)}\md_{tan}^{m-1}(t).
			\dal \deq
		\end{lemm}
		
		\begin{proof}
			For any $|\ah| \le m-2$, the equations \eqref{eqr} yields that
			\beq\label{3302}
			\begin{aligned}
				&\;\frac{d}{dt}\frac{1}{2}\i(|\p^{\ah}\p_3 u|^2+|\p^{\ah} \p_3 \vro|^2+\f{|\p^{\al_h} \p_3 b|^2}{1+ \vro}  )dx +\i(| \p^{\ah} \p_{13} u|^2+|\p^{\ah} \p_3 \du|^2+| \p^{\ah} \p_3 \nabla_h b|^2)dx\\
				&\; + \var \i(| \p^{\ah} \p_{23} u|^2+|\p^{\ah} \p_3^2 u|^2+| \p^{\ah}\p_3^2 b|^2)dx
				\\
				=&\;-\i \p^{\ah} \p_3 (u\cdot \nabla \vro +\vro \, \du )\cdot \p^{\ah} \p_3 \vro \ dx
				-\i \p^{\ah} \p_3 ( u \cdot \nabla u+ \vro \nabla \vro)\cdot \p^{\ah} \p_3 u\ dx\\
				& -\i \p^{\ah} \p_3  ( \f{\vro}{1+\vro}(\p_1^2 u+\nabla \du-\nabla b_2 +\p_2 b)) \cdot \p^{\ah} \p_3   u\ dx  +\i \p^{\ah} \p_3 ( \f{1}{1+\vro} b \cdot \nabla b) \cdot \p^{\ah} \p_3  u\ dx  \\
				&\;-\f12 \i \p^{\ah} \p_3( \f{1}{1+\vro} \nabla (|b|^2)) \cdot \p^{\ah} \p_3 u\ dx  +\i \f{1}{2(1+ \vro)^2} |\p^{\al_h} \p_3 b|^2  (\du +u \cdot \nabla \vro + \vro \, \du)\ dx\\
				&\;-\i \f{1}{1+ \vro} \p^{\ah} \p_3 (b \, \du + u \cdot \nabla b)\cdot \p^{\ah} \p_3 b \ dx +\i \f{1}{1+ \vro} \p^{\ah} \p_3 (b \cdot \nabla u)\cdot \p^{\ah} \p_3 b \ dx\\
				&\;
				+\i \f{\vro}{1+\vro} \p^{\al_h} \p_3 (e_2  \du-\Delta_h  b -\var \p_3^2 b
				- \p_2 u )\cdot \p^{\ah} \p_3 b\ dx \\
				&\; -\var \i \p^{\ah} \p_3  ( \f{\vro}{1+\vro}(\p_2^2 u+\p_3^2 u)) \cdot \p^{\ah} \p_3   u\ dx
				:=\sum_{i=1}^{10} III_i.
			\end{aligned}
			\deq
			\textbf{First of all, we deal with the case $|\alpha_h| = 0$. } 
			Integrating by part and using the  inequality \eqref{ie:Sobolev}, we have
			\begin{align*}
				III_1  = &\;  - \i (\p_3 u \cdot \nabla \vro +  u \cdot \p_3  \nabla \vro + \p_3 \vro\, \du + \vro  \, \p_3 \du ) \cdot \p_3 \vro \ dx\\
				= &\; - \i (\p_3 u \cdot \nabla \vro + \vro  \, \p_3 \du ) \cdot \p_3 \vro \ dx - \f12  \i   |\p_3 \vro|^2 \, \du \ dx\\
				\lesssim &\; \| \p_3 \vro \|_{L^2}^{\f12}\| \p_2 \p_3 \vro \|_{L^2}^{\f12}  (\| \p_3 u_h\|_{L^2}^{\f12} \| \p_1 \p_3 u_h\|_{L^2}^{\f12} \| \nabla_h \vro\|_{L^2}^{\f12} \| \p_3 \nabla_h \vro\|_{L^2}^{\f12}+  \| \p_3 u_3\|_{L^2}^{\f12} \| \p_3^2 u_3\|_{L^2}^{\f12} \| \p_3 \vro\|_{L^2}^{\f12} \| \p_{13} \vro\|_{L^2}^{\f12}) \\
				&\;  + \| \p_3 \du\|_{L^2} \| \vro \|_{L^2}^{\f14} \| \p_2 \vro \|_{L^2}^{\f14}\| \p_3 \vro \|_{L^2}^{\f14}\| \p_{23} \vro \|_{L^2}^{\f14} 
				\| \p_3 \vro\|_{L^2}^{\f12} \| \p_{13} \vro\|_{L^2}^{\f12}\\
				&\; + \| \p_3 \vro\|_{L^2} \|\nabla_h \p_{3} \vro\|_{L^2} \| \du \|_{L^2}^{\f12} \|\p_3 \du\|_{L^2}^{\f12},
			\end{align*}
			and 
			\beqq \bal
			III_2 = &\; - \i (\p_3 u \cdot \nabla u + u \cdot \nabla \p_3 u + \f12 \p_3 \nabla( \vro^2)  ) \cdot \p_3 u \ dx \\
			= &\; - \i \p_3 u \cdot \nabla u \cdot \p_3 u \ dx  + \f12 \i |\p_3 u|^2 \du \ dx + \f12 \i \p_3 (\vro^2) \, \du \ dx \\
			\lesssim &\;  \| \p_3 u \|_{L^2}^{\f12}\| \p_2 \p_3 u\|_{L^2}^{\f12}  ( \| \p_3 u_h\|_{L^2}^{\f12} \| \p_1 \p_3 u_h\|_{L^2}^{\f12} \| \nabla_h u\|_{L^2}^{\f12} \| \p_3 \nabla_h u\|_{L^2}^{\f12} + \| \p_3 u_3\|_{L^2}^{\f12} \| \p_3^2 u_3\|_{L^2}^{\f12} \| \p_3 u\|_{L^2}^{\f12} \| \p_{13} u\|_{L^2}^{\f12})\\
			&\; + \|\p_3 u\|_{L^2} \|\nabla_h \p_3 u\|_{L^2} \|\du\|_{L^2}^{\f12}\|\p_3 \du\|_{L^2}^{\f12}  + \| \du\|_{L^2}^{\f12}\| \p_3 \du\|_{L^2}^{\f12} \| \vro \|_{L^2}^{\f12} \| \p_2 \vro \|_{L^2}^{\f12}
			\| \p_3 \vro\|_{L^2}^{\f12} \| \p_{13} \vro\|_{L^2}^{\f12}.\\
			\dal \deqq
			Next, we estimate the term $III_3$.  
			\beqq \bal
			III_3 =  &\; \i \p_3 (\f{\vro}{1+\vro}) \p_1^2 u  \cdot \p_3 u \ dx + \i \p_3 (\f{\vro}{1+\vro}) ( \nabla \du +\nabla b_2 -\p_2 b ) \cdot \p_3 u \ dx  \\
			&\; +  \i \f{\vro}{1+\vro}(\p_1^2   \p_3 u + \nabla \p_3 \du +\p_3 \nabla b_2 ) \cdot \p_3 u \ dx  -  \i \f{\vro}{1+\vro} \p_{23} b \cdot \p_3 u \ dx \\
			=  &\;- \i \p_1 u   \cdot \p_1 (\p_3 (\f{\vro}{1+\vro}) \p_3 u) \ dx + \i \p_3 (\f{\vro}{1+\vro}) ( \nabla \du +\p_3 \nabla b_2 -\p_2 b ) \cdot \p_3 u \ dx  \\
			&\; - \i  \p_{13}  u \cdot \p_1  (\f{\vro}{1+\vro} \p_3 u) \ dx
			- \i ( \p_3 \du + \p_3 b_2) \, {\rm{div}} (\f{\vro}{1+\vro} \p_3 u) \ dx  
			-  \i \f{\vro}{1+\vro} \p_{23} b \cdot \p_3 u \ dx.
			\dal  \deqq
			We only estimate the fourth term on the right side and the other terms can be estimated in a similar way.  Using the  anisotropic type inequality \eqref{ie:Sobolev}, we have
			\beqq \bal
			&\; - \i ( \p_3 \du + \p_3 b_2) \, {\rm{div}} (\f{\vro}{1+\vro} \p_3 u) \ dx  \\  
			=&\; -\i ( \p_3 \du + \p_3 b_2) \,  (\f{\vro}{1+\vro} \p_3 \du + u_h \cdot \nabla_h(\f{\vro}{1+\vro}) +  u_3 \, \p_3(\f{\vro}{1+\vro}) ) \ dx  \\
			\lesssim &\; \| \p_3 \du \|_{L^2}^2 \| \vro\|_{L^{\infty}} + \| \p_3 \du \|_{L^2} \| \p_3 b_2 \|_{L^2}^{\f12} \| \p_{23} b_2 \|_{L^2}^{\f12} \|(\vro, \p_3 \vro)\|_{L^2}^{\f12}  \|\p_1 (\vro, \p_3 \vro)\|_{L^2}^{\f12} \\
			&\; +  \|\p_3 b_2\|_{L^2}^{\f12}\|\p_{23} b_2\|_{L^2}^{\f12}( \|u_h\|_{L^2}^{\f12} \| \p_1 u_h\|_{L^2}^{\f12} \| \nabla_h (\f{\vro}{1+\vro})\|_{L^2}^{\f12} \| \p_3 \nabla_h (\f{\vro}{1+\vro})\|_{L^2}^{\f12} \\
			&\; + \| u_3\|_{L^2}^{\f12} \| \p_3 u_3\|_{L^2}^{\f12} \| \p_3 (\f{\vro}{1+\vro})\|_{L^2}^{\f12} \| \p_{13} (\f{\vro}{1+\vro})\|_{L^2}^{\f12})\\
			&\; +\|\p_3 \du\|_{L^2}( \|u_h\|_{L^2}^{\f14} \| \p_1 u_h\|_{L^2}^{\f14} \p_2 u_h\|_{L^2}^{\f14} \|\p_{12} u_h\|_{L^2}^{\f14} \| \nabla_h (\f{\vro}{1+\vro})\|_{L^2}^{\f12} \| \p_3 \nabla_h (\f{\vro}{1+\vro})\|_{L^2}^{\f12} \\
			&\; + \| u_3\|_{L^2}^{\f14} \| \p_3 u_3\|_{L^2}^{\f14}  \| \p_{1} u_3\|_{L^2}^{\f14}  \| \p_{13} u_3\|_{L^2}^{\f14} \| \p_3 (\f{\vro}{1+\vro})\|_{L^2}^{\f12} \| \p_{23} (\f{\vro}{1+\vro})\|_{L^2}^{\f12}),
			\dal \deqq
			therefore under the assumption \eqref{assumption}, we can obtain that
			\beqq
			III_3 \lesssim  \sqrt{\me^m(t)}\md_{tan}^{m-1}(t).
			\deqq
			Similarly, integrating by parts, we have
			\beqq \bal
			III_{10} = &\;-\var \i \p_3 (\f{\vro}{1+\vro}) ( \p_2^2 u+\p_3^2 u ) \cdot \p_3 u \ dx  -\var \i \f{\vro}{1+\vro}(\p_2^2   \p_3 u + \p_3^3 u) \cdot \p_3 u \ dx \\
			= &\;-\var \i \p_3 (\f{\vro}{1+\vro}) ( \p_2^2 u+\p_3^2 u ) \cdot \p_3 u \ dx  + \var \i \p_2 (\f{\vro}{1+\vro} \p_{3} u)
			\cdot \p_{23}u \ dx \\
			&\; + \var \i \p_3 (\f{\vro}{1+\vro} \p_{3} u)
			\cdot \p_{3}^2 u \ dx \\
			\lesssim &\; \sqrt{\me^m(t)}\md_{tan}^{m-1}(t).
			\dal \deqq
			Next, we deal with the terms $III_4$ and $III_8$ together.
			\begin{align*}
				&\; III_4 +III_8 \\
				= &\;  \i \f{1}{1+\vro} ( b \cdot \nabla \p_3 u \cdot \p_3 b +  b \cdot \nabla \p_3 b \cdot \p_3 u )\, dx  + \i \f{1}{1+\vro} (\p_3 b \cdot \nabla u \cdot \p_3 b + \p_3 b \cdot \nabla b \cdot \p_3 u )\, dx  \\
				&\; +  \i  \p_3 ( \f{1}{1+\vro}) b \cdot \nabla b \cdot \p_3  u\ dx \\
				= &\; - \i b \cdot \nabla (\f{1}{1+\vro}) \p_3 u \cdot \p_3 b  \ dx + \i \f{1}{1+\vro} (\p_3 b \cdot \nabla u \cdot \p_3 b + \p_3 b \cdot \nabla b \cdot \p_3 u )\, dx  \\
				&\; +  \i  \p_3 ( \f{1}{1+\vro}) b \cdot \nabla b \cdot \p_3  u\ dx.
			\end{align*}
			Using the  anisotropic type inequality \eqref{ie:Sobolev}, it is easy to check that
			\beqq
			III_4 + III_8 \lesssim  \sqrt{\me^m(t)}\md_{tan}^{m-1}(t),
			\deqq
			and in a similar way, we have
			\beqq
			III_5 +III_6 +III_7 +III_9 \lesssim  \sqrt{\me^m(t)}\md_{tan}^{m-1}(t).
			\deqq
			Thus, substituting the estimates from $III_1$  to $III_{10}$ into \eqref{3302}, we have
			\beqq \bal
			&\; \frac{d}{dt}\frac{1}{2}\i(| \p_3 \vro|^2+|\p_3 u|^2+\f{1}{1+ \vro} |\p_3 b|^2 )dx +\i(|\p_{13} u|^2+| \p_3 \du|^2+|\nabla_h \p_3 b|^2)dx \\
			&\; + \var \i (|\p_{23} u|^2 + |\p_{3}^2 u|^2 +|\p_{3}^2 b|^2 )\ dx \\
			\lesssim &\;  \sqrt{\me^m(t)}\md_{tan}^{m-1}(t).
			\dal \deqq
			\textbf{Now let us deal with the case $|\alpha_h| = m-2$}.  Integrating by parts, we have
			\beqq \bal
			III_1 = &\;  -\i \p^{\ah}( \p_3  u\cdot \nabla \vro +  u\cdot \p_3  \nabla \vro  +  \p_3  \vro \,   \du  + \vro \,  \p_3  \du )\cdot \p^{\ah} \p_3 \vro \ dx \\
			= &\;  -\i \p^{\ah}( \p_3  u \cdot \nabla \vro ) \cdot \p^{\ah} \p_3 \vro \ dx + \f12 \i  |\p_3 \p^{\ah} \vro|^2 \, \du  \ dx\\
			&\; - \sum_{0 < \beta_h \le \ah} C_{\ah}^{\beta_h} \i  (\p^{\beta_h} u_h  \cdot  \p^{\ah-\beta_h} \nabla_h \p_3 \vro +\p^{\beta_h} u_3  \cdot  \p^{\ah-\beta_h} \p_3^2 \vro ) \cdot   \p^{\ah} \p_3 \vro \ dx \\
			&\; -\i \p^{\ah}(  \p_3  \vro \,   \du  + \vro \,  \p_3  \du )\cdot \p^{\ah} \p_3 \vro \ dx\\
			:= &\; \sum_{i=1}^{4} III_{1,i}.
			\dal \deqq
			First, we deal with the terms $III_{1,1}$ and $III_{1,2}$.  Using the  anisotropic type inequality \eqref{ie:Sobolev}, we can obtain
			\beqq \bal
			III_{1,1} + III_{1,2} 
			\lesssim  &\; \|  \p^{\ah} \p_3 \vro\|_{L^2}   \sum_{ [\f{\alpha_h}{2}] < \beta_h \le  \ah}  \|\p^{\beta_h} \p_3 u \|_{L^2}^{\f12} \|\p_1 \p^{\beta_h} \p_3 u \|_{L^2}^{\f12}
			\|\p^{\al_h-\beta_h}  \nabla \vro \|_{H^2} \\
			&\; +  \|  \p^{\ah} \p_3 \vro\|_{L^2} 
			\sum_{0 \le \beta_h \le [\f{\alpha_h}{2}] } \|\p^{\beta_h} \p_3 u\|_{L^{\infty}}
			\|\p^{\al_h-\beta_h}  \nabla \vro \|_{L^2}^{\f12}
			+ \|  \p^{\ah} \p_3 \vro\|_{L^2}^2 \|\du\|_{L^{\infty}}.
			\dal \deqq
			Next, we deal with the term $III_{1,3}$.
			\begin{align*}
				III_{1,3} \lesssim &\;  \|  \p^{\ah} \p_3 \vro\|_{L^2} \Big( \sum_{0 < \beta_h \le [\f{\alpha_h}{2}] } (\|  \p^{\beta_h} u_h \|_{L^{\infty}}
				\|\p^{\al_h-\beta_h}  \nabla_h \p_3 \vro \|_{L^2}  +\|  \p^{\beta_h} u_3 \|_{L^{\infty}}
				\|\p^{\al_h-\beta_h}  \p_3^2 \vro \|_{L^2} )\\
				&\; +
				\sum_{ [\f{\alpha_h}{2}] < \beta_h \le  \ah}  (\|\p^{\beta_h} u_h\|_{L^2}^{\f12}
				\|\p^{\beta_h}  \p_3 u_h\|_{L^2}^{\f12}
				\|\p^{\al_h-\beta_h}  \nabla_h \p_3 \vro \|_{H_{tan}^2} \\
				&\;+\|\p^{\beta_h} u_3\|_{L^2}^{\f12}
				\|\p^{\beta_h}  \p_3 u_3\|_{L^2}^{\f12}
				\|\p^{\al_h-\beta_h}  \p_3^2 \vro \|_{H_{tan}^2} )\Big)\\
				\lesssim  &\;
				\sqrt{\me^m(t)}\md_{tan}^{m-1}(t),
			\end{align*}
			where, for $0 < |\beta_h| \le [\f{\alpha_h}{2}]$, we have used the following estimate
			\beqq
			\| \p^{\beta_h} u_3 \|_{L^{\infty}} \lesssim \|(\p^{\beta_h} u_3,\p_3 \p^{\beta_h} u_3 )\|_{L^2}^{\f14}  \|\p_h(\p^{\beta_h} u_3, \p_3 \p^{\beta_h} u_3)\|_{H_{tan}^1}^{\f34}.
			\deqq
			Similarly, we have
			\beqq
			III_{1,4}  =-\i \p^{\ah}(  \p_3  \vro \,   \du  + \vro \,  \p_3  \du )\cdot \p^{\ah} \p_3 \vro \ dx \lesssim 
			\sqrt{\me^m(t)}\md_{tan}^{m-1}(t),
			\deqq
			which yields that
			\beqq 
			III_{1}  \lesssim 
			\sqrt{\me^m(t)}\md_{tan}^{m-1}(t).
			\deqq
			Estimate of the term $III_2$.  Integrating by parts, we have
			\beqq \bal
			III_2 = &\; - \i \p^{\ah} (\p_3 u \cdot \nabla u + u \cdot \nabla \p_3 u + \f12 \p_3 \nabla( \vro^2)  ) \cdot \p^{\ah} \p_3 u \ dx \\
			= &\; - \i \p^{\ah} (\p_3 u \cdot \nabla u) \cdot \p^{\ah} \p_3 u \ dx  - \sum_{0 < \beta_h \le \al_h} C_{\ah}^{\beta_h} \i \p^{\beta_h} u \cdot \nabla \p^{\ah-\beta_h} \p_3 u \cdot \p^{\ah} \p_3 u \ dx\\
			&\;+ \f12 \i |\p^{\ah} \p_3 u|^2 \du \ dx + \f12 \i \p^{\ah} \p_3 (\vro^2) \cdot \p^{\ah} \p_3 \du \ dx.
			\dal \deqq
			Using the  anisotropic type inequality \eqref{ie:Sobolev},  we can check that
			\beqq 
			\bal
			III_2
			\lesssim &\; \| \p^{\ah} \p_3 u\|_{L^2}^{\f12}  \| \p^{\ah} \p_{13} u\|_{L^2}^{\f12} \Big( \sum_{0 \le \beta_h \le [\f{\alpha_h}{2}] } \|  \p^{\beta_h} \p_3 u \|_{H^2}
			\|\p^{\al_h-\beta_h}  \nabla u \|_{L^2}  \\
			&\; +
			\sum_{0 < \beta_h \le [\f{\alpha_h}{2}] } \|  \p^{\beta_h} u\|_{H_{tan}^1}^{\f12}
			\| \p_3 \p^{\beta_h} u\|_{H_{tan}^1}^{\f12}
			\|\nabla \p^{\al_h-\beta_h}  \p_3 u \|_{L^2} \\
			&\; +
			\sum_{ [\f{\alpha_h}{2}] < \beta_h \le  \ah}  (\|\p^{\beta_h} \p_3 u\|_{L^2}
			\|\p^{\al_h-\beta_h}  \nabla u \|_{H^2} +\|\p^{\beta_h} u \|_{L^2}^{\f12}
			\|\p^{\beta_h}  \p_3 u \|_{L^2}^{\f12}
			\|\nabla \p^{\al_h-\beta_h}  \p_3 u\|_{H_{tan}^1})\Big)\\
			&\; +  \| \p^{\ah} \p_3 \du\|_{L^2} 
			(
			\sum_{0 \le \beta_h \le [\f{\alpha_h}{2}] } \| \p^{\beta_h} \vro\|_{H^2} \| \p^{\ah-\beta_h} \p_3 \vro\|_{L^2}  + \sum_{ [\f{\alpha_h}{2}] < \beta_h \le  \ah}  
			\| \p^{\beta_h} \vro\|_{L^2} \| \p^{\ah-\beta_h} \p_3 \vro\|_{H^2} )\\
			&\; + \| \du\|_{L^{\infty}}   \| \p^{\ah} \p_3 u\|_{L^2}^2 \\
			\lesssim  &\;  \sqrt{\me^m(t)}\md_{tan}^{m-1}(t).
			\dal \deqq
			Estimate of the term $III_3$.  We first split the term as follows.
			\beqq \bal
			III_3 = &\; -\i \p^{\ah} \p_3  ( \f{\vro}{1+\vro} \p_1^2 u) \cdot \p^{\ah} \p_3   u\ dx 
			-\i \p^{\ah} \p_3  ( \f{\vro}{1+\vro} \nabla \du) \cdot \p^{\ah} \p_3   u\ dx  \\
			&\; +\i \p^{\ah} \p_3  ( \f{\vro}{1+\vro}(\nabla b_2 -\p_2 b)) \cdot \p^{\ah} \p_3   u\ dx
			:= \sum_{i=1}^3 III_{3,i}.
			\dal \deqq
			Integrating by parts, we have
			\begin{align*}
				III_{3,1} = &\;  -\i 
				\p^{\ah}  ( \p_3 (\f{\vro}{1+\vro}) \p_1^2 u) \cdot \p^{\ah} \p_3   u\ dx - \i  \f{\vro}{1+\vro} \p^{\ah}   \p_1^2 \p_3 u \cdot \p^{\ah} \p_3   u\ dx  \\
				&\;  - \sum_{0 < \beta_h \le \al_h} C_{\ah}^{\beta_h} \i \p^{\beta_h} (\f{\vro}{1+\vro}) \cdot \p^{\ah-\beta_h} \p_1^2  \p_3 u
				\cdot \p^{\ah} \p_3   u\ dx \\
				= &\;  - \sum_{0 \le \beta_h \le \al_h} C_{\ah}^{\beta_h} 
				\i \p^{\beta_h}   \p_3 (\f{\vro}{1+\vro}) \p^{\ah-\beta_h} \p_1^2 u \cdot \p^{\ah} \p_3   u\ dx \\
				&\;  - \sum_{0 < \beta_h \le \al_h} C_{\ah}^{\beta_h} \i \p^{\beta_h} (\f{\vro}{1+\vro}) \cdot \p^{\ah-\beta_h} \p_1^2  \p_3 u 
				\cdot \p^{\ah} \p_3  u \ dx \\
				&\; + \i \p^{\ah}   \p_{13} u \cdot \p_1 (\f{\vro}{1+\vro}   \p^{\ah} \p_3   u) \ dx. 
			\end{align*} 
			Under the assumption \eqref{assumption} and using the  anisotropic type inequality \eqref{ie:Sobolev}, we have
			\beqq \bal
			III_{3,1} \lesssim &\; \| \p^{\ah} \p_3 u\|_{L^2}^{\f12}  \| \p^{\ah} \p_{13} u\|_{L^2}^{\f12} \Big( \sum_{0 \le \beta_h \le [\f{\alpha_h}{2}] }\|  \p^{\beta_h} \p_3 (\f{\vro}{1+\vro}) \|_{H_{tan}^1}^{\f12}
			\|  \p^{\beta_h} \p_3^2 (\f{\vro}{1+\vro}) \|_{H_{tan}^1}^{\f12}
			\|\p^{\al_h-\beta_h}  \p_1^2 u \|_{L^2}  \\
			&\; +
			\sum_{0 < \beta_h \le [\f{\alpha_h}{2}] } \|  \p^{\beta_h} (\f{\vro}{1+\vro}) \|_{H_{tan}^1}^{\f12}
			\| \p_3 \p^{\beta_h} (\f{\vro}{1+\vro}) \|_{H_{tan}^1}^{\f12}
			\|\p^{\al_h-\beta_h}  \p_1^2 \p_3 u \|_{L^2} \\
			&\; +
			\sum_{ [\f{\alpha_h}{2}] < \beta_h \le  \ah}  (\|  \p^{\beta_h} \p_3 (\f{\vro}{1+\vro}) \|_{L^2}
			\|\p^{\al_h-\beta_h}  \p_1^2 u \|_{H^2} +\|\p^{\beta_h} (\f{\vro}{1+\vro})  \|_{H_{tan}^1}
			\|\p^{\al_h-\beta_h} \p_1^2 \p_3 u \|_{H^1})\Big)\\
			&\; + 
			\|\p^{\ah}   \p_{13} u\|_{L^2} ( \|\p^{\ah}   \p_{13} u\|_{L^2} \| \vro\|_{L^{\infty}} +  \|\p^{\ah} \p_{3} u\|_{L^2} \| \p_1( \f{\vro}{1+\vro} )\|_{L^{\infty}} )\\
			\lesssim  &\;  \sqrt{\me^m(t)}\md_{tan}^{m-1}(t).
			\dal \deqq
			Similarly,  integrating by parts, we have
			\beqq \bal
			III_{3,2} = &\;  -\i 
			\p^{\ah}  \p_3\left( \nabla (\f{\vro}{1+\vro} \du) - \nabla (\f{\vro}{1+\vro})   \du \right) \cdot \p^{\ah} \p_3   u\ dx  \\
			= &\;   \i 
			\p^{\ah} \p_3 (  \f{\vro}{1+\vro} \du ) \cdot \p^{\ah} \p_3  \du\ dx   + \i \p^{\ah}  ( \nabla ( \f{\vro}{1+\vro})  \p_3 \du) \cdot \p^{\ah} \p_3   u\ dx  \\
			&\; +\i \du \, \p^{\ah} \p_3   u \cdot \nabla  \p^{\ah} \p_3 (\f{\vro}{1+\vro})  \ dx + \sum_{0 \le \beta_h < \al_h} C_{\ah}^{\beta_h} 
			\i \p^{\ah-\beta_h} \du \, \p^{\ah} \p_3   u \cdot \nabla \p^{\beta_h}   \p_3 (\f{\vro}{1+\vro})   \ dx \\
			= &\; \i 
			\p^{\ah} ( \p_3  (\f{\vro}{1+\vro}) \du + \f{\vro}{1+\vro} \p_3   \du ) \cdot \p^{\ah} \p_3  \du\ dx   + \i \p^{\ah}  ( \nabla (\f{\vro}{1+\vro})  \p_3 \du) \cdot \p^{\ah} \p_3   u\ dx  \\
			&\; -\i \p^{\ah} \p_3 (\f{\vro}{1+\vro}) {\rm{div}}( \du \, \p^{\ah} \p_3   u )\ dx + \sum_{0 \le \beta_h < \al_h} C_{\ah}^{\beta_h} 
			\i \p^{\ah-\beta_h} \du \, \p^{\ah} \p_3   u \cdot \nabla \p^{\beta_h}   \p_3 (\f{\vro}{1+\vro})   \ dx  \\
			\lesssim  &\;  \sqrt{\me^m(t)}\md_{tan}^{m-1}(t),
			\dal \deqq 
			and 
			\beqq \bal
			III_{3,3} =&\; \i \p^{\ah}   \left( \p_3(\f{\vro}{1+\vro})(\nabla b_2  -\p_2 b) \right) \cdot \p^{\ah} \p_3   u\ dx - \i \p^{\ah}   \left( \f{\vro}{1+\vro} \p_{23} b \right) \cdot \p^{\ah} \p_3   u\ dx \\
			&\;  + \i  \f{\vro}{1+\vro}\p^{\ah} \p_3   u  \cdot \nabla  \p^{\ah} \p_3 b_2  \ dx + \sum_{0 < \beta_h \le \al_h} C_{\ah}^{\beta_h}  \i  \p^{\beta_h} (\f{\vro}{1+\vro} ) \p^{\ah} \p_3   u \cdot \nabla \p^{\ah-\beta_h}  \p_3 b_2 \ dx \\
			=&\; \i \p^{\ah}   \left( \p_3(\f{\vro}{1+\vro})(\nabla b_2  -\p_2 b) \right) \cdot \p^{\ah} \p_3   u\ dx - \i \p^{\ah}   \left( \f{\vro}{1+\vro} \p_{23} b \right) \cdot \p^{\ah} \p_3   u\ dx \\
			&\;  - \i  \p^{\ah}  \p_3 b_2 \, {\rm{div}}( \f{\vro}{1+\vro}\p^{\ah} \p_3 u) \ dx + \sum_{0 < \beta_h \le \al_h} C_{\ah}^{\beta_h}  \i  \p^{\beta_h} (\f{\vro}{1+\vro} ) \p^{\ah} \p_3   u \cdot \nabla \p^{\ah-\beta_h}  \p_3 b_2 \ dx \\
			\lesssim &\; \sqrt{\me^m(t)}\md_{tan}^{m-1}(t).
			\dal \deqq
			Thus, combining the estimates from $III_{3,1}$ to  $III_{3,3}$, we have
			\beqq 
			III_{3} \lesssim  \sqrt{\me^m(t)}\md_{tan}^{m-1}(t).
			\deqq
			Next, we deal with the terms $III_4$ and $III_8$ together. We split them as follows:
			\begin{align*}
				III_4 = &\; \i \p^{\ah} (  \p_3(\f{1}{1+\vro}) b \cdot \nabla b) \cdot \p^{\ah} \p_3  u\ dx  + \i \p^{\ah} (  \f{1}{1+\vro}  \p_3 b \cdot \nabla b) \cdot \p^{\ah} \p_3  u\ dx \\
				&\; + \i   \f{1}{1+\vro}  b \cdot \nabla \p^{\ah}  \p_3  b \cdot \p^{\ah} \p_3  u\ dx  \\
				&\; + \sum_{0 < \beta_h \le \al_h} C_{\ah}^{\beta_h} \i \f{1}{1+ \vro} \p^{\beta_h} b \cdot \nabla \p^{\ah-\beta_h}  \p_3 b
				\cdot \p^{\ah} \p_3  u \ dx \\
				&\; + \sum_{0 < \beta_h \le \al_h} C_{\ah}^{\beta_h} \i\p^{\beta_h} ( \f{1}{1+ \vro} )\cdot \p^{\ah-\beta_h} (b \cdot \nabla \p_3 b )
				\cdot \p^{\ah} \p_3  u \ dx \\
				:=  &\; \sum_{i=1}^{5} III_{4,i},       
			\end{align*}
			and
			\beqq \bal
			III_8=
			&\;  \i \f{1}{1+ \vro} \p^{\ah}  (\p_3 b \cdot \nabla u)\cdot \p^{\ah} \p_3 b \ dx + \i \f{1}{1+ \vro}   b \cdot \nabla \p^{\ah}  \p_3  u\cdot \p^{\ah} \p_3 b \ dx \\
			&\; + \sum_{0 < \beta_h \le \al_h} C_{\ah}^{\beta_h} \i \f{1}{1+ \vro} \p^{\beta_h} b \cdot  \nabla \p^{\ah-\beta_h}  \p_3 u 
			\cdot \p^{\ah} \p_3  b \ dx\\
			:=  &\; \sum_{i=1}^{3} III_{8,i}.
			\dal \deqq
			Integrating by parts and using the condition ${\rm div} b=0$, we have
			\beqq \bal
			III_{4,3} + III_{8,2}
			= &\; \i \f{1}{1+ \vro}  b \cdot \nabla (\p^{\ah} \p_3  u \cdot \p^{\ah} \p_3 b) \ dx = -\i b \cdot  \nabla (\f{1}{1+ \vro})  \, \p^{\ah} \p_3  u \cdot \p^{\ah} \p_3 b \ dx \\
			\lesssim &\; \| \p^{\ah} \p_3  u\|_{L^2}^{\f12}\| \p_1 \p^{\ah} \p_3  u\|_{L^2}^{\f12} \| \p^{\ah} \p_3  b\|_{L^2}^{\f12}\| \p_2 \p^{\ah} \p_3  b\|_{L^2}^{\f12} \| b\|_{L^{\infty}} \| \nabla (\f{1}{1+\vro}) \|_{L^2}^{\f12} \| \p_3 \nabla (\f{1}{1+\vro}) \|_{L^2} \\
			\lesssim &\; \sqrt{\me^m(t)}\md_{tan}^{m-1}(t),
			\dal \deqq
			here we have used the assumption \eqref{assumption}.
			Under the assumption \eqref{assumption} and using the  anisotropic type inequality \eqref{ie:Sobolev}, we have
			\beqq \bal
			&\; III_{4,1} + III_{4,2} \\
			= &\; \sum_{0 \le \beta_h \le \al_h} C_{\ah}^{\beta_h} \i (\p^{\beta_h} \p_3 ( \f{1}{1+ \vro} ) \cdot \p^{\ah-\beta_h} (b \cdot \nabla b) + \p^{\beta_h} ( \f{1}{1+ \vro} ) \cdot \p^{\ah-\beta_h} (\p_3 b \cdot \nabla b) )
			\cdot \p^{\ah} \p_3  u \ dx \\
			\lesssim &\; \| \p^{\ah} \p_3  u \|_{L^2}^{\f12} \| \p^{\ah} \p_{13} u \|_{L^2}^{\f12} \Big( \sum_{0 \le \beta_h \le [\f{\alpha_h}{2}] }\|  \p^{\beta_h} \p_3 (\f{\vro}{1+\vro}) \|_{H_{tan}^1}^{\f12}
			\|  \p^{\beta_h} \p_3^2 (\f{\vro}{1+\vro}) \|_{H_{tan}^1}^{\f12}
			\|\p^{\al_h-\beta_h}  (b \cdot \nabla b) \|_{L^2}  \\
			&\; +
			\sum_{0 \le  \beta_h \le [\f{\alpha_h}{2}] } \|  \p^{\beta_h} (\f{\vro}{1+\vro}) \|_{H_{tan}^1}^{\f12}
			\| \p_3 \p^{\beta_h} (\f{\vro}{1+\vro}) \|_{H_{tan}^1}^{\f12}
			\|\p^{\al_h-\beta_h}  (\p_3 b \cdot \nabla b) \|_{L^2} \\
			&\; +
			\sum_{ [\f{\alpha_h}{2}] < \beta_h \le  \ah}  (\|  \p^{\beta_h} \p_3 (\f{\vro}{1+\vro}) \|_{L^2}
			\|\p^{\al_h-\beta_h}  (b \cdot \nabla b) \|_{H^2} +\|\p^{\beta_h} (\f{\vro}{1+\vro})  \|_{H^1}
			\|\p^{\al_h-\beta_h} (\p_3 b \cdot \nabla b) \|_{H_{tan}^1})\Big)\\
			\lesssim &\; \sqrt{\me^m(t)}\md_{tan}^{m-1}(t),
			\dal \deqq
			and similarly, we have
			\beqq \bal
			III_{4,4} + III_{4,5}  +
			III_{8,1} + III_{8,3}  \lesssim &\; \sqrt{\me^m(t)}\md_{tan}^{m-1}(t),
			\dal \deqq
			which, together with the estimate of $III_{4,3}$ and $III_{8,2}$, yields that
			\beqq
			III_{4} + III_{8}  \lesssim  \sqrt{\me^m(t)}\md_{tan}^{m-1}(t).
			\deqq
			In a similar way, we have
			\beqq
			III_5 +III_6 +III_7 +III_9  + III_{10} \lesssim  \sqrt{\me^m(t)}\md_{tan}^{m-1}(t).
			\deqq
			Thus, for $|\al_h|=m-2$,   substituting the estimates from $III_1$  to $III_{10}$ into \eqref{3302}, we have 
			\beqq \bal
			&\; \frac{d}{dt}\frac{1}{2}\i(|\p^{\ah} \nabla_h \p_3 u|^2+| \p^{\ah} \nabla_h \p_3 \vro|^2+\f{1}{1+ \vro} |\p^{\ah} \nabla_h \p_3 b|^2 )dx +\i(|\p_{13} \p^{\ah}u|^2+|\p^{\ah} \p_3 \du|^2)dx \\
			&\; +\i |\nabla_h \p^{\ah}\p_3 b|^2 dx+ \var \i(|\p_{23} \p^{\ah}u|^2+|\p^{\ah} \p_3^2 u|^2+| \p^{\ah}\p_3^2 b|^2)dx \\
			\lesssim  &\;  \sqrt{\me^m(t)}\md_{tan}^{m-1}(t).
			\dal \deqq
			Combining the above estimate and the estimate of $|\ah|=0$, we finish the proof of this lemma.
		\end{proof}
		
		\subsection{Enhanced dissipation estimates of density and velocity field}
		In this subsection, we will establish the dissipative estimates of  the density in the horizontal directions and velocity field in the $x_2$ direction. On one hand, the enhanced
		dissipative estimate for the density
		comes from the pressure in the momentum equation $\eqref{eqr}_2$.
		On the other hand, the dissipative estimate for the velocity field arises from the good stability effect of magnetic field near a background magnetic field.
		
		\begin{lemm}
			Under the assumption \eqref{assumption}, for any smooth solution $(\vro ,u, b)$ of equation \eqref{eqr}, it holds for $|\ah| = k $($k=m-1 $ or $k=m-2$), 
			\beq \label{3401}
			\begin{aligned}
				&\; \f{d}{dt} \i \Big(u_h \cdot \nabla_h  \vro +  \p_{2} b \cdot u +  \sum_{|\ah|=k}( \p^{\ah}  u_h \cdot \nabla_h \p^{\ah} \vro +  \p_{2} \p^{\ah} b \cdot  \p^{\ah} u ) \Big)\ dx  + \|(\nabla_h  \vro, \p_{2} u)\|_{H_{tan}^k}^2  \\
				\lesssim &\;\var \| (\p_2^2 u, \p_3^2 u)\|_{H_{tan}^k}^2 + 
				\|(\p_1 u, \du, \nabla_h b)\|_{H^{k+1}_{tan}}^2
				+\sqrt{\me^m(t)}\md_{tan}^{k+1}(t).
			\end{aligned}
			\deq
		\end{lemm}
		\begin{proof}
			For any $|\ah|\le m-1$, the equations $\eqref{eqr}_2$  and $\eqref{eqr}_3$ yield that
			\beqq \bal
			\|\p^{\ah} \nabla_h \vro\|_{L^2}^2  = &\; -\i \p_t \p^{\ah} u_h \cdot \nabla_h \p^{\ah} \vro \ dx + \i \p^{\ah} ((F_2)_h + \p_1^2 u_h + \nabla_h \du + \p_2 b_h - \nabla_h b_2) \cdot \nabla_h \p^{\ah} \vro \ dx \\
			&\; + \var \i \p^{\ah} (\p_2^2 u_h + \p_3^2 u_h) \cdot \nabla_h \p^{\ah} \vro \ dx,  \\
			\|\p^{\ah} \p_2 u\|_{L^2}^2  = &\;  \i \p_t \p^{\ah} b \cdot \p^{\ah} \p_2 u\ dx + \i \p^{\ah} (e_2  \du - \Delta_h b + u \cdot \nabla b + b\,\du- b \cdot \nabla u) \cdot \p^{\ah} \p_2 u \ dx\\
			&\; -\var \i \p^{\ah} \p_3^2  b \cdot \p^{\ah} \p_2 u \ dx.
			\dal \deqq
			On one hand, using the density equation $\eqref{eqr}_1$ and integrating by parts, we conclude
			\beqq \bal
			&\; -\i \p_t \p^{\ah} u_h \cdot \nabla_h \p^{\ah} \vro \ dx \\ 
			= &\; - \p_t \i \p^{\ah} u_h \cdot \nabla_h \p^{\ah} \vro \ dx  +\i  \p^{\ah} u_h \cdot \nabla_h \p^{\ah} \p_t \vro \ dx \\
			= &\;  -  \p_t \i \p^{\ah} u_h \cdot \nabla_h \p^{\ah} \vro \ dx   -\i  \p^{\ah} u_h \cdot \nabla_h \p^{\ah} (\du + u \cdot \nabla \vro + \vro \, \du)\ dx \\
			= &\;  -  \p_t \i \p^{\ah} u_h \cdot \nabla_h \p^{\ah} \vro \ dx   +\i  \p^{\ah} \nabla_h \cdot  u_h \cdot \p^{\ah} (\du + u \cdot \nabla \vro + \vro \, \du)\ dx. 
			\dal \deqq
			On the other hand, we can obtain
			\begin{align*}
				&\; \i \p_t \p^{\ah} b \cdot \p^{\ah} \p_2 u\ dx \\ 
				= &\;  \p_t \i \p^{\ah} b \cdot \p^{\ah} \p_2 u\ dx  -\i  \p^{\ah} b \cdot \p^{\ah} \p_t \p_2 u \ dx \\
				= &\; -   \p_t \i \p^{\ah} \p_2 b \cdot \p^{\ah} u \ dx   -\i  \p^{\ah} b \cdot \p^{\ah}  \p_2 (F_2 + \p_1^2 u + \nabla \du + \p_2 b - \nabla b_2 -\nabla \vro) \ dx \\
				&\;  - \var \i  \p^{\ah} b \cdot \p^{\ah}  \p_2 (\p_2^2 u + \p_3^2 u) \ dx\\
				= &\;  -  \p_t \i \p^{\ah} \p_2 b \cdot \p^{\ah} u \ dx   +\i  \p^{\ah} \p_2 b \cdot \p^{\ah}   (F_2 + \p_1^2 u + \nabla \du + \p_2 b - \nabla b_2 -\nabla \vro) \ dx\\
				&\;  + \var \i  \p^{\ah} \p_2 b \cdot \p^{\ah} (\p_2^2 u + \p_3^2 u) \ dx.
			\end{align*}
			Thus, we have
			\begin{align}\label{equ-phu-vro}
				&\; \f{d}{dt} \i( \p^{\ah}  u_h \cdot \nabla_h \p^{\ah} \vro +  \p_{2} \p^{\ah} b \cdot  \p^{\ah} u )\ dx  + \|(\p^{\ah} \nabla_h  \vro, \p^{\ah} \p_{2} u)\|_{L^2}^2 \notag\\
				= &\;  \i  \p^{\ah}  \nabla_h \cdot   u_h \cdot \p^{\ah} (\du + u \cdot \nabla \vro + \vro \, \du)\ dx \notag \\
				&\; + \i \p^{\ah} (\p_1^2 u_h +\nabla_h \du + \p_2 b_h - \nabla_h b_2) \cdot \nabla_h  \p^{\ah} \vro \ dx \notag\\
				&\; + \i    \p^{\ah} (e_2  \du - \Delta_h b + u \cdot \nabla b + b\,\du- b \cdot \nabla u) \cdot \p^{\ah} \p_{2} u\ dx \\
				&\; +\i  \p^{\ah} ( \p_1^2 u + \nabla \du + \p_2 b - \nabla b_2 -\nabla \vro)  \cdot  \p^{\ah} \p_{2} b\ dx + \i \p^{\ah}  (F_2)_h  \cdot \nabla_h  \p^{\ah}\vro \ dx\notag\\
				&\; + \i \p^{\ah} F_2 \cdot \p^{\ah} \p_{2} b \ dx 
				+ \var \i \p^{\ah} (\p_2^2 u_h + \p_3^2 u_h) \cdot \nabla_h \p^{\ah} \vro \ dx\notag\\
				&\; -\var \i \p^{\ah} \p_3^2  b \cdot \p^{\ah} \p_2 u \ dx  + \var \i   \p^{\ah} (\p_2^2 u + \p_3^2 u) \cdot\p^{\ah} \p_2 b \ dx\notag\\
				:= &\;  \sum_{i=1}^{9} IV_{i}.\notag
			\end{align}
			Using the H\"older inequality and the condition ${\rm div} b=0$, we can directly obtain that
			\beqq \bal
			&\; IV_2 \le   \nu  \| \p^{\ah} \nabla_h \vro\|_{L^2}^2 + C_{\nu} \|\p^{\ah} (\p_1^2 u, \nabla_h \du, \nabla_h b)\|_{L^2}^2,\\
			&\; IV_4 =  \i  \p^{\ah} \nabla ( \du  -  b_2 - \vro)  \cdot \p^{\ah} \p_{2} b\ dx + \i  \p^{\ah} ( \p_1^2 u +\p_2 b)  \cdot  \p^{\ah} \p_{2} b\ dx \lesssim \| \p^{\ah} (\p_1^2  u,\p_2 b)\|_{L^2}^2,\\
			&\; IV_7 +IV_8 +IV_9 \le  \nu  \| \p^{\ah} \nabla_h \vro\|_{L^2}^2 + C_{\nu}( \var \|\p^{\ah} (\p_2^2 u, \p_3^2 u)\|_{L^2}^2 + \|\p^{\ah} \p_2 b\|_{L^2}^2).
			\dal \deqq
			As for the other terms, we first deal with the case $|\ah|=0$.  Using the  anisotropic type inequality \eqref{ie:Sobolev}, we can obtain
			\beqq \bal
			IV_1 \lesssim &\;\| \nabla_h u\|_{L^2}  (\| \du \|_{L^2} + \| \vro \|_{L^{\infty}}\| \du \|_{L^2}) + \| \nabla_h u\|_{L^2}^{\f12} \| \p_1 \nabla_h u\|_{L^2}^{\f12} \|u_h \|_{L^2}^{\f12} \|\p_1 u_h \|_{L^2}^{\f12}  \|\nabla_h \vro\|_{L^2}^{\f12}\|\p_3 \nabla_h \vro\|_{L^2}^{\f12} \\&\; +\| \nabla_h u\|_{L^2}^{\f12} \| \p_1 \nabla_h u\|_{L^2}^{\f12} \|u_3 \|_{L^2}^{\f12} \|\p_3 u_3 \|_{L^2}^{\f12}  \|\p_3  \vro\|_{L^2}^{\f12}\|\p_{23} \vro\|_{L^2}^{\f12},\\
			IV_3 \le &\;  \nu \| \p_2 u\|_{L^2}^2 +C_{\nu} \| (\du,\Delta_h b)\|_{L^2}^2 +C\| \p_2 u\|_{L^2} \| \du\|_{L^2} \| b\|_{L^{\infty}} \\
			&\; +C \| \p_2 u\|_{L^2}^{\f12} \| \p_{12} u\|_{L^2}^{\f12} (\|u_h\|_{L^2}^{\f12} \|\p_2 u_h\|_{L^2}^{\f12}  \| \nabla_h b\|_{L^2}^{\f12}\| \p_3 \nabla_h b\|_{L^2}^{\f12}  +\|u_3\|_{L^2}^{\f12} \|\p_3 u_3 \|_{L^2}^{\f12}  \| \p_3 b\|_{L^2}^{\f12}\| \p_{23} b\|_{L^2}^{\f12} \\
			&\; + \|b_h\|_{L^2}^{\f12} \|\p_2 b_h\|_{L^2}^{\f12}  \| \nabla_h u\|_{L^2}^{\f12}\| \p_3 \nabla_h u\|_{L^2}^{\f12}  +\|b_3\|_{L^2}^{\f12} \|\p_3 b_3 \|_{L^2}^{\f12}  \| \p_3 u\|_{L^2}^{\f12}\| \p_{23} u\|_{L^2}^{\f12}).
			\dal \deqq 
			Similarly, for the nonlinear term $IV_5$, we have
			\begin{align*}
				IV_5 = &\; - \i  (u \cdot \nabla u_h +\vro \nabla_h \vro ) \cdot \nabla_h  \vro\ dx - \var \i  \f{\vro}{1+\vro}(\p_2^2 u_h+ \p_3^2 u_h) \cdot \nabla_h   \vro \ dx \\
				&\;  - \i  \f{\vro}{1+\vro}(\p_1^2 u_h+\nabla_h \du-\nabla_h b_2 +\p_2 b_h)  \cdot  \nabla_h  \vro \ dx\\
				&\; + \i   \f{1}{1+\vro}(b\cdot \nabla b_h - \f12 \nabla_h |b|^2)  \cdot \nabla_h   \vro  \ dx  \\
				\lesssim &\; \sqrt{\me^m(t)} \md_{tan}^{k+1}(t).
			\end{align*}
			Similarly, we have
			\beqq
			IV_6  = \i   (F_2)_h \cdot  \p_2 b_h \ dx + \i   (F_2)_3 \cdot  \p_2 b_3 \ dx
			\lesssim \sqrt{\me^m(t)} \md_{tan}^{k+1}(t).
			\deqq
			Thus, we have 
			\beqq \bal
			&\f{d}{dt} \i(  u_h \cdot \nabla_h \vro +  \p_{2} b \cdot u )\ dx  + \|(\nabla_h  \vro, \p_{2} u)\|_{L^2}^2\\
			\lesssim &\var \|(\p_2^2 u, \p_3^2 u)\|_{L^2}^2 + 
			\|(\p_1 u, \du, \nabla_h b)\|_{H^{1}_{tan}}^2
			+\sqrt{\me^m(t)}\md_{tan}^{k+1}(t).
			\dal \deqq
			\textbf{Now let us deal with the case $|\alpha_h| = k$($k=m-1$ or $k=m-2$)}. 
			Using the  anisotropic type inequality \eqref{ie:Sobolev}, we have
			\begin{align*}
				IV_1 \lesssim &\; \|\p^{\ah}  \nabla_h  u \|_{L^2}^{\f12} \|\p_1 \p^{\ah}  \nabla_h  u \|_{L^2}^{\f12}  \sum_{0 \le \beta_h \le \ah } \|\p^{\beta_h} \vro \|^{\f12}  \|\p_3 \p^{\beta_h}  \vro\|_{L^2}^{\f12}
				\|\p^{\al_h-\beta_h} \du \|_{L^2}^{\f12} \|\p_2 \p^{\al_h-\beta_h} \du \|_{L^2}^{\f12} \\
				&\; + \|\p^{\ah}  \nabla_h  u \|_{L^2}^{\f12} \|\p_1 \p^{\ah}  \nabla_h  u \|_{L^2}^{\f12} (
				\sum_{0 \le \beta_h \le [\f{\alpha_h}{2}] } \|\p^{\beta_h} u\|_{H_{tan}^1}^{\f12} \|\p_3  \p^{\beta_h} u\|_{H_{tan}^1}^{\f12}
				\|\p^{\al_h-\beta_h}  \nabla \vro \|_{L^2} \\
				&\; + \sum_{ [\f{\alpha_h}{2}] < \beta_h \le  \ah}  \|\p^{\beta_h} u \|_{L^2}^{\f12} \|\p_3 \p^{\beta_h} u \|_{L^2}^{\f12}
				\|\p^{\al_h-\beta_h}  \nabla \vro \|_{L^2}^{\f12}
				\|\p_2 \p^{\al_h-\beta_h}  \nabla \vro \|_{L^2}^{\f12}) + \|\p^{\ah}  \nabla_h  u \|_{L^2} \|\p^{\ah} \du  \|_{L^2}  \\
				\le &\; \nu \|\p^{\ah} \p_2 u\|_{L^2}^2 
                + C_{\nu}  \|\p^{\ah} (\p_1 u, \du)\|_{L^2}^2 +C \sqrt{\me^m(t)} \md_{tan}^{k+1}(t).
			\end{align*}
			Similarly, we can obtain that
			\beqq
			IV_3 \le  \nu \|\p^{\ah} \p_2 u\|_{L^2}^2  + C_{\nu} \|\p^{\ah} (\Delta_h b,  \du)\|_{L^2}^2 +C \sqrt{\me^m(t)} \md_{tan}^{k+1}(t).
			\deqq
			It remains to estimate the last term $IV_5$.
			\beqq \bal
			IV_5 = &\; - \i \p^{\ah} (u \cdot \nabla u_h +\vro \nabla_h \vro ) \cdot \nabla_h \p^{\ah} \vro\ dx  - \var \i \p^{\ah}\left( \f{\vro}{1+\vro}(\p_2^2 u_h+ \p_3^2 u_h)\right) \cdot \nabla_h \p^{\ah} \vro \ dx\\ 
			&\;  - \i \p^{\ah} \left(\f{\vro}{1+\vro}(\p_1^2 u_h+\nabla_h \du-\nabla_h b_2 +\p_2 b_h)\right)  \cdot \nabla_h \p^{\ah} \vro\ dx\\
			&\; + \i \p^{\ah}\left( \f{1}{1+\vro}(b\cdot \nabla b_h - \f12 \nabla_h (|b|^2)\right)  \cdot \nabla_h \p^{\ah} \vro\ dx  \\
			:= &\; \sum_{i=1}^{4} IV_{5,i}.
			\dal \deqq
			Using the  anisotropic type inequality \eqref{ie:Sobolev}, we have 
			\beqq \bal
			&\;  -\i \p^{\ah} (u \cdot \nabla u_h ) \cdot \nabla_h \p^{\ah} \vro \ dx  \\
			= &\; - \sum_{0 \le \beta_h \le \ah} C_{\ah}^{\beta_h} \i  \p^{\beta_h} u_h  \cdot  \p^{\ah-\beta_h} \nabla u \cdot \nabla_h \p^{\ah} \vro \ dx \\
			\lesssim  &\; \| \nabla_h \p^{\ah} \vro\|_{L^2} (
			\sum_{0 \le \beta_h \le [\f{\alpha_h}{2}] } \|\p^{\beta_h} u\|_{H_{tan}^1}^{\f12} \|\p_3  \p^{\beta_h} u\|_{H_{tan}^1}^{\f12}
			\|\p^{\al_h-\beta_h}  \nabla u \|_{L^2}^{\f12} 
			\|\p^{\al_h-\beta_h} \p_1 \nabla u \|_{L^2}^{\f12} \\
			&\; + \sum_{ [\f{\alpha_h}{2}] < \beta_h \le  \ah}  \|\p^{\beta_h} u \|_{L^2}^{\f12} \|\p_3 \p^{\beta_h} u \|_{L^2}^{\f12}
			\|\p^{\al_h-\beta_h}  \nabla u \|_{H_{tan}^2}),
			\dal \deqq
			and 
			\begin{align*}
				&\; - \i \p^{\ah} (\vro \nabla_h \vro ) \cdot \nabla_h \p^{\ah} \vro \ dx \\
				= &\; - \sum_{0 < \beta_h \le \ah} C_{\ah}^{\beta_h} \i  \p^{\beta_h} \vro \cdot \p^{\ah-\beta_h} \nabla_h \vro  \cdot \nabla_h \p^{\ah} \vro\ dx  -  \i \vro |\p^{\ah} \nabla_h \vro|^2 \ dx  \\
				\lesssim  &\; \| \nabla_h \p^{\ah} \vro\|_{L^2}(
				\sum_{0 < \beta_h \le [\f{\alpha_h}{2}] } \|\p^{\beta_h} \vro\|_{H_{tan}^1}^{\f12} \|\p_3  \p^{\beta_h} \vro\|_{H_{tan}^1}^{\f12}
				\|\p^{\al_h-\beta_h}  \nabla_h \vro \|_{L^2}^{\f12} 
				\|\p^{\al_h-\beta_h} \p_1 \nabla_h \vro \|_{L^2}^{\f12} \\
				&\; + \sum_{ [\f{\alpha_h}{2}] < \beta_h \le  \ah}  \|\p^{\beta_h} \vro \|_{L^2}^{\f12} \|\p_3 \p^{\beta_h} \vro \|_{L^2}^{\f12}
				\|\p^{\al_h-\beta_h}  \nabla_h \vro \|_{H_{tan}^2}) + \| \vro\|_{L^{\infty}} \| \nabla_h  \p^{\ah} \vro\|_{L^2}^2,
			\end{align*}
			which yields that
			\beqq 
			IV_{5,1} \lesssim \sqrt{\me^m(t)} \md_{tan}^{k+1}(t).
			\deqq
			Similarly, under the assumption \eqref{assumption}, we can estimate the other terms as follows: 
			\beqq
			IV_{5,2}+IV_{5,3}+IV_{5,4} \lesssim \sqrt{\me^m(t)} \md_{tan}^{k+1}(t),
			\deqq
			which, together with the estimate of $IV_{5,1}$, yields that
			\beqq
			IV_{5} \lesssim \sqrt{\me^m(t)} \md_{tan}^{k+1}(t).
			\deqq
			Similarly, we can obtain the estimate of $IV_6$.
			\beqq
			IV_{6} = \i \p^{\ah}  (F_2)_h \cdot \p^{\ah}  \p_2 b_h \ dx + \i \p^{\ah}  (F_2)_3 \cdot \p^{\ah}  \p_2 b_3 \ dx  \lesssim \sqrt{\me^m(t)} \md_{tan}^{k+1}(t).
			\deqq
			Substituting the estimates from $IV_1$  to $IV_9$ into \eqref{equ-phu-vro}, and using the smallness of $\nu$, we have
			\beqq \bal
			&\; \f{d}{dt} \i( \p^{\ah}  u_h \cdot \nabla_h \p^{\ah} \vro +  \p_{2} \p^{\ah} b \cdot  \p^{\ah} u )\ dx  + \|(\p^{\ah} \nabla_h  \vro, \p^{\ah} \p_{2} u)\|_{L^2}^2 \\
			\lesssim &\; \var \|\p^{\ah} (\p_2^2 u, \p_3^2 u)\|_{L^2}^2 + \|\p^{\ah} (\p_1 u, \p_1^2 u) \|_{L^2}^2 + \|\p^{\ah} (\du, \nabla_h \du) \|_{L^2}^2 \\
			&\; +  \|\p^{\ah} (\nabla_h b, \Delta_h b) \|_{L^2}^2 + \sqrt{\me^m(t)} \md_{tan}^{k+1}(t).
			\dal \deqq
			Combining the above estimate and the estimate of  $|\ah|=0$, we finish the proof of this lemma. 
		\end{proof}
		Next, we will establish the dissipative estimates for vertical derivative of density in horizontal directions and velocity in the $x_2$ direction as follows.
		\begin{lemm}
			Under the assumption \eqref{assumption}, for any smooth solution $(\vro ,u, b)$ of equation \eqref{eqr}, it holds
			\beq \label{3501}
			\begin{aligned}
				&\;\f{d}{dt} \i (\p_3^{m-1} u_h \cdot \nabla_h \p_3^{m-1} \vro +    \p_{2} \p_3^{m-1} b \cdot  \p_3^{m-1} u )\ dx  + \|(\p_3^{m-1} \nabla_h  \vro, \p_3^{m-1} \p_{2} u)\|_{L^2}^2\\
				\lesssim & \; \var \| \p_3^{m-1} (\p_2^2 u, \p_3^2 u)\|_{L^2}^2 
				+ \| (\p_3^{m-1}  \p_1 u, \du, \nabla_h b)\|_{H_{tan}^1}^2  + \me^m(t)^{\f12}  \md^m(t) + \me^m(t)^{\f58}  \md^m(t)^{\f{1}{2}}  \md_{tan}^{m-1} (t)^{\f{3}{8}}.
			\end{aligned} \deq
		\end{lemm}
		\begin{proof}
			Similar to the equality \eqref{equ-phu-vro}, we can obtain that 
			\begin{align*}
				&\; \f{d}{dt} \i (\p_3^{m-1} u_h \cdot \nabla_h \p_3^{m-1} \vro +    \p_{2} \p_3^{m-1} b \cdot  \p_3^{m-1} u )\ dx  + \|(\p_3^{m-1} \nabla_h  \vro, \p_3^{m-1} \p_{2} u)\|_{L^2}^2 \\
				= &\;  \i \p_3^{m-1} \nabla_h \cdot   u_h \cdot\p_3^{m-1} (\du + u \cdot \nabla \vro + \vro \, \du)\ dx  \\
				&\; + \i \p_3^{m-1} (\p_1^2 u_h +\nabla_h \du + \p_2 b_h - \nabla_h b_2) \cdot \p_3^{m-1} \nabla_h \vro \ dx\\
				&\; + \i \p_3^{m-1} (e_2  \du - \Delta_h b + u \cdot \nabla b + b\,\du- b \cdot \nabla u)   \cdot  \p_3^{m-1} \p_{2} u \ dx \\
				&\; +\i  \p_3^{m-1} ( \p_1^2 u + \nabla \du + \p_2 b - \nabla b_2 -\nabla \vro)  \cdot  \p_3^{m-1} \p_{2} b\ dx + \i \p_3^{m-1} (F_2)_h  \cdot \nabla_h  \p_3^{m-1} \vro \ dx \\
				&\;+ \i
				\p_3^{m-1} F_2  \cdot \p_3^{m-1} \p_2 b \ dx + \var \i \p_3^{m-1} (\p_2^2 u_h + \p_3^2 u_h) \cdot \nabla_h \p_3^{m-1} \vro \ dx\\
				&\; -\var \i \p_3^{m+1}  b \cdot \p_3^{m-1} \p_2 u \ dx  + \var \i   \p_3^{m-1} (\p_2^2 u + \p_3^2 u)\cdot \p_3^{m-1} \p_2 b  \ dx\\
				:= &\;  \sum_{i=1}^{9} V_{i}.
			\end{align*}
			
			Estimate of the term $V_1$. It is easy to check that 
			\begin{align*}
				V_1 = &\;  -\i  \p_3^{m-1} \nabla_h \cdot   u_h \, \p_3^{m-1} \du \ dx - \sum_{0 \le l \le m-1} C_{m-1}^{l} \i  \p_3^{m-1} \nabla_h \cdot   u_h \, \p_3^{l} u_3 \cdot  \p_3^{m-l} \vro \ dx\\
				&\;-\sum_{0 \le l \le m-1} C_{m-1}^{l} \i  \p_3^{m-1}  \p_3 \nabla_h \cdot   u_h \, \p_3^{l} u_h \cdot  \p_3^{m-1-l} \nabla_h \vro \ dx\\
				&\;  - \sum_{0 \le l \le m-1} C_{m-1}^{l} \i  \p_3^{m-1}  \nabla_h \cdot   u_h \, \p_3^{l} \vro \cdot \p_3^{m-1-l} \du\ dx \\
				:= &\; \sum_{i=1}^{4}V_{1,i}.
			\end{align*}
			First, we deal with the term $V_{1,2}$. Using the  anisotropic type inequality \eqref{ie:Sobolev}, when $l=0$, we have
			\beqq \bal
			&\; \i  \p_3^{m-1} \nabla_h \cdot   u_h \,  u_3 \cdot  \p_3^{m}\vro \ dx \\
			\lesssim &\; \| \p_3^{m-1}  \nabla_h u\|_{L^2}^{\f12} \| \p_1 \p_3^{m-1}\nabla_h u\|_{L^2}^{\f12} \| \p_3^{m}\vro  \|_{L^2}
			\| u_3 \|_{L^2}^{\f14} \| \p_3 u_3 \|_{L^2}^{\f14} \|  \p_2 u_3\|_{L^2}^{\f14}
			\| \p_{23} u_3 \|_{L^2}^{\f14} \\
			\lesssim &\;\me^m(t)^{\f58}  \md^m(t)^{\f{1}{2}}  \md_{tan}^{m-1} (t)^{\f{3}{8}}.
			\dal \deqq
			For $1\le l \le m-1$, we have
			\beqq \bal
			&\; \i  \p_3^{m-1}   \nabla_h \cdot   u_h \, \p_3^{l} u_3 \cdot  \p_3^{m-l} \vro \ dx \\
			\lesssim &\; \| \p_3^{m-1}  \nabla_h u\|_{L^2}^{\f12} \|\p_1 \p_3^{m-1} \nabla_h u\|_{L^2}^{\f12} \| \p_3^{m-l}\vro  \|_{L^2}^{\f12}
			\| \p_2 \p_3^{m-l}\vro  \|_{L^2}^{\f12}
			\| \p_3^l u_3 \|_{L^2}^{\f12} \| \p_3^{l+1} u_3 \|_{L^2}^{\f12}  \\
			\lesssim &\;\me^m(t)^{\f12}  \md^m(t).
			\dal \deqq
			Similarly we have
			\beqq
			V_{1,2} + V_{1,3} +V_{1,4}  \lesssim \me^m(t)^{\f12}  \md^m(t).
			\deqq
			Thus, we have
			\beq \label{3502} \bal
			V_1 \le 
            &\; \nu\| \p_3^{m-1}  \p_2 u\|_{L^2}^2 
            + C_{\nu} \| (\p_3^{m-1}  \p_1 u, \p_3^{m-1}  \du)\|_{L^2}^2 \\
            &+ C\me^m(t)^{\f12}  \md^m(t) 
            + C\me^m(t)^{\f58}  \md^m(t)^{\f{1}{2}}  \md_{tan}^{m-1} (t)^{\f{3}{8}}.
			\dal \deq
			Using the condition ${\rm div} b=0$, it is easy to check that
			\beq\label{3503} \bal
			V_4 = &\;  \i  \p_3^{m-1}  \nabla ( \du  -  b_2 - \vro)  \cdot  \p_3^{m-1}  \p_{2} b\ dx +\i \p_3^{m-1}  ( \p_1^2 u + \p_2 b )  \cdot  \p_3^{m-1}  \p_{2} b\ dx \\
			\lesssim &\;   \| (\p_3^{m-1} \p_2 b , \p_3^{m-1}  \p_1^2 u)\|_{L^2}^2 ,
			\dal \deq
			and using the  anisotropic type inequality \eqref{ie:Sobolev} , we have
			\beq\label{3504} \bal
			V_2 \le &\;  \nu  \| \p_3^{m-1} \nabla_h \vro\|_{L^2}^2 +  C_{\nu} \| (\p_3^{m-1} \p_1^2 u, \p_3^{m-1}  \nabla_h \du, \p_3^{m-1}  \nabla_h b)\|_{L^2}^2,\\
			V_3 \le &\; \nu  \| \p_3^{m-1} \p_2 u\|_{L^2}^2 +  C_{\nu} \| (\p_3^{m-1} \du , \p_3^{m-1}  \Delta_h b)\|_{L^2}^2 + C \me^m(t)^{\f12}  \md^m(t).\\
			\dal \deq 
			Next, we estimate the term $V_5$.
			\begin{align*}
				V_5 = &\; - \i \p_3^{m-1} (u \cdot \nabla u_h +\vro \nabla_h \vro ) \cdot \nabla_h  \p_3^{m-1} \vro \ dx \\
				&\; - \var \i \p_3^{m-1} \left( \f{\vro}{1+\vro}(\p_2^2 u_h+ \p_3^2 u_h)\right) \cdot \nabla_h  \p_3^{m-1} \vro  \ dx\\ 
				&\;  - \i \p_3^{m-1} \left(\f{\vro}{1+\vro}(\p_1^2 u_h+\nabla_h \du-\nabla_h b_2 +\p_2 b_h)\right)  \cdot \nabla_h  \p_3^{m-1} \vro \ dx\\
				&\; + \i \p_3^{m-1} \left( \f{1}{1+\vro}(b\cdot \nabla b_h - \f12 \nabla_h |b|^2) \right)  \cdot \nabla_h  \p_3^{m-1} \vro \ dx  \\
				:= &\; \sum_{i=1}^{4} V_{5,i}.
			\end{align*}
			Using the  anisotropic type inequality \eqref{ie:Sobolev} , we have
			\begin{align*}
				V_{5,1} = &\; - \sum_{0 \le l \le m-1} C_{m-1}^{l} \i  \p_3^{l} u  \cdot  \p_3^{m-l-1} \nabla u \cdot \nabla_h  \p_3^{m-1} \vro \ dx \\
				&\; - \sum_{1 \le l \le m-1} C_{m-1}^{l} \i  \p_3^{l} \vro \cdot \p_3^{m-l-1} \nabla_h \vro  \cdot  \nabla_h  \p_3^{m-1} \vro\ dx   -  \i \vro | \nabla_h  \p_3^{m-1} \vro|^2 \ dx  \\
				\lesssim &\;  \Big( \sum_{0 \le l \le [\f{m-1}{2}]} (\| \p_3^{l} u\|_{L^2}^{\f14}  \| \p_3^{l+1} u \|_{L^2}^{\f14} 
				\| \p_3^{l} \p_2 u\|_{L^2}^{\f14}\| \p_3^{l+1} \p_2 u\|_{L^2}^{\f14}
				\| \p_3^{m-1-l} \nabla u \|_{L^2}^{\f12} 
				\| \p_1 \p_3^{m-1-l} \nabla u \|_{L^2}^{\f12} \\
				&\; + \sum_{1 \le l \le [\f{m-1}{2}]} \| \p_3^{l} \vro \|_{L^2}^{\f14}  \| \p_1 \p_3^{l}  \vro  \|_{L^2}^{\f14} 
				\| \p_3^{l} \p_2  \vro \|_{L^2}^{\f14}\| \p_3^{l} \p_{12}  \vro \|_{L^2}^{\f14}
				\| \p_3^{m-1-l} \nabla_h \vro \|_{L^2}^{\f12} 
				\|\p_3^{m-l} \nabla_h \vro \|_{L^2}^{\f12}\\
				&\; +\sum_{[\f{m-1}{2}] < l \le m-1 } (\| \p_3^{l} u\|_{L^2}^{\f14}  \| \p_3^{l+1} u \|_{L^2}^{\f14} 
				\| \p_3^{l} \p_1 u\|_{L^2}^{\f14}\| \p_3^{l+1} \p_1 u\|_{L^2}^{\f14}
				\| \p_3^{m-1-l} \nabla u \|_{L^2}^{\f12} 
				\| \p_2 \p_3^{m-1-l} \nabla u \|_{L^2}^{\f12} \\
				&\; +  \| \p_3^{l} \vro \|_{L^2}^{\f12}  \| \p_1 \p_3^{l}  \vro  \|_{L^2}^{\f12} 
				\| \p_3^{m-1-l} \nabla_h \vro \|_{L^2}^{\f14} 
				\|\p_3^{m-l} \nabla_h \vro \|_{L^2}^{\f14}
				\|\p_2 \p_3^{m-l} \nabla_h \vro \|_{L^2}^{\f14}
				\|\p_2 \p_3^{m-l} \nabla_h \vro \|_{L^2}^{\f14}
				) \Big)\\
				&\; \times \|\nabla_h  \p_3^{m-1} \vro\|_{L^2} + \| \nabla_h  \p_3^{m-1} \vro\|_{L^2}^2  \| \vro\|_{L^{\infty}} \\
				\lesssim &\; \me^m(t)^{\f12}  \md^m(t).  
			\end{align*}
			Similarly, under the assumption \eqref{assumption}, we can obtain that
			\beqq \bal
			V_{5,2} + V_{5,3} +V_{5,4}\ \lesssim  &\; \me^m(t)^{\f12}  \md^m(t),
			\dal \deqq
			which, together with the estimate of $V_{5,1}$, yields that
			\beq \label{3505} 
			V_5 \lesssim \me^m(t)^{\f12}\md^m(t).
			\deq
			Similarly, integrating by parts, we have
			\beq \label{3506}  \bal
			V_6 = &\; \i
			\p_3^{m-1} (F_2)_h  \cdot \p_3^{m-1} \p_2 b_h \ dx +  \i
			\p_3^{m-1} (F_2)_3  \cdot \p_3^{m-1} \p_2 b_3 \ dx \\
			= &\; \i
			\p_3^{m-1} (F_2)_h  \cdot \p_3^{m-1} \p_2 b_h \ dx -
			\i
			\p_3^{m-2} \p_2 (F_2)_3  \cdot \p_3^{m-1} (\p_1 b_1 + \p_2 b_2) \ dx\\
			\lesssim &\; \me^m(t)^{\f12}\md^m(t),
			\dal \deq
			and using the H\"older inequality, we can directly obtain that
			\beq\label{3507}  \bal
			&\; V_7 +V_8 +V_9 \le  \nu  \| ( \p_3^{m-1} \nabla_h \vro,  \p_3^{m-1} \p_2 b)\|_{L^2}^2 + C_{\nu} \var \| \p_3^{m-1} (\p_2^2 u, \p_3^2 u)\|_{L^2}^2.
			\dal \deq
			Combining of the estimates \eqref{3502}-\eqref{3507} and using the smallness of $\nu$, we can obtain that
			\begin{align*}
				&\;\f{d}{dt} \i (\p_3^{m-1} u_h \cdot \nabla_h \p_3^{m-1} \vro +    \p_{2} \p_3^{m-1} b \cdot  \p_3^{m-1} u )\ dx  + \|(\p_3^{m-1} \nabla_h  \vro, \p_3^{m-1} \p_{2} u)\|_{L^2}^2\\
				\lesssim & \; \var \| \p_3^{m-1} (\p_2^2 u, \p_3^2 u)\|_{L^2}^2 + \| \p_3^{m-1}  (\p_1 u, \p_1^2 u) \|_{L^2}^2 + \|\p_3^{m-1}  (\du , \nabla_h \du) \|_{L^2}^2  + \| \p_3^{m-1}  (\nabla_h b, \Delta_h b)\|_{L^2}^2  \\
				&\; + \me^m(t)^{\f12}  \md^m(t) + \me^m(t)^{\f58}  \md^m(t)^{\f{1}{2}}  \md_{tan}^{m-1} (t)^{\f{3}{8}}.
			\end{align*}
			Therefore, we finish the proof of this lemma.
		\end{proof}
		Finally, we will establish the dissipative estimates for one-order vertical derivative of density in the horizontal directions and velocity in the $x_2$ direction as follows.
		\begin{lemm}\label{lemma3.7}
			For any smooth solution $(\vro ,u, b)$ of equation \eqref{eqr}, it holds 
			\beq \label{3601}
			\begin{aligned}
				&\frac{d}{dt}
				\i  \Big( \nabla_h u_3 \cdot \nabla_h  \p_3 \vro +\p_{23} b \cdot  \p_3 u + \sum_{|\ah|=m-3} (\p^{\ah} \nabla_h u_3 \cdot \nabla_h \p^{\ah} \p_3 \vro + \p^{\ah}\p_{23} b \cdot  \p^{\ah}\p_3 u) \Big) \ dx \\
				&\; +\|(\p_3  \nabla_h \vro, \p_{23}u)\|_{H_{tan}^{m-3}}^2 \\
				\lesssim &\;  \var\| (\p_2^2 u, \p_3^2 u, \p_{23} u)\|_{H^{m-2}_{tan}}^2 +  \var\| \p_3^2 b\|_{H^{m-3}_{tan}}^2  +\|  (\nabla_h u, \p_3 \du, \p_1 \nabla u,\nabla_h \nabla b)\|_{H^{m-2}_{tan}}^2 
				+\sqrt{\me^m(t)}\md_{tan}^{m-1}(t).
			\end{aligned}
			\deq
		\end{lemm}
		\begin{proof}
			For  $|\ah|\le m-3$, using the equation  \eqref{ieq-nablavro} of $\p_3 \vro$, we have
			\begin{align*}
				&\; \|\p^{\ah} \nabla_h \p_3 \vro\|_{L^2}^2  \\
				= &\; -\i \p_t \p^{\ah} \nabla_h u_3 \cdot \p^{\ah}  \nabla_h \p_3 \vro \ dx + \i \p^{\ah} \nabla_h ((F_2)_3+ \p_3 \du + \p_2 b_3 - \p_3 b_2) \cdot \p^{\ah} \nabla_h  \p_3 \vro \ dx \\
				&\; + \i \p^{\ah} \nabla_h (\p_1^2 u_3 + \var \p_2^2 u_3 + \var \p_3^2 u_3) \cdot \p^{\ah} \nabla_h  \p_3 \vro \ dx \\
				= &\; -\p_t \i  \p^{\ah} \nabla_h u_3 \cdot \p^{\ah}  \nabla_h \p_3 \vro \ dx + \i  \p^{\ah} \Delta_h u_3 \cdot \p^{\ah}  \p_3 (\du + \vro \du + u\cdot \nabla \vro) \ dx\\
				&\; + \i \p^{\ah} \nabla_h ((F_2)_3+ \p_3 \du +\p_1^2 u_3+ \p_2 b_3 - \p_3 b_2+ \var \p_2^2 u_3 + \var \p_3^2 u_3) \cdot \p^{\ah} \nabla_h  \p_3 \vro \ dx,
			\end{align*}
			which, together with the equation $\eqref{eqr}_2$ and $\eqref{eqr}_3$, yields that
			\beqq \bal
			&\; \f{d}{dt} \i( \p^{\ah} \nabla_h u_3 \cdot \nabla_h \p^{\ah} \p_3 \vro +  \p_{23} \p^{\ah} b \cdot  \p^{\ah} \p_3 u )\ dx  + \|(\p^{\ah} \p_3 \nabla_h  \vro, \p^{\ah} \p_{23} u)\|_{L^2}^2 \\
			= &\;  \i  \p^{\ah} \Delta_h u_3 \cdot \p^{\ah}\p_3 (\du + u \cdot \nabla \vro + \vro \, \du)\ dx  \\
			&\; + \i \p^{\ah} \nabla_h (\p_1^2 u_3 +\p_3 \du + \p_2 b_3 - \p_3 b_2) \cdot \nabla_h \p_3 \p^{\ah} \vro \ dx\\
			&\; + \i  \p^{\ah}\p_3 (e_2  \du - \Delta_h b + u \cdot \nabla b + b\,\du- b \cdot \nabla u) \cdot \p^{\ah} \p_{23} u  \ dx \\
			&\; +\i  \p^{\ah} \p_3 ( \p_1^2 u + \nabla \du + \p_2 b - \nabla b_2 -\nabla \vro)  \cdot  \p^{\ah} \p_{23} b\ dx + \i \p^{\ah} \nabla_h (F_2)_3  \cdot \nabla_h  \p_3 \p^{\ah}\vro\ dx \\
			&\; + \i \p^{\ah} \p_3 F_2  \cdot 
			\p^{\ah} \p_{23} b \ dx + \var \i \p^{\ah} \nabla_h (\p_2^2 u_3 + \p_3^2 u_3) \cdot \nabla_h \p^{\ah} \p_3 \vro \ dx\\
			&\; -\var \i \p^{\ah} \p_3^3  b \cdot \p^{\ah} \p_{23} u \ dx  + \var \i   \p^{\ah}\p_3 (\p_2^2 u + \p_3^2 u) \cdot \p^{\ah} \p_{23} b  \ dx\\
			:= &\;  \sum_{i=1}^{9} VI_{i}.
			\dal \deqq
			Using the H\"older inequality and the condition ${\rm div} b=0$, we can directly obtain that
			\beqq \bal
			VI_2 \le &\;  \nu  \| \nabla_h \p^{\ah} \p_3   \vro\|_{L^2}^2 + C_{\nu} \|\p^{\ah} \nabla_h( \p_1^2 u, \p_3 \du, \p_2 b, \p_3 b)\|_{L^2}^2,\\
			VI_4 = &\; \i  \p^{\ah} \p_3 \nabla ( \du  -  b_2 - \vro)  \cdot \p^{\ah} \p_{23} b\ dx + \i  \p^{\ah}\p_3 ( \p_1^2 u +\p_2 b)  \cdot  \p^{\ah} \p_{23} b\ dx \lesssim \| \p^{\ah} \p_3 (\p_1^2  u,\p_2 b)\|_{L^2}^2,
			\dal \deqq
			and 
			\beqq \bal
			VI_7 \lesssim &\;  \nu  \| \nabla_h \p^{\ah} \p_3   \vro\|_{L^2}^2 + \var\| \p^{\ah}\nabla_h (\p_2^2 u, \p_3^2 u)\|_{L^2}^2,\\
			VI_8 + VI_9 \lesssim &\; \var \|  \p^{\ah}  (\p_2^2 \p_3 u, \p_2 \p_3^2 u) \|_{L^2}^2 + \var \|  \p^{\ah}  ( \p_{23} b,\p_{3}^2 b) \|_{L^2}^2.
			\dal \deqq
			For the nonlinear term $VI_5$ and $VI_6$, we decompose them as follows:
			\beqq \bal
			VI_5 = &\; - \i \p^{\ah} \nabla_h (u \cdot \nabla u_3 +\vro \p_3 \vro ) \cdot \nabla_h \p^{\ah} \p_3 \vro \ dx \\
			&\; - \var \i \p^{\ah} \nabla_h (\f{\vro}{1+\vro}(\p_2^2 u_3+\p_3^2 u_3)) \cdot \nabla_h \p^{\ah} \p_3 \vro\ dx \\
			&\;  - \i \p^{\ah} \nabla_h (\f{\vro}{1+\vro}(\p_1^2 u_3+\p_3 \du-\p_3 b_2 +\p_2 b_3))  \cdot\nabla_h \p^{\ah} \p_3 \vro \ dx\\
			&\; + \i \p^{\ah} \nabla_h \left( \f{1}{1+\vro}(b\cdot \nabla b_3 - b\, \p_3 b)\right)  \cdot \nabla_h \p^{\ah} \p_3 \vro\ dx  \\
			:= &\; \sum_{i=1}^{4} VI_{5,i},
			\dal \deqq
			and
			\begin{align*}
				VI_6 = &\; - \i \p^{\ah} \p_3 (u \cdot \nabla u + \nabla (\f12 |\vro|^2 )) \cdot \p^{\ah} \p_{23} b\ dx   - \var \i \p^{\ah} \p_3 (\f{\vro}{1+\vro}(\p_2^2 u+\p_3^2 u)) \cdot \p^{\ah} \p_{23} b\ dx \\
				&\;  - \i \p^{\ah} \p_3 (\f{\vro}{1+\vro}(\p_1^2 u+ \nabla \du-\nabla b_2 +\p_2 b))  \cdot \p^{\ah} \p_{23} b \ dx\\
				&\; + \i \p^{\ah} \p_3 \left( \f{1}{1+\vro}(b\cdot \nabla b - \f12 \nabla( |b|^2))\right)  \cdot \p^{\ah} \p_{23} b\ dx  \\
				:= &\; \sum_{i=1}^{4} VI_{6,i}.
			\end{align*}
			When $|\ah|= 0$, we first estimate the term $VI_5$.
			\begin{align*}
				VI_{5,1}= &\; - \i  (\nabla_h u \cdot \nabla u_3 + u \cdot \nabla \nabla_h  u_3 + \p_3 \vro \nabla_h \vro + \vro \nabla_h \p_3 \vro  ) \cdot \p_3 \nabla_h   \vro \ dx \\
				\lesssim &\;  \|\p_3 \nabla_h   \vro \|_{L^2}(\| \nabla u_3 \|_{L^2}^{\f14} \| \p_{1} \nabla u_3 \|_{L^2}^{\f14} \| \p_{2} \nabla u_3\|_{L^2}^{\f14} \| \p_{12} \nabla u_3\|_{L^2}^{\f14} \| \nabla_h u\|_{L^2}^{\f12} \| \p_3 \nabla_h u\|_{L^2}^{\f12} \\
				&\; 
				+ \| u \|_{L^{\infty}} \| \nabla \nabla_h u\|_{L^2}  + \|  \vro \|_{L^{\infty}} \| \p_3 \nabla_h \vro\|_{L^2}
				+ \| \p_3 \vro \|_{L^2}^{\f12} \| \p_{23} \vro\|_{L^2}^{\f12}  \| \nabla_h \vro \|_{H_{tan}^1}^{\f12}\| \p_3 \nabla_h \vro \|_{H_{tan}^1}^{\f12})\\
				\lesssim &\; \sqrt{\me^m(t)}\md_{tan}^{m-1}(t).
			\end{align*}
			As for the term $VI_{5,4}$, we have
			\beqq \bal
			VI_{5,4} = &\; \i  \left( \nabla_h (\f{1}{1+\vro}) (b\cdot \nabla b_3 - b \,\p_3 b) + \f{1}{1+\vro} (\nabla_h b \cdot \nabla b_3  +b\cdot \nabla_h \nabla  b_h - \nabla_h (b\, \p_3 b) ) \right)  \cdot \nabla_h \p_3 \vro \ dx \\
			\lesssim &\; \|\p_3 \nabla_h   \vro\|_{L^2} ( \| b\|_{L^{\infty}} \| \nabla b \|_{L^2}^{\f14} \| \p_{1} \nabla b \|_{L^2}^{\f14} \| \p_2 \nabla b \|_{L^2}^{\f14} \| \p_{12} \nabla b\|_{L^2}^{\f14} \| \nabla_h (\f{1}{1+\vro})\|_{L^2}^{\f12} \| \p_{3} \nabla_h (\f{1}{1+\vro})\|_{L^2}^{\f12} \\
			&\; +  \| \nabla_h b \|_{L^2}^{\f12} \| \p_{3} \nabla_h b  \|_{L^2}^{\f12}\| \nabla b \|_{L^2}^{\f14} \|  \p_{1} \nabla b\|_{L^2}^{\f14} \| \p_2 \nabla b \|_{L^2}^{\f14} 
			\|  \p_{12}  \nabla b\|_{L^2}^{\f14} + \| b\|_{L^{\infty}}  \|\nabla \nabla_h  b \|_{L^2})\\
			\lesssim &\; \sqrt{\me^m(t)}\md_{tan}^{m-1}(t).
			\dal \deqq
			Similarly, we have
			\beqq
			VI_{5,2} +VI_{5,3} \lesssim \sqrt{\me^m(t)}\md_{tan}^{m-1}(t).
			\deqq
			Thus we can obtain that
			\beqq
			VI_{5} \lesssim \sqrt{\me^m(t)}\md_{tan}^{m-1}(t).
			\deqq
			Now we estimate the term $VI_6$. 
			Using the condition  ${\rm div}b=0$, we have
			\beqq \bal
			VI_{6,1}= &\; - \i  (\p_3 u \cdot \nabla u + u \cdot \nabla \p_3  u ) \cdot  \p_{23} b \ dx + \f12 \i \p_3(|\vro|^2) \, \p_{23} {\rm div} b \ dx\\
			\lesssim &\;  \|\p_{23} b \|_{L^2} (\| \p_3 u_h \|_{L^2}^{\f12} \| \p_{13} u_h \|_{L^2}^{\f12} \| \nabla_h u \|_{L^2}^{\f14} \| \p_{3} \nabla_h u \|_{L^2}^{\f14} \| \p_2 \nabla_h u\|_{L^2}^{\f14} \| \p_{23} \nabla_h u\|_{L^2}^{\f14} \\
			&\; + \| \p_3 u_3 \|_{L^2}^{\f12} \| \p_{13} u_3 \|_{L^2}^{\f12} \| \p_3 u \|_{L^2}^{\f14} \| \p_{13} u \|_{L^2}^{\f14}  \| \p_{23} u \|_{L^2}^{\f14}  \| \p_{123} u\|_{L^2}^{\f12} 
			+ \| u_h \|_{L^{\infty}} \| \p_3 \nabla_h u\|_{L^2} \\
			&\; + \| u_3 \|_{L^2}^{\f14} \| \p_2 u_3 \|_{L^2}^{\f14}\| \p_3 u_3 \|_{L^2}^{\f14} \| \p_{23} u_3 \|_{L^2}^{\f14}
			\| \p_3^2  u\|_{L^2}^{\f12} \| \p_1 \p_3^2  u\|_{L^2}^{\f12}) \\
			\lesssim &\; \sqrt{\me^m(t)}\md_{tan}^{m-1}(t)
			+ \|\p_{23} b \|_{L^2}\| u_3 \|_{L^2}^{\f14} \| \p_2 u_3 \|_{L^2}^{\f14}\| \p_3 u_3 \|_{L^2}^{\f14} \| \p_{23} u_3 \|_{L^2}^{\f14}
			\| \p_3^2  u\|_{L^2}^{\f12} \| \p_1 \p_3^2  u\|_{L^2}^{\f12}.
			\dal \deqq
			Using the interpolation inequality and $m \ge 4$,  we have
			\beqq
			\| \p_1 \p_3^2  u\|_{L^2}^{\f12} \lesssim \| \p_1 \p_3 u\|_{L^2}^{\f14} \| \p_1 \p_3^3 u\|_{L^2}^{\f14},
			\deqq
			which yields that
			\beqq \bal
			&\; \|\p_{23} b\|_{L^2}\| u_3 \|_{L^2}^{\f14} \| \p_2 u_3 \|_{L^2}^{\f14}\| \p_3 u_3 \|_{L^2}^{\f14} \| \p_{23} u_3 \|_{L^2}^{\f14}
			\| \p_3^2  u\|_{L^2}^{\f12} \| \p_1 \p_3^2  u\|_{L^2}^{\f12} \\
			\lesssim &\; \|\p_{23} b \|_{L^2}\| u_3 \|_{L^2}^{\f14} \| \p_2 u_3 \|_{L^2}^{\f14}\| \p_3 u_3 \|_{L^2}^{\f14} \| \p_{23} u_3 \|_{L^2}^{\f14}
			\| \p_3^2  u\|_{L^2}^{\f12} \| \p_1 \p_3 u\|_{L^2}^{\f14} \| \p_1 \p_3^3 u\|_{L^2}^{\f14} \\
			\lesssim &\; \sqrt{\me^m(t)}\md_{tan}^{m-1}(t) .
			\dal \deqq
			Thus we have
			\beqq
			VI_{6,1} \lesssim  \sqrt{\me^m(t)}\md_{tan}^{m-1}(t).
			\deqq
			As for the term $VI_{6,4}$, we have
			\beqq \bal
			VI_{6,4} = &\;   - \i  \p_3(\f{1}{1+\vro} b \cdot \nabla b) \cdot \p_{23 }b \ dx - \f12 \i \p_3 \nabla(\f{1}{1+\vro} |b|^2 )  \cdot \p_{23 }b \ dx \\
			&\; + \f12 \i \p_3 (\nabla(\f{1}{1+\vro}) |b|^2 )  \cdot \p_{23 }b \ dx \\
			= &\;   - \i  \p_3( \f{1}{1+\vro}) b \cdot \nabla b \cdot \p_{23 }b \ dx - \i \f{1}{1+\vro}  \p_3(b \cdot \nabla b) \cdot \p_{23 }b \ dx \\
			&\; + \f12 \i |b|^2 \p_{23 }b  \cdot \p_3 \nabla(\f{1}{1+\vro})   \ dx  + \i b\, \p_3 b\, \p_{23 }b  \cdot \nabla (\f{1}{1+\vro})   \ dx.  \\
			\dal \deqq
			We can check that
			\beqq  \bal
			&\; VI_{6,4}  \\
			\lesssim &\; \|\p_{23 }b \|_{L^2} ( \| b \|_{L^{\infty}} \| \p_3 b \|_{L^2}^{\f14} \| \p_{13} b \|_{L^2}^{\f14} \| \p_{3}^2 b \|_{L^2}^{\f14} \| \p_{1} \p_3^2 b \|_{L^2}^{\f14} \| \p_3 (\f{1}{1+\vro})\|_{L^2}^{\f12} \| \p_{23} (\f{1}{1+\vro})\|_{L^2}^{\f12} \\
			&\; + \| b_h \|_{L^{\infty}} \| \nabla_h b\|_{L^2}^{\f14} \| \p_{3} \nabla_h b \|_{L^2}^{\f14} \| \p_{1} \nabla_h b \|_{L^2}^{\f14} \| \p_{13}  \nabla_h b \|_{L^2}^{\f14} \| \p_3 (\f{1}{1+\vro})\|_{L^2}^{\f12} \| \p_{23} (\f{1}{1+\vro})\|_{L^2}^{\f12}
			\\
			&\; + \| b\|_{L^{\infty}}  \|\p_3 \nabla_h  b \|_{L^2}
			+ \| b_3 \|_{L^2}^{\f14} \| \p_{3} b_3\|_{L^2}^{\f14} \| \p_{2} b_3\|_{L^2}^{\f14} \| \p_{23} b_3\|_{L^2}^{\f14} \| \p_3^2 b \|_{L^2}^{\f12} \| \p_1\p_{3}^2 b\|_{L^2}^{\f12}\\
			&\; + \| \p_3 b\|_{L^{\infty}} \| \nabla_h b_3\|_{L^2} 
			+  \| \p_{3} b\|_{H^2}\| \p_3 b_3 \|_{L^2}^{\f12} \| \p_{3}^2 b_3\|_{L^2}^{\f12} + \| b\|_{L^{\infty}}^2 \|\p_3 \nabla_h  (\f{1}{1+\vro})\|_{L^2} \\
			&\;+  \| b\|_{L^{\infty}} \| \p_3 b\|_{L^{\infty}} \| \nabla_h  (\f{1}{1+\vro})\|_{L^2})
			+ \| \p_{23} b_3 \|_{L^2}^{\f12} \| \p_2 \p_3^2 b_3 \|_{L^2}^{\f12} \| b\|_{L^{\infty}} \| \p_3^2 (\f{1}{1+\vro}) \|_{L^2}^{\f12} \| \p_1 \p_3^2 (\f{1}{1+\vro}) \|_{L^2}^{\f12} \|b\|_{L^2}^{\f12} \|\p_2 b\|_{L^2}^{\f12}\\
			\lesssim &\; \sqrt{\me^m(t)}\md_{tan}^{m-1}(t)  +  \|b\|_{L^{\infty}} (\md_{tan}^{m-1}(t)^{\f34}\me^m(t)^{\f34} + \md_{tan}^{m-1}(t)^{\f78} \me^m(t)^{\f58})+ \|\p_1 \p_3^2 b\|_{L^2}^{\f12} \md_{tan}^{m-1}(t)^{\f78}  \me^m(t)^{\f38} . \\
			\dal \deqq
			Using the inequality \eqref{ie:Sobolev} and the interpolation inequality, we can obtain that
			\beqq \bal
			\|b\|_{L^{\infty}} \lesssim &\; \| (b, \p_3 b) \|_{L^2}^{\f14}  \|\nabla_h(b, \p_3 b)\|_{H_{tan}^1}^{\f34} \lesssim  \md_{tan}^{m-1}(t)^{\f38} \me^m(t)^{\f18},\\
			\|\p_1 \p_3^2 b\|_{L^2}^{\f12} \lesssim &\;  \|\p_1 \p_3 b\|_{L^2}^{\f14}  \|\p_1 \p_3^3 b\|_{L^2}^{\f14} \lesssim  \md_{tan}^{m-1}(t)^{\f18} \me^m(t)^{\f18},
			\dal \deqq 
			which yields that
			\beqq
			VI_{6,4} \lesssim  \sqrt{\me^m(t)}\md_{tan}^{m-1}(t).
			\deqq
			Similarly, we have
			\beqq \bal
			VI_{6,2} =&\; \var \i ( \p_3 (\f{1}{1+\vro}) \, \p_2^2u+ \f{1}{1+\vro} \p_2^2 \p_3 u )\cdot \p_{23 }b  - \var \i \f{1}{1+\vro}  \p_3^2 u \cdot \p_{2} \p_3^2 b \ dx \\
			\lesssim &\; \sqrt{\me^m(t)}\md_{tan}^{m-1}(t),\\
			VI_{6,3} \lesssim &\; \sqrt{\me^m(t)}\md_{tan}^{m-1}(t).
			\dal \deqq
			Thus we can obtain that
			\beqq
			VI_{6} \lesssim \sqrt{\me^m(t)}\md_{tan}^{m-1}(t).
			\deqq
			Finally, we estimate the term $VI_1$.
			\begin{align*}
				VI_1= &\; - \i \Delta_h u_3 \cdot (\p_3 \du + \p_3 u \cdot \nabla \vro + u \cdot \nabla  \p_3 \vro +  \p_3  \vro \, \du +   \vro \, \p_3 \du)\ dx \\
				\lesssim &\; \|(\Delta_h u_3,\p_3 \du)  \|_{L^2}^2 + \|\Delta_h u_3 \|_{L^2}^{\f12}  \|\p_{1} \Delta_h u_3 \|_{L^2}^{\f12}  (\| \p_3 u_h \|_{L^2}^{\f12} \| \p_{23} u_h \|_{L^2}^{\f12} \| \nabla_h \vro\|_{L^2}^{\f12} \| \p_3 \nabla_h \vro\|_{L^2}^{\f12} \\
				&\; + \| \p_3 u_3 \|_{L^2}^{\f12} \| \p_3^2 u_3 \|_{L^2}^{\f12} \| \p_3 \vro\|_{L^2}^{\f12} \| \p_{23} \vro\|_{L^2}^{\f12} 
				+ \| u_h \|_{H_{tan}^1}^{\f12} \| \p_3 u_h \|_{H_{tan}^1}^{\f12} \| \p_3 \nabla_h \vro\|_{L^2} \\
				&\; + \| u_3 \|_{L^2}^{\f14} \| \p_2 u_3 \|_{L^2}^{\f14}\| \p_3 u_3 \|_{L^2}^{\f14} \| \p_{23} u_3 \|_{L^2}^{\f14}
				\| \p_3^2  \vro\|_{L^2} 
				+ \| \p_3 \vro \|_{L^2}^{\f12} \| \p_{23} \vro\|_{L^2}^{\f12} \| \du\|_{L^2}^{\f12} \| \p_{3} \du\|_{L^2}^{\f12} \\
				&\; + \|  \vro \|_{H_{tan}^1}^{\f12} \| \p_{3} \vro\|_{H_{tan}^1}^{\f12} \| \p_3 \du\|_{L^2})\\
				\lesssim &\;  \|(\Delta_h u_3,\p_3 \du)  \|_{L^2}^2 + \sqrt{\me^m(t)}\md_{tan}^{m-1}(t),
			\end{align*} 
			where we have used the following estimate 
			\beqq \bal
			\| \p_{3}^2 \vro\|_{L^2} \lesssim \| \p_3 \vro\|_{L^2}^{\f12}  \| \p_3^3 \vro\|_{L^2}^{\f12} \lesssim \|\hs  \p_3 \vro\|_{L^2}^{\frac{1}{2(1+s)}}
			\|\nabla_h  \p_3 \vro\|_{L^2}^{\frac{s}{2(1+s)}} \| \p_3^3 \vro\|_{L^2}^{\f12},
			\dal \deqq
			and $s > \f{9}{10}$.
			Similarly, we can estimate the term $VI_3$ as follows.
			\beqq
			VI_{3} \le \nu \| \p_{23} u \|_{L^2}^2 + C_{\nu}\| (\p_3 \du, \p_3 \Delta_h b) \|_{L^2}^2 +  \sqrt{\me^m(t)}\md_{tan}^{m-1}(t).
			\deqq
			Combining the estimates from $VI_1$ to $VI_6$, and using the smallness of $\nu$, for $|\ah|=0$, we have
			\beqq \bal
			&\; \f{d}{dt} \i( \nabla_h u_3 \cdot \nabla_h \p_3 \vro +  \p_{23} b \cdot   \p_3 u )\ dx  + \|( \p_3 \nabla_h  \vro,  \p_{23} u)\|_{L^2}^2 \\
			\lesssim &\;  \var\| \nabla_h (\p_2^2 u, \p_3^2 u, \p_{23} u)\|_{L^2}^2  + \var \|\p_{3}^2 b \|_{L^2}^2  +   \|(\Delta_h u,\nabla \p_1^2 u)  \|_{L^2}^2 + \|(\p_3 \du, \p_3 \nabla_h \du)\|_{L^2}^2\\
			&\; +  \| (\p_3 \nabla_h b ,\nabla_h^2 b ,\p_3 \Delta_h b) \|_{L^2}^2+\sqrt{\me^m(t)}\md_{tan}^{m-1}(t).
			\dal \deqq
			\textbf{Now, we deal with the case $|\alpha_h| =m-3$}. 
			First, we deal with the term $VI_1$. 
			\beqq \bal
			VI_1= &\; - \i \p^{\ah} \Delta_h  u_3 \cdot \p^{\ah} \p_3 \du \ dx  - \i \p^{\ah} \Delta_h  u_3 \cdot \p^{\ah} ( \p_3 u \cdot \nabla \vro )\ dx  \\
			&\; - \i \p^{\ah} \Delta_h  u_3 \cdot \p^{\ah}  (u \cdot \nabla  \p_3 \vro )\ dx
			- \i \p^{\ah} \Delta_h  u_3 \cdot \p^{\ah} (  \p_3  \vro \, \du +   \vro \, \p_3 \du)\ dx \\
			: = &\; \sum_{i=1}^{4} VI_{1,i}.
			\dal \deqq
			It is easy to check that
			\beqq
			VI_{1,1} \lesssim \|(\p^{\ah} \Delta_h  u_3, \p^{\ah} \p_3 \du)\|_{L^2}^2,
			\deqq
			and 
			\beqq \bal
			VI_{1,2} \lesssim &\; \| \p^{\ah} \Delta_h  u_3\|_{L^2}^{\f12} \| \p_1 \p^{\ah} \Delta_h  u_3\|_{L^2}^{\f12} \sum_{0 \le \beta_h \le \alpha_h} \|\p^{\beta_h} \p_3 u \|_{L^2}^{\f12} \|\p_{23} \p^{\beta_h} u \|_{L^2}^{\f12}
			\|\p^{\al_h-\beta_h}  \nabla \vro \|_{L^2}^{\f12} \|\nabla \p^{\al_h-\beta_h} \p_3  \vro \|_{L^2}^{\f12} .
			\dal \deqq
			As for the term $VI_{1,3}$, using the  anisotropic type inequality \eqref{ie:Sobolev}, we can arrive 
			\beqq  \bal
			VI_{1,3}= &\; - \i \p^{\ah} \Delta_h  u_3 \cdot  (u_h \cdot  \nabla_h \p^{\ah}   \p_3\vro + u_3 \, \p^{\ah}  \p_3^2 \vro )\ dx  - \sum_{0 < \beta_h \le \alpha_h } C_{\ah}^{\beta_h} \i  \p^{\beta_h}  u \cdot \nabla  \p^{\ah-\beta_h}  \p_3 \vro \,  \p^{\ah} \Delta_h  u_3\ dx  \\ 
			\lesssim &\;  \| \p^{\ah} \Delta_h  u_3\|_{L^2}^{\f12} \|  \p_1 \p^{\ah} \Delta_h  u_3\|_{L^2}^{\f12} \Big(\| u_h\|_{H_{tan}^1}^{\f12}  \| \p_3 u_h\|_{H_{tan}^1}^{\f12} \| \p^{\ah} \nabla_h  \p_3\vro\|_{L^2} \\
			&\; + \| u_3\|_{L^2}^{\f14} \| \p_3 u_3\|_{L^2}^{\f14}\| \p_2 u_3\|_{L^2}^{\f14} \| \p_{23} u_3\|_{L^2}^{\f14}  \| \p^{\ah}  \p_3^2 \vro\|_{L^2}   \\
			&\; +
			\sum_{0 < \beta_h \le \ah } \|\p^{\beta_h}  u\|_{L^2}^{\f12} \|\p_3  \p^{\beta_h} u\|_{L^2}^{\f12}
			\|\p^{\al_h-\beta_h}  \nabla \p_3 \vro \|_{L^2}^{\f12} \|\p^{\al_h-\beta_h}  \nabla \p_3^2 \vro \|_{L^2}^{\f12}\Big) \\
			\lesssim &\; \sqrt{\me^m(t)}\md_{tan}^{m-1}(t),
			\dal \deqq
			where we have used the interpolation inequality
			\beqq
			\| \p^{\ah}  \p_3^2 \vro\|_{L^2} \lesssim   \| \p^{\ah}  \p_3 \vro\|_{L^2}^{\f12}  \| \p^{\ah}  \p_3^3 \vro\|_{L^2}^{\f12} \lesssim \me^m(t)^{\f14} \md_{tan}^{m-1}(t)^{\f14}.
			\deqq
			Similarly, we have
			\beqq
			VI_{1,4} \lesssim \sqrt{\me^m(t)}\md_{tan}^{m-1}(t),
			\deqq
			which yields that
			\beqq
			VI_{1} \lesssim \sqrt{\me^m(t)}\md_{tan}^{m-1}(t).
			\deqq
			Next, we deal  with the nonlinear term $VI_6$. Integrating by parts, we have
			\beqq \bal
			VI_{6,1} = &\; - \i \p^{\ah}  (\p_3 u \cdot \nabla u ) \cdot  \p^{\ah} \p_{23 }b \ dx - \i \p^{\ah}  ( u \cdot \nabla \p_3  u ) \cdot \p^{\ah} \p_{23 }b\ dx
			+ \f12 \i \p^{\ah} \p_3 |\vro|^2 \cdot \p^{\ah} \p_{23 } {\rm div} b\ dx\\
			= &\; - \i \p^{\ah}  (\p_3 u \cdot \nabla u ) \cdot  \p^{\ah} \p_{23 }b \ dx - \i \p^{\ah}  ( u \cdot \nabla \p_3  u ) \cdot \p^{\ah} \p_{23 }b\ dx. \\
			\dal \deqq
			Using the  anisotropic type inequality \eqref{ie:Sobolev}, we can arrive 
			\beqq \bal
			&\; \i \p^{\ah}  (\p_3 u \cdot \nabla u ) \cdot \p^{\ah} \p_{23 }b  \ dx \\
			\lesssim &\; \| \p^{\ah} \p_{23 }b\|_{L^2} \Big(\sum_{0 \le \beta_h \le [\f{\alpha_h}{2}] } (\|\p^{\beta_h} \p_3 u_h\|_{H_{tan}^1}^{\f12} \|\p^{\beta_h} \p_3^2 u_h\|_{H_{tan}^1}^{\f12}
			\|\p^{\al_h-\beta_h}  \nabla_h u \|_{L^2}^{\f12} \|\p^{\al_h-\beta_h}  \p_1 \nabla_h u \|_{L^2}^{\f12} \\
			&\; +  \|\p^{\beta_h} \p_3 u_3\|_{H_{tan}^1}^{\f12} \|\p^{\beta_h} \p_3^2 u_3\|_{H_{tan}^1}^{\f12}
			\|\p^{\al_h-\beta_h}  \p_3 u \|_{L^2}^{\f12} \|\p^{\al_h-\beta_h}  \p_{13} u \|_{L^2}^{\f12}) \\
			&\; + \sum_{ [\f{\alpha_h}{2}] < \beta_h \le  \ah}   (\|\p^{\beta_h} \p_3 u_h\|_{L^2}^{\f12} \|\p^{\beta_h} \p_{13} u_h\|_{L^2}^{\f12}
			\|\p^{\al_h-\beta_h}  \nabla_h u \|_{H_{tan}^1}^{\f12} \|\p^{\al_h-\beta_h}  \p_3 \nabla_h u \|_{H_{tan}^1}^{\f12} \\
			&\; + \|\p^{\beta_h} \p_3 u_3\|_{L^2}^{\f12} \|\p^{\beta_h} \p_{3}^2 u_3\|_{L^2}^{\f12}
			\|\p^{\al_h-\beta_h}  \p_3 u \|_{H_{tan}^2} ) \Big)\\
			\lesssim &\; \sqrt{\me^m(t)}\md_{tan}^{m-1}(t),
			\dal \deqq
			and similarly, we have
			\begin{align*}
				&\;  - \i \p^{\ah}  ( u \cdot \nabla \p_3  u ) \cdot \p^{\ah} \p_{23 }b \ dx \\
				= &\;  - \i      u_3 \cdot  \p^{\ah} \p_3^2  u  \cdot \p^{\ah} \p_{23 }b \ dx - \i   u_h \cdot \nabla_h \p^{\ah} \p_3  u \cdot \p^{\ah} \p_{23 }b \ dx \\
				&\; - \sum_{0 \le \beta_h < \ah} C_{\ah}^{\beta_h} \i \p^{\beta_h}   u \cdot \p^{\ah-\beta_h} \nabla \p_3  u \cdot \p^{\ah} \p_{23 }b \ dx \\
				\lesssim &\; \|u_3\|_{L^{\infty}} \|
				\p^{\ah} \p_{23 }b \|_{L^2} \|  \p^{\ah} \p_3^2  u\|_{L^2} + \sqrt{\me^m(t)}\md_{tan}^{m-1}(t).
			\end{align*}
			Using the inequality \eqref{ie:Sobolev} and the interpolation inequality, we can obtain that
			\beq \label{ieq-help} \bal
			\|u_3\|_{L^{\infty}} \lesssim &\; \|\p_2 u_3\|_{L^2}^{\f18}\|\p_{12} u_3\|_{L^2}^{\f18}
			\|\p_3 u_3\|_{L^2}^{\f18}\|\p_{13} u_3\|_{L^2}^{\f18}\|\p_{23} u_3\|_{L^2}^{\f18}\|\p_{123} u_3\|_{L^2}^{\f18} \lesssim  \md_{tan}^{m-1}(t)^{\f14} \me^m(t)^{\f14},\\
			\|\p^{\ah} \p_3^2 u\|_{L^2} \lesssim &\;  \|\p^{\ah} \p_3 u \|_{L^2}^{\f12}  \|\p^{\ah} \p_3^3 u\|_{L^2}^{\f12} \lesssim  \md_{tan}^{m-1}(t)^{\f14} \me^m(t)^{\f14},
			\dal \deq
			which yields that
			\beqq \bal
			&\; - \i \p^{\ah}  ( u \cdot \nabla \p_3  u) \cdot 
			\p^{\ah} \p_{23 }b \ dx 
			\lesssim  \sqrt{\me^m(t)}\md_{tan}^{m-1}(t).
			\dal \deqq
			Thus, we can arrive that 
			\beqq
			VI_{6,1} \lesssim  \sqrt{\me^m(t)}\md_{tan}^{m-1}(t).
			\deqq
			By the same method, we can obtain that
			\begin{align*}
				VI_{6,3} 
				= &\; - \i \p^{\ah}  (\p_3(\f{\vro}{1+\vro}) (\p_1^2 u +\p_2 b) + \f{\vro}{1+\vro} (\p_1^2 \p_3 u +\p_{23} b) )  \cdot \p^{\ah} \p_{23 }b\ dx\\
				&\; - \i  \p^{\ah} \p_{23 }b \cdot\p^{\ah} \p_3 \nabla(\f{\vro}{1+\vro} ( \du-  b_2) )\ dx\\
				&\; + \i  \p^{\ah} \p_{23 }b \cdot\p^{\ah} \p_3 (\nabla(\f{\vro}{1+\vro}) ( \du-  b_2) )\ dx\\
				= &\; - \i \p^{\ah}  (\p_3(\f{\vro}{1+\vro}) (\p_1^2 u +\p_2 b) + \f{\vro}{1+\vro} (\p_1^2 \p_3 u +\p_{23} b) )  \cdot \p^{\ah} \p_{23 }b\ dx\\
				&\; + \i  \p^{\ah} \p_{23 }b \cdot\p^{\ah} \p_3 (\nabla(\f{\vro}{1+\vro}) ( \du-  b_2) )\ dx\\
				\lesssim &\; \sqrt{\me^m(t)}\md_{tan}^{m-1}(t),
			\end{align*}
			and  similar to the estimate \eqref{ieq-help}, we have
			\beqq \bal
			VI_{6,2} =  &\; - \var \i \p^{\ah} \p_3  ( \f{\vro}{1+\vro} \p_2^2 u)  \cdot \p^{\ah} \p_{23 }b\ dx + \var \i \p^{\ah} ( \f{\vro}{1+\vro} \p_2^2 u)  \cdot \p^{\ah} \p_{2} \p_3^2 b\ dx \\
			\lesssim &\; \sqrt{\me^m(t)}\md_{tan}^{m-1}(t),\\
			VI_{6,4} =  &\; \i \p^{\ah} \p_3 ( \f{1}{1+\vro}b\cdot \nabla b)   \cdot   \p^{\ah} \p_{23 }b\ dx  - \f12 \i \p^{\ah} \p_3 \nabla ( \f{1}{1+\vro} |b|^2 ) \cdot   \p^{\ah} \p_{23 }b\ dx  \\
			&\; - \f12 \i \p^{\ah} \p_3  ( \nabla(\f{1}{1+\vro}) |b|^2 ) \cdot   \p^{\ah} \p_{23 }b\ dx \\
			=  &\; \i \p^{\ah} (\p_3 (\f{1}{1+\vro}) b\cdot \nabla b + \f{1}{1+\vro} \p_3 b\cdot \nabla b + b \cdot  \p_3 \nabla b -  b \, \p_3 b \,\nabla( \f{1}{1+\vro}) )  \cdot   \p^{\ah} \p_{23 }b\ dx \\ 
			&\; - \f12 \i \p^{\ah} \p_{23 }b_3 \cdot  \p^{\ah} ( \p_3^2(\f{1}{1+\vro}) |b|^2)  \ dx   -\f12 \i \p^{\ah} \p_{23 }b_h \cdot  \p^{\ah} ( \p_3 \p_h (\f{1}{1+\vro}) |b|^2)  \ dx \\
			\lesssim &\; \sqrt{\me^m(t)}\md_{tan}^{m-1}(t) + \|b\|_{L^{\infty}}^2 \|   \p^{\ah} \p_{23 }b \|_{L^2} \|  \p^{\ah} \p_3^2 (\f{1}{1+\vro})\|_{L^2} \\
			\lesssim &\; \sqrt{\me^m(t)}\md_{tan}^{m-1}(t).
			\dal \deqq
			Combining the estimates from $VI_{6,1}$ to $VI_{6,4}$, we have
			\beqq 
			VI_6 \lesssim \sqrt{\me^m(t)}\md_{tan}^{m-1}(t).
			\deqq
			Similar to the estimates of $VI_1$ and $VI_6$, we can obtain that 
			\beqq \bal
			VI_3  \le &\; \nu \| \p^{\ah} \p_{23} u \|_{L^2}^2 + C_{\nu}\| (\p^{\ah} \p_3 \du, \p^{\ah} \p_3 \Delta_h b) \|_{L^2}^2 +  \sqrt{\me^m(t)}\md_{tan}^{m-1}(t),\\
			VI_5 \lesssim &\; \sqrt{\me^m(t)}\md_{tan}^{m-1}(t).
			\dal \deqq
			Combining the estimates from $VI_1$ to $VI_9$, and using the smallness of $\nu$, for $|\ah| = m-3$, we have
			\beqq \bal
			&\; \f{d}{dt} \i( \p^{\ah} \p_3 u_h \cdot \p^{\ah}  \nabla_h \p_3 \vro +  \p^{\ah}  \p_{23} b \cdot \p^{\ah}   \p_3 u )\ dx  + \|( \p^{\ah}  \p_3 \nabla_h  \vro,  \p^{\ah} \p_{23} u)\|_{L^2}^2 \\
			\lesssim &\;    \var\| \p^{\ah}\nabla_h (\p_2^2 u, \p_3^2 u, \p_{23} u)\|_{L^2}^2  +\var \|  \p^{\ah}  \p_{3}^2 b \|_{L^2}^2 + \|\p^{\ah}(  \Delta_h u, \nabla \p_1^2 u)  \|_{L^2}^2 + \|\p^{\ah}  \p_3( \du, \nabla_h \du)\|_{L^2}^2 \\
			&\; +  \| \p^{\ah}\nabla_h (  \nabla b, \p_3 \nabla_h b) \|_{L^2}^2+\sqrt{\me^m(t)}\md_{tan}^{m-1}(t).
			\dal \deqq
			Thus, together with the estimate for $|\ah|=0$, we finish the proof of this lemma.
		\end{proof}
		\subsection{Negative derivative estimate}
		
		In this subsection, we will establish the estimate in negative
		Sobolev space that will play an important role in establishing
		the decay rate estimate.
		
		\begin{lemm}
			Under the assumption \eqref{assumption},
			the smooth solution $(\vro ,u, b)$ of equation \eqref{eqr} has the estimate
			\beq\label{3701}
			\begin{aligned}
				&\; \|\hs( \vro, u, b)(t)\|_{L^2}^2 +\int_0^t \|\hs(\p_1  u, \du, \nabla_h b)\|_{L^2}^2 d\tau  + \var \int_0^t \|\hs(\p_2  u,\p_3 u,\p_3 b)\|_{L^2}^2 d\tau \\
				\lesssim &\;
				\|\hs( \vro, u, b)(0)\|_{L^2}^2 + \| (\vro,u)(0)\|_{H^m}^2 \|\hs u(0)\|_{L^2} +\delta^{\f32}.
			\end{aligned}
			\deq
		\end{lemm}
		\begin{proof}
			The equation \eqref{eqr} yields directly
			\beq\label{3702}
			\begin{aligned}
				&\;\frac{d}{dt}\frac{1}{2}\i(|\hs \vro|^2+|\hs u|^2+|\hs b|^2)dx
				+\i (|\p_1 \hs u|^2 dx+  | \hs \du|^2 + |\nabla_h \hs b|^2 )dx\\
				&\;  + \var \i (|\p_2 \hs u|^2 +|\p_3 \hs u|^2+ |\p_3 \hs b|^2) dx \\
				=&\; -\i \hs(u \cdot \nabla \vro) \cdot \hs \vro \ dx
				-\i \hs( \vro \, \du) \cdot \hs \vro \ dx \\
				&\; +\i \hs(-u \cdot \nabla b -b \, \du + b \cdot \nabla u ) \cdot \hs b \ dx+\i \hs F_2 \cdot \hs u \ dx\\
				:=&\;VII_1+VII_2+VII_3 +VII_4.
			\end{aligned}
			\deq
			Using the inequalities \eqref{ie:Sobolev}  and \eqref{a3},
			we may deduce that
			\beqq
			\begin{aligned}
				VII_1
				=&-\int_{\mathbb{R}}\int_{\mathbb{R}^2}
				\hs(u_h \cdot \nabla_h \vro+u_3 \p_3 \vro)\cdot \hs \vro \ dx_h dx_3\\
				\lesssim
				&\int_{\mathbb{R}}(\|\hs(u_h \cdot \nabla_h \vro)\|_{L^2(\mathbb{R}^2)}
				+\|\hs(u_3 \p_3 \vro)\|_{L^2(\mathbb{R}^2)})
				\|\hs \vro\|_{L^2(\mathbb{R}^2)}\ dx_3\\
				\lesssim
				&\int_{\mathbb{R}}(\|u_h \cdot \nabla_h \vro\|_{L^{\frac{1}{\frac{1}{2}+\frac{s}{2}}}(\mathbb{R}^2)}
				+\|u_3 \p_3 \vro\|_{L^{\frac{1}{\frac{1}{2}+\frac{s}{2}}}(\mathbb{R}^2)})
				\|\hs \vro\|_{L^2(\mathbb{R}^2)} \ dx_3\\
				\lesssim
				&\int_{\mathbb{R}}(\|u_h\|_{L^{\frac{2}{s}}(\mathbb{R}^2)}
				\|\nabla_h \vro\|_{L^2(\mathbb{R}^2)}
				+\|u_3\|_{L^{\frac{2}{s}}(\mathbb{R}^2)}
				\|\p_3 \vro\|_{L^2(\mathbb{R}^2)})
				\|\hs \vro\|_{L^2(\mathbb{R}^2)}\ dx_3\\
				\lesssim
				&\left.\|\|u_h\|_{L^\infty(\mathbb{R})}\right\|_{L^{\frac{2}{s}}(\mathbb{R}^2)}
				\|\nabla_h \vro\|_{L^2(\mathbb{R}^3)}
				\|\hs \vro\|_{L^2(\mathbb{R}^3)}\\
				&+\left.\|\|u_3\|_{L^\infty(\mathbb{R})}\right\|_{L^{\frac{2}{s}}(\mathbb{R}^2)}
				\|\p_3 \vro\|_{L^2(\mathbb{R}^3)}
				\|\hs \vro\|_{L^2(\mathbb{R}^3)}\\
				\lesssim
				&(\|u_h\|_{L^2}\|\p_2 u_h\|_{L^2}
				+\|\p_1 u_h\|_{L^2}\|\p_{12} u_h\|_{L^2})^{\frac{1-s}{2}}
				\|u_h\|_{L^2}^{\frac{2s-1}{2}}\|\p_3 u_h\|_{L^2}^{\frac{1}{2}}
				\|\nabla_h \vro\|_{L^2}\|\hs \vro\|_{L^2}\\
				&+(\|u_3\|_{L^2}\|\p_2 u_3\|_{L^2}
				+\|\p_1 u_3\|_{L^2}\|\p_{12} u_3\|_{L^2})^{\frac{1-s}{2}}
				\|u_3\|_{L^2}^{\frac{2s-1}{2}}\|\p_3 u_3\|_{L^2}^{\frac{1}{2}}
				\|\p_3 \vro\|_{L^2}\|\hs \vro\|_{L^2}\\
				\lesssim
				&\sqrt{\mathcal{E}_{tan}^{m-1}(t)}\sqrt{\md_{tan}^{m-1}(t)}
				\|\hs \vro\|_{L^2}
				+\mathcal{E}_{tan}^{m-1}(t)^{\frac34}
				\md_{tan}^{m-1}(t)^{\frac14}
				\|\hs \vro\|_{L^2}.
			\end{aligned}
			\deqq
			Similarly, we can obtain the estimate
			\beqq
			\begin{aligned}
				VII_2
				\lesssim
				&\sqrt{\mathcal{E}_{tan}^{m-1}(t)}\sqrt{\md_{tan}^{m-1}(t)}
				\|\hs \vro\|_{L^2},\\
				VII_3
				\lesssim
				&\sqrt{\mathcal{E}_{tan}^{m-1}(t)}\sqrt{\md_{tan}^{m-1}(t)}
				\|\hs b\|_{L^2}
				+\mathcal{E}_{tan}^{m-1}(t)^{\frac34}
				\md_{tan}^{m-1}(t)^{\frac14}
				\|\hs b\|_{L^2}.
			\end{aligned}
			\deqq
			As for the term $VII_4$, we split it as follows:
			\beqq \bal
			VII_4   = &\;\i \hs (-u \cdot \nabla u  - \f{\vro}{1+\vro}(\p_1^2 u+\nabla \du+\p_2 b) + \f{1}{1+\vro}(b\cdot \nabla b )) \cdot \hs u \ dx  \\
			&\;- \i \hs (\vro \nabla \vro) \cdot \hs u \ dx - \i \hs ( \f{1}{2(1+\vro)} \nabla |b|^2) \cdot \hs u \ dx  \\
			&\;+ \i \hs ( \f{\vro}{1+\vro} \nabla_h b_2) \cdot \hs u_h \ dx + \i \hs ( \f{\vro}{1+\vro} \p_3 b_2) \cdot \hs u_3 \ dx\\
			&\; - \var \i \hs ( \f{\vro}{1+\vro}(\p_2^2 u + \p_3^2 u)) \cdot \hs u \ dx  \\
			:= &\; \sum_{i=1}^{6} VII_{4,i}.
			\dal \deqq
			First, we can obtain the estimate for the following  terms similar to the terms from $VII_1$ 	to $VII_3$.
			\beqq
			\begin{aligned}
				VII_{4,1}
				\lesssim
				&\sqrt{\mathcal{E}_{tan}^{m-1}(t)}\sqrt{\md_{tan}^{m-1}(t)}
				\|\hs u\|_{L^2}
				+\mathcal{E}_{tan}^{m-1}(t)^{\frac34}
				\md_{tan}^{m-1}(t)^{\frac14}
				\|\hs u\|_{L^2},\\
				VII_{4,4}
				\lesssim
				& \sqrt{\mathcal{E}_{tan}^{m-1}(t)}\sqrt{\md_{tan}^{m-1}(t)}
				\|\hs u\|_{L^2},\\
				VII_{4,6}
				\lesssim
				&\sqrt{\mathcal{E}_{tan}^{m-1}(t)}\sqrt{\md_{tan}^{m-1}(t)}
				\|\hs u\|_{L^2}.\\
			\end{aligned}
			\deqq
			Integrating by parts, we have
			\beqq \bal
			VII_{4,2} = &  \f12 \i \hs (\vro^2) \cdot \hs \du \ dx \\
			\lesssim & (\|\vro\|_{L^2}\|\p_2 \vro\|_{L^2}
			+\|\p_1 \vro\|_{L^2}\|\p_{12} \vro\|_{L^2})^{\frac{1-s}{2}}
			\|\vro\|_{L^2}^{\frac{2s-1}{2}}\|\p_3 \vro\|_{L^2}^{\frac{1}{2}}
			\| \vro\|_{L^2} \|\hs \du\|_{L^2} \\
			\lesssim & \nu  \|\hs \du\|_{L^2}^2 + C_{\nu} \mathcal{E}_{tan}^{m-1}(t)^2.
			\dal \deqq
			As for the term $VII_{4,3}$, integrating by parts, we can obtain
			\beqq \bal
			VII_{4,3} = &\;  - \f12 \i \hs \nabla (\f{|b|^2}{1+\vro}) \cdot \hs u \ dx + \f12 \i \hs  (|b|^2\nabla(\f{1}{1+\vro})) \cdot \hs u \ dx \\
			= &\;  \f12 \i \hs (\f{|b|^2}{1+\vro}) \cdot \hs \du \ dx +\f12 \i \hs (|b|^2\nabla(\f{1}{1+\vro})) \cdot \hs u \ dx \\
			\lesssim &\; (\|b\|_{L^2}\|\p_2 b\|_{L^2}
			+\|\p_1 b\|_{L^2}\|\p_{12} b\|_{L^2})^{\frac{1-s}{2}}
			\|b\|_{L^2}^{\frac{2s-1}{2}}\|\p_3 b\|_{L^2}^{\frac{1}{2}}
			\| \f{b}{1+\vro} \|_{L^2} \|\hs \du\|_{L^2}  \\
			&\; +(\|b\|_{L^2}\|\p_2 b\|_{L^2}
			+\|\p_1 b\|_{L^2}\|\p_{12} b\|_{L^2})^{\frac{1-s}{2}}
			\|b\|_{L^2}^{\frac{2s-1}{2}}\|\p_3 b\|_{L^2}^{\frac{1}{2}}
			\| b\|_{L^{\infty}} \|\nabla(\f{1}{1+\vro})\|_{L^2} \|\hs u\|_{L^2}\\
			\le &\; \nu \|\hs \du\|_{L^2}^2 + C_{\nu} \mathcal{E}_{tan}^{m-1}(t)^2 + \|\hs u\|_{L^2} \mathcal{E}_{tan}^{m-1}(t)^{\f{3}{2}},
			\dal \deqq 
			where we have used the following estimate  
			$$\|b\|_{L^{\infty}} \lesssim \|b\|_{H_{tan}^{2}}^{\f12} \| \p_3 b\|_{H_{tan}^{2}}^{\f12} \lesssim \mathcal{E}_{tan}^{m-1}(t)^{\f12}$$
			by the inequality \eqref{ie:Sobolev}. 
			Finally, we deal with the most difficult term $VII_{4,5}$. Using the velocity equation $\eqref{eqr}_2$, we split the term as follows:
			\begin{align*}
				VII_{4,5}
				= &\; - \i \hs ( \f{\vro}{1+\vro} \p_t u_3 )\cdot \hs u_3 \ dx - \i \hs ( \f{\vro }{1+\vro}\p_3 \vro)  \cdot \hs u_3 \ dx \\
				&\; +\i \hs ( \f{\vro}{1+\vro} (\p_1^2 u_3 + \var \p_2^2 u_3 + \var \p_3^2 u_3 +\p_3 \du  +\p_2 b_3 )) \cdot \hs u_3 \ dx   \\
				&\; + \i \hs  (\f{\vro}{1+\vro}  (F_2)_3 )\cdot \hs u_3 \ dx\\
				= &\; - \p_t \i \hs ( \f{\vro}{1+\vro} u_3 )\cdot \hs u_3 \ dx  +   \i \hs ( \f{\vro}{1+\vro} u_3 )\cdot \hs \p_t u_3 \ dx \\
				&\; -  \i \hs ( \p_t \vro \cdot \f{ u_3}{(1+\vro)^2} )\cdot \hs u_3 \ dx   - \i \hs ( \f{\vro }{1+\vro}\p_3 \vro)  \cdot \hs u_3 \ dx   \\
				&\; +\i \hs ( \f{\vro}{1+\vro} (\p_1^2 u_3 + \var \p_2^2 u_3 + \var \p_3^2 u_3 +\p_3 \du  +\p_2 b_3 )) \cdot \hs u_3 \ dx   \\
				&\; + \i \hs  (\f{\vro}{1+\vro}  (F_2)_3 )\cdot \hs u_3 \ dx\\
				:=&\; - \p_t \i \hs ( \f{\vro}{1+\vro} u_3 )\cdot \hs u_3 \ dx  +\sum_{i=1}^{5} VII_{4,5,i}. 
			\end{align*}
			Substituting the density equation $\eqref{eqr}_1$ and velocity equation $\eqref{eqr}_2$ into the term $VII_{4,5,1}$, integrating by parts, we have
			\begin{align*}
				VII_{4,5,1}  
				= &\;  \i \hs ( \f{\vro}{1+\vro} u_3 )\cdot \hs ((F_2)_3 + \p_1^2 u_3 +\var \p_2^2 u_3 + \var\p_3^2 u_3+ \p_3 \du -\p_3 \vro +\p_2 b_3 -\p_3 b_2  ) \ dx \\
				= &\;  \i \hs ( \f{\vro}{1+\vro} u_3 )\cdot \hs ((F_2)_3 +\p_2 b_3) \ dx  -  \i \hs \p_1 ( \f{\vro}{1+\vro} u_3 )\cdot  \hs \p_1 u_3 \ dx    \\
				&\; - \var \i \hs \p_2 ( \f{\vro}{1+\vro} u_3 )\cdot  \hs \p_2 u_3 \ dx  - \var \i \hs \p_3 ( \f{\vro}{1+\vro} u_3 )\cdot  \hs \p_3 u_3 \ dx \\
				&\; -\i \hs (\p_3 ( \f{\vro}{1+\vro} ) u_3 )\cdot \hs (\du -\vro - b_2) \ dx   \\ 
				&\;- \i \hs ( \f{\vro}{1+\vro} (\du -\p_1 u_1 -\p_2 u_2 )) \cdot \hs  (\du -\vro - b_2) \ dx. 
			\end{align*}
			Thus, we have
			\beqq \bal 
			VII_{4,5,1} \lesssim &\; (\| u_3\|_{L^2}\|\p_2  u_3\|_{L^2}
			+\|\p_1 u_3\|_{L^2}\|\p_{12}  u_3\|_{L^2})^{\frac{1-s}{2}}
			\| u_3\|_{L^2}^{\frac{2s-1}{2}}\|\p_3  u_3\|_{L^2}^{\frac{1}{2}} \|\p_3 ( \f{\vro}{1+\vro}) \|_{L^2}\\
			&\; \times (\|\hs (F_2)_3 \|_{L^2} + \|\hs \p_2 b_3 \|_{L^2}) + (\|\hs (\du ,\p_1 u)\|_{L^2} + \var \|\hs (\p_2 u, \p_3 u)\|_{L^2} )\mathcal{E}_{tan}^{m-1}(t) \\
			&\; + \|\hs( \vro, u, b)\|_{L^2} (\mathcal{E}_{tan}^{m-1}(t)^{\f{3}{2}}+\sqrt{\mathcal{E}_{tan}^{m-1}(t)}\sqrt{\md_{tan}^{m-1}(t)}) \\
			\le &\; \nu (\|\hs (\du ,\p_1 u,\p_2 b)\|_{L^2}^2 + \var \|\hs (\p_2 u, \p_3 u)\|_{L^2}^2  )+ C_{\nu} \mathcal{E}_{tan}^{m-1}(t)^2+ \|\hs( \vro, u, b)\|_{L^2} \mathcal{E}_{tan}^{m-1}(t)^{\f{3}{2}} \\
			&\;  + \|\hs( \vro, u, b)\|_{L^2} \sqrt{\mathcal{E}_{tan}^{m-1}(t)}\sqrt{\md_{tan}^{m-1}(t)} +  \mathcal{E}_{tan}^{m-1}(t)^{\f32} ( \md_{tan}^{m-1}(t)^{\f12}+\mathcal{E}_{tan}^{m-1}(t)^{\f12} ),
			\dal \deqq
			and
			\beqq \bal
			VII_{4,5,2} = &\; \i \hs ( (u \cdot \nabla \vro + \vro \, \du + \du)\cdot \f{ u_3}{(1+\vro)^2} )\cdot \hs u_3 \ dx \\
			\lesssim &\; \|\hs u \|_{L^2} (\sqrt{\mathcal{E}_{tan}^{m-1}(t)}\sqrt{\md_{tan}^{m-1}(t)}+ \mathcal{E}_{tan}^{m-1}(t)^{\f{5}{4}} \md_{tan}^{m-1}(t)^{\f14} ).
			\dal \deqq
			Integrating by parts and using the equation $\p_3 u_3 = \du -\p_1 u_1 - \p_2 u_2$, we have 
			\begin{align*}
				VII_{4,5,3} =&\; -\f12 \i \hs \p_3( \f{\vro^2 }{1+\vro})  \cdot \hs u_3 \ dx + \f12 \i \hs ( \p_3 (\f{1}{1+\vro}) \vro^2)  \cdot \hs u_3 \ dx\\
				= &\;\f12 \i \hs ( \f{\vro^2 }{1+\vro})  \cdot \hs (\du -\p_1 u_1 - \p_2 u_2) \ dx  - \f12 \i \hs ( \p_3 \vro \cdot \f{\vro^2}{(1+\vro)^2})  \cdot \hs u_3 \ dx\\
				=  &\;\f12 \i \hs ( \f{\vro^2 }{1+\vro})  \cdot \hs (\du -\p_1 u_1 ) \ dx + \f12 \i \hs \p_2( \f{\vro^2 }{1+\vro})  \cdot \hs  u_2 \ dx  \\
				&\; - \f12 \i \hs ( \p_3 \vro \cdot \f{\vro^2}{(1+\vro)^2})  \cdot \hs u_3 \ dx\\
				\le &\; \nu \|\hs (\du, \p_1 u_1)\|_{L^2}^2 + C_{\nu} \mathcal{E}_{tan}^{m-1}(t)^2 + \|\hs u\|_{L^2} (\sqrt{\mathcal{E}_{tan}^{m-1}(t)}\sqrt{\md_{tan}^{m-1}(t)} +\mathcal{E}_{tan}^{m-1}(t)^{\f{3}{2}} ).
			\end{align*}
			Similarly, we have
			\beqq \bal
			VII_{4,5,4}+VII_{4,5,5} \le &\; \nu \|\hs (\du, \p_1 u_1)\|_{L^2}^2 + C_{\nu} \mathcal{E}_{tan}^{m-1}(t)^2  + \|\hs u\|_{L^2} (\sqrt{\mathcal{E}_{tan}^{m-1}(t)}\sqrt{\md_{tan}^{m-1}(t)}+\mathcal{E}_{tan}^{m-1}(t)^{\f{3}{2}}).
			\dal \deqq
			Thus we obtain the estimate
			\beqq \bal
			VII_{4,5} \le &\;  - \p_t \i \hs ( \f{\vro}{1+\vro} u_3 )\cdot \hs u_3 \ dx  + \nu (\|\hs (\du ,\p_1 u,\p_2 b)\|_{L^2}^2 + \var \|\hs (\p_2 u, \p_3 u)\|_{L^2}^2  )  \\
			&\;+ C_{\nu} \mathcal{E}_{tan}^{m-1}(t)^2  + \|\hs( \vro, u, b)\|_{L^2} (\sqrt{\mathcal{E}_{tan}^{m-1}(t)}\sqrt{\md_{tan}^{m-1}(t)}  +\mathcal{E}_{tan}^{m-1}(t)^{\f{5}{4}} \md_{tan}^{m-1}(t)^{\f14}+ \mathcal{E}_{tan}^{m-1}(t)^{\f{3}{2}})\\
			&\; +   \mathcal{E}_{tan}^{m-1}(t)^{\f32} ( \md_{tan}^{m-1}(t)^{\f12}+\mathcal{E}_{tan}^{m-1}(t)^{\f12} ),
			\dal \deqq
			which, combining the estimates of $VII_{4,1}$ to $VII_{4,4}$, yields that
			\beqq \bal
			VII_{4} \le &\;  - \p_t \i \hs ( \f{\vro}{1+\vro} u_3 )\cdot \hs u_3 \ dx  + \nu 
			(\|\hs (\du, \p_1 u,\p_2 b)\|_{L^2}^2  + \var \|\hs (\p_2 u, \p_3 u)\|_{L^2}^2)  \\
			&\; + \|\hs( \vro, u, b)\|_{L^2} (\sqrt{\mathcal{E}_{tan}^{m-1}(t)}\sqrt{\md_{tan}^{m-1}(t)}  +\mathcal{E}_{tan}^{m-1}(t)^{\frac34}
			\md_{tan}^{m-1}(t)^{\frac14} +\mathcal{E}_{tan}^{m-1}(t)^{\f{3}{2}})\\
			&\; + \mathcal{E}_{tan}^{m-1}(t)^{\f32} ( \md_{tan}^{m-1}(t)^{\f12}+\mathcal{E}_{tan}^{m-1}(t)^{\f12} ).
			\dal \deqq
			Substituting the estimates of terms for $VII_1$  through $VII_4$
			into \eqref{3702}, using the smallness of $\nu$ and integrating over $[0, t]$, we obtain
			\beq \label{3703}
			\begin{aligned}
				&\;\frac{1}{2}\|\hs (\vro, u, b)(t)\|_{L^2}^2 +\i \hs ( \f{\vro}{1+\vro} u_3 )(t) \cdot \hs u_3(t) \ dx
				+\int_0^t \|\hs(\p_1  u, \du, \nabla_h b)\|_{L^2}^2 d\tau  \\
				&\;+ \var \int_0^t \|\hs(\p_2  u,\p_3 u,\p_3 b)\|_{L^2}^2 d\tau\\
				\lesssim
				&\;\frac{1}{2}\|\hs (\vro, u, b)(0)\|_{L^2}^2
				+\i \hs ( \f{\vro}{1+\vro} u_3 )(0) \cdot \hs u_3(0) \ dx\\
				&\; +\underset{0\le \tau \le t}{\sup}\|\hs (\vro, u, b)(\tau)\|_{L^2}
				\int_0^t (\sqrt{\mathcal{E}_{tan}^{m-1}(\tau)}\sqrt{\md_{tan}^{m-1}(\tau)} + \mathcal{E}_{tan}^{m-1}(\tau)^{\frac34} \md_{tan}^{m-1}(\tau)^{\frac14}+\mathcal{E}_{tan}^{m-1}(\tau)^{\f32})d\tau\\
				&\; + \int_0^t (  \mathcal{E}_{tan}^{m-1}(\tau)^{\f32} \md_{tan}^{m-1}(\tau)^{\frac12} +\mathcal{E}_{tan}^{m-1}(\tau)^{2} ) d\tau.
			\end{aligned}
			\deq
			Using the assumption \eqref{assumption} and H\"{o}lder inequality, we have
			\beq\label{3704}
			\begin{aligned}
				\int_0^t  \sqrt{\mathcal{E}_{tan}^{m-1}(\tau)}\sqrt{\md_{tan}^{m-1}(\tau)} d\tau
				&\le
				\left\{\int_0^t   \mathcal{E}_{tan}^{m-1}(\tau)
				(1+\tau)^{-\sigma} d\tau\right\}^{\frac{1}{2}}
				\left\{\int_0^t  \md_{tan}^{m-1}(\tau) (1+\tau)^{\sigma}
				d\tau\right\}^{\frac{1}{2}}\\
				&\lesssim
				\delta \left\{\int_0^t (1+\tau)^{-(s+\sigma)} d\tau\right\}^{\frac{1}{2}}
				\lesssim \delta, \\
				\int_0^t  \mathcal{E}_{tan}^{m-1}(\tau)^{\frac34} \md_{tan}^{m-1}(\tau)^{\frac14} d\tau
				\lesssim
				&\left\{\int_0^t  \mathcal{E}_{tan}^{m-1}(\tau)
				(1+\tau)^{-\frac13\sigma} d\tau\right\}^{\frac{3}{4}}
				\left\{\int_0^t  \md_{tan}^{m-1}(\tau) (1+\tau)^{\sigma}
				d\tau\right\}^{\frac{1}{4}}\\
				\lesssim&
				\delta \left\{\int_0^t (1+\tau)^{-(s+\frac13\sigma)} d\tau\right\}^{\frac{3}{4}}
				\lesssim \delta,
			\end{aligned}
			\deq
			and
			\beq\label{3705}
			\begin{aligned}
				\int_0^t \mathcal{E}_{tan}^{m-1}(\tau)^{\f32} d\tau \lesssim \delta^{\f32}, ~~~~\int_0^t (  \mathcal{E}_{tan}^{m-1}(\tau)^{\f32} \md_{tan}^{m-1}(\tau)^{\frac12} +&\mathcal{E}_{tan}^{m-1}(\tau)^{2} ) d\tau
				\lesssim \delta^2,
			\end{aligned}
			\deq
			where we use $\frac{9}{10}<\sigma<s<1$ in the above estimate.
			It is easy to check that
			\beqq
			\i \hs ( \f{\vro}{1+\vro} u_3 )(t) \cdot \hs u_3(t) \ dx \lesssim \| (\vro, u)\|_{H^m}^2 \|\hs u_3\|_{L^2} \le  \mathcal{E}^m(t)^{\f32}.
			\deqq
			Substituting above inequality and the estimates from \eqref{3704} to  \eqref{3705} into \eqref{3703}, under the assumption \eqref{assumption}, we can obtain  that
			\beqq \bal
			&\; \|\hs( \vro, u, b)(t)\|_{L^2}^2 +\int_0^t \|\hs(\p_1  u, \du, \nabla_h b)\|_{L^2}^2 d\tau  + \var \int_0^t \|\hs(\p_2  u,\p_3 u,\p_3 b)\|_{L^2}^2 d\tau \\
			\lesssim &\;
			\|\hs( \vro, u, b)(0)\|_{L^2}^2 + \|(\vro,u)(0)\|_{H^m}^2 \|\hs u(0)\|_{L^2}+\delta^{\f32}.
			\dal  \deqq
			Therefore, we complete the proof of this lemma.
		\end{proof}
		
		Next, we establish estimate of vertical derivative
		of solution in negative Sobolev space.
		\begin{lemm}
			Under the assumption \eqref{assumption},
			the smooth solution $(\vro ,u, b)$ of equation \eqref{eqr} has the estimate
			\beq\label{3801}
			\begin{aligned}
				&\; \|\hs  \p_3 (\vro, u, b)(t)\|_{L^2}^2
				+\int_0^t \| \hs \p_3 (\p_1 u, \du, \nabla_h  b)\|_{L^2}^2 d\tau + \var \int_0^t \|\hs \p_3 (\p_2 u,\p_3u,\p_3 b)\|_{L^2}^2 d\tau \\
				\lesssim  &\; 
				\|\hs  \p_3 (\vro, u, b)(0)\|_{L^2}^2+\delta^2.
			\end{aligned}
			\deq
		\end{lemm}
		\begin{proof}
			Using the equation \eqref{eqr} and
			integrating by parts, it is easy to check that
			\begin{equation}\label{3802}
				\begin{aligned}
					&\frac{d}{dt}\frac{1}{2} \|\hs  \p_3 (\vro, u, b)(t)\|_{L^2}^2 +\| \hs \p_3 (\p_1 u, \du, \nabla_h  b)\|_{L^2}^2+ \var \|\hs \p_3 (\p_2 u,\p_3u,\p_3 b)\|_{L^2}^2\\
					=&-\i \hs \p_3 (u \cdot \nabla \vro) \cdot \hs \p_3  \vro \ dx
					-\i \hs \p_3 ( \vro \, \du) \cdot \hs \p_3  \vro \ dx\\
					& +\i \hs \p_3 (-u \cdot \nabla b -b \, \du+ b \cdot \nabla u ) \cdot \hs \p_3  b \ dx - \i \hs \p_3 ( u \cdot \nabla u) \cdot \hs \p_3 u \ dx 
					\\
					&\;- \i \hs \p_3 ( \vro \nabla \vro ) \cdot \hs \p_3 u \ dx
					-  \i \hs \p_3 (\f{\vro}{1+\vro}(\p_1^2 u+\nabla \du +\p_2 b) ) \cdot \hs \p_3 u \ dx\\
					&\;+  \i \hs \p_3 (\f{\vro}{1+\vro} \nabla b_2) \cdot \hs \p_3 u \ dx   + \i \hs \p_3 (\f{1}{1+\vro} b\cdot \nabla b) \cdot \hs \p_3 u \ dx \\
					&\; -  \f12\i \hs \p_3 (\f{1}{1+\vro}( \nabla |b|^2 )) \cdot \hs \p_3 u \ dx - \var \i \hs \p_3 (\f{\vro}{1+\vro}(\p_2^2 u+\p_3^2 u )) \cdot \hs \p_3 u \ dx \\
					:=&\;\sum_{i=1}^{10}VIII_i.
				\end{aligned}
			\end{equation}
			The H\"{o}lder's inequality and  inequality \eqref{a3} in Lemma \ref{H-L} yield directly
			\begin{equation}\label{3803}
				\begin{aligned}
					VIII_1 = &\int_{\mathbb{R}^3} \hs( \p_3 u \cdot \nabla  \vro+ u\cdot \nabla \p_3 \vro ) \cdot \hs \p_3 \vro \ dx \\
					\lesssim & \int_{\mathbb{R}}
					(\|\p_3 u_h \cdot \nabla_h \vro \|_{L^{\frac{1}{\frac12+\frac{s}{2}}}(\mathbb{R}^2)} +\|\p_3 u_3 \partial_3 \vro \|_{L^{\frac{1}{\frac12+\frac{s}{2}}}(\mathbb{R}^2)} 
					+ \|u_h \cdot \nabla_h \p_3 \vro \|_{L^{\frac{1}{\frac12+\frac{s}{2}}}(\mathbb{R}^2)}
					+\|u_3 \partial_3^2 \vro \|_{L^{\frac{1}{\frac12+\frac{s}{2}}}(\mathbb{R}^2)})\\
					&\; \times
					\|\hs \p_3 \vro \|_{L^2(\mathbb{R}^2)}dx_3.
				\end{aligned}
			\end{equation}
			Similar to the estimate of term $VII_1$, it is easy to check that
			\begin{align}\label{3804}
				&\;\int_{\mathbb{R}}
				(\|\p_3 u_h \cdot \nabla_h \vro \|_{L^{\frac{1}{\frac12+\frac{s}{2}}}(\mathbb{R}^2)} +\|\p_3 u_3 \partial_3 \vro \|_{L^{\frac{1}{\frac12+\frac{s}{2}}}(\mathbb{R}^2)} 
				+ \|u_h \cdot \nabla_h \p_3 \vro \|_{L^{\frac{1}{\frac12+\frac{s}{2}}}(\mathbb{R}^2)})
				\|\hs \p_3 \vro\|_{L^2(\mathbb{R}^2)}\ dx_3\notag\\
				\lesssim&\;
				\Big
				(
				\left(\|\nabla_h \vro\|_{L^2}\|\partial_{1} \p_{h}  \vro\|_{L^2}
				+\|\p_{2} \partial_{h}  \vro\|_{L^2}\| \p_{12} \partial_{h} \vro\|_{L^2}\right)^{\frac{1-s}{2}}
				\|\nabla_h \vro\|_{L^2}^{\frac{2s-1}{2}}
				\| \partial_3 \nabla_h  \vro\|_{L^2}^{\frac12}
				\| \p_3 u_h\|_{L^2} \notag\\
				&\; +\left(\|\partial_3 u_3\|_{L^2}\|\partial_{13} u_3\|_{L^2}
				+\|\p_{23} u_3\|_{L^2}\| \p_{123} u_3\|_{L^2}\right)^{\frac{1-s}{2}}
				\|\partial_3 u_3\|_{L^2}^{\frac{2s-1}{2}}
				\| \partial_3^2 u_3\|_{L^2}^{\frac12}
				\| \p_3 \vro\|_{L^2}  \\
				&\; + \left(\|u_h\|_{L^2}\|\partial_{1} u_h\|_{L^2}
				+\|\p_{2} u_h\|_{L^2}\| \p_{12} u_h\|_{L^2}\right)^{\frac{1-s}{2}}
				\|u_h\|_{L^2}^{\frac{2s-1}{2}}
				\| \partial_3 u_h\|_{L^2}^{\frac12}
				\| \p_h \p_3 \vro\|_{L^2} 
				\Big)
				\|\hs \p_3 \vro\|_{L^2}\notag\\
				\lesssim&\; \sqrt{\mathcal{E}_{tan}^{m-1}(t)} \sqrt{\md_{tan}^{m-1}(t)}
				\|\hs \p_3 \vro\|_{L^2}.\notag
			\end{align}
			Now, let us deal with the difficult term $\int_{\mathbb{R}}
			\|u_3 \partial_3^2 \vro \|_{L^{\frac{1}{\frac12+\frac{s}{2}}}(\mathbb{R}^2)}
			\|\hs \p_3 \vro \|_{L^2(\mathbb{R}^2)}dx_3$.
			Indeed, we may check that
			\begin{equation}\label{3805}
				\begin{aligned}
					&~~~~\int_{\mathbb{R}}
					\|u_3 \partial_3^2 \vro\|_{L^{\frac{1}{\frac12+\frac{s}{2}}}(\mathbb{R}^2)}
					\|\hs \p_3 \vro \|_{L^2(\mathbb{R}^2)}\ dx_3\\
					&\lesssim\int_{\mathbb{R}}
					\left( \left\|\|u_3\|_{L^\infty(\mathbb{R})}\right\|_{L^{\frac2s}(\mathbb{R}^2)}\right)
					\|\p_3^2 \vro\|_{L^2(\mathbb{R}^2)}
					\|\hs \p_3 \vro\|_{L^2(\mathbb{R}^2)}\ dx_3\\
					&\lesssim (\|u_3\|_{L^2}\|\partial_2 u_3\|_{L^2}
					+\|\partial_1 u_3\|_{L^2}\|\partial_{12}u_3\|_{L^2})^{\frac{1-s}{2}}
					\|u_3\|_{L^2}^{\frac{2s-1}{2}}
					\|\partial_3 u_3\|_{L^2}^{\frac12} 
					\|\p_3 \vro\|_{L^2}^{\f23} \|\p_3^4 \vro\|_{L^2}^{\f13} \|\hs \p_3 \vro\|_{L^2}\\
					&\lesssim \|\p_3^4 \vro\|_{L^2}^{\frac13}
					\mathcal{E}_{tan}^{m-1}(t)^{\frac{7}{12}}
					\md_{tan}^{m-1}(t)^{\frac14}
					\|\hs \p_3 \vro\|_{L^2}.
				\end{aligned}
			\end{equation}
			Substituting estimates \eqref{3804} and \eqref{3805}  into \eqref{3803}, we have
			\beqq
			\begin{aligned}
				VIII_1
				\lesssim&\;
				\sqrt{\mathcal{E}_{tan}^{m-1}(t)} \sqrt{\md_{tan}^{m-1}(t)} 
				\|\hs \p_3 \vro\|_{L^2} +\|\p_3^4 \vro\|_{L^2}^{\frac13}
				\mathcal{E}_{tan}^{m-1}(t)^{\frac{7}{12}}
				\md_{tan}^{m-1}(t)^{\frac14}
				\|\hs  \p_3 \vro\|_{L^2}.
			\end{aligned}
			\deqq
			Similarly, it is easy to check that
			\beqq
			\begin{aligned}
				VIII_2
				\lesssim
				&\sqrt{\mathcal{E}_{tan}^{m-1}(t)} \sqrt{\md_{tan}^{m-1}(t)}
				\|\hs \p_3 \vro\|_{L^2},\\
				VIII_3 
				\lesssim
				&\sqrt{\mathcal{E}_{tan}^{m-1}(t)} \sqrt{\md_{tan}^{m-1}(t)}
				\|\hs \p_3 b\|_{L^2}
				+\|(\p_3^4 u, \p_3^4 b)\|_{L^2}^{\frac13}
				\mathcal{E}_{tan}^{m-1}(t)^{\frac{7}{12}}
				\md_{tan}^{m-1}(t)^{\frac14}
				\|\hs \p_3 b\|_{L^2},\\
				VIII_4 
				\lesssim
				&\sqrt{\mathcal{E}_{tan}^{m-1}(t)} \sqrt{\md_{tan}^{m-1}(t)}
				\|\hs \p_3 u\|_{L^2} +\|\p_3^4 u\|_{L^2}^{\frac13}
				\mathcal{E}_{tan}^{m-1}(t)^{\frac{7}{12}}
				\md_{tan}^{m-1}(t)^{\frac14}
				\|\hs \p_3 u\|_{L^2},\\
				VIII_8 
				\lesssim
				&(1+ \sqrt{\mathcal{E}^{m}(t)})(\sqrt{\mathcal{E}_{tan}^{m-1}(t)} \sqrt{\md_{tan}^{m-1}(t)} + \mathcal{E}_{tan}^{m-1}(t)^{\frac34} \md_{tan}^{m-1}(t)^{\frac14})
				\|\hs \p_3 b\|_{L^2} \\
				&\; +\|\p_3^4 b\|_{L^2}^{\frac13}
				\mathcal{E}_{tan}^{m-1}(t)^{\frac{7}{12}}
				\md_{tan}^{m-1}(t)^{\frac14}
				\|\hs \p_3 b\|_{L^2}.
			\end{aligned}
			\deqq
			Using the H\"{o}lder's inequality and inequality \eqref{a3} in Lemma \ref{H-L} yield directly 
			\beqq
			\begin{aligned}
				VIII_6 &\;=-   \i \hs (\p_3 (\f{\vro}{1+\vro}) \nabla \du) ) \cdot \hs \p_3 u \ dx  -   \i \hs (\f{\vro}{1+\vro} \p_3 \nabla \du) ) \cdot \hs \p_3 u \ dx\\
				&\;-  \i \hs \p_3 (\f{\vro}{1+\vro}(\p_1^2 u +\p_2 b) ) \cdot \hs \p_3 u \ dx  \\
				\lesssim &\;
				\Big( (\| \nabla \du \|_{L^2}\|\partial_2  \nabla \du\|_{L^2}
				+\|\partial_1  \nabla \du\|_{L^2}
				\|\partial_{12} \nabla \du \|_{L^2})^{\frac{1-s}{2}}
				\| \nabla \du\|_{L^2}^{\frac{2s-1}{2}}
				\|\partial_3  \nabla \du\|_{L^2}^{\frac12}
				\|\p_3 (\f{\vro}{1+ \vro})\|_{L^2} \\  
				&\; +(\|\vro \|_{L^2}\|\partial_2  \vro\|_{L^2}
				+\|\partial_1  \vro\|_{L^2}
				\|\partial_{12}  \vro\|_{L^2})^{\frac{1-s}{2}}
				\| \vro\|_{L^2}^{\frac{2s-1}{2}}
				\|\partial_3 \vro\|_{L^2}^{\frac12} 
				\|\p_3 \nabla \du \|_{L^2} \Big) \|\hs \p_3 u  \|_{L^2} \\
				&\; +\sqrt{\mathcal{E}_{tan}^{m-1}(t)} \sqrt{\md_{tan}^{m-1}(t)}
				\|\hs \p_3 u\|_{L^2}  \\
				\lesssim &\; \sqrt{\mathcal{E}_{tan}^{m-1}(t)} \sqrt{\md_{tan}^{m-1}(t)}
				\|\hs \p_3 u\|_{L^2} +	\|\hs \p_3 u\|_{L^2}\|\p_3^3 \nabla \du\|_{L^2}^{\frac13}
				\mathcal{E}_{tan}^{m-1}(t)^{\frac{1}{2}}
				\md_{tan}^{m-1}(t)^{\frac13} \\
				&\;+\|\hs \p_3 u\|_{L^2} \|\p_3^3 \nabla \du\|_{L^2}^{\frac16}
				\mathcal{E}_{tan}^{m-1}(t)^{\frac{1}{2}}
				\md_{tan}^{m-1}(t)^{\frac{5}{12}},
			\end{aligned}
			\deqq
			where we have used the estimate
			\beqq
			\|\p_3 \nabla \du \|_{L^2} \lesssim \| \nabla \du \|_{L^2}^{\f23} \|\p_3^3 \nabla \du \|_{L^2}^{\f13}.
			\deqq
			Integrating by parts, we have
			\begin{align*}
				VIII_9 = &\; -\f12 \i \hs \p_3 \nabla ( \f{\vro}{1+\vro} |b|^2) \cdot \hs \p_3 u \ dx + \f12 \i \hs \p_3 (\nabla (\f{\vro}{1+\vro})  |b|^2) \cdot \hs \p_3 u \ dx \\
				= &\; \f12 \i \hs ( \f{\vro}{1+\vro} |b|^2) \cdot \hs \p_3 \du \ dx + \f12 \i \hs \p_3 (\nabla_h (\f{\vro}{1+\vro})  |b|^2) \cdot \hs \p_3 u_h \ dx \\
				&\;+ \f12 \i \hs \p_3 (\p_3 (\f{\vro}{1+\vro})  |b|^2) \cdot \hs (\du-\p_1 u_1 -\p_2 u_2) \ dx \\
				= &\; \f12 \i \hs ( \f{\vro}{1+\vro} |b|^2) \cdot \hs \p_3 \du \ dx + \f12 \i \hs \p_3 (\nabla_h (\f{\vro}{1+\vro})  |b|^2) \cdot \hs \p_3 u_h \ dx \\
				&\; - \f12 \i \hs (\p_3 (\f{\vro}{1+\vro})  |b|^2) \cdot \hs \p_3 (\du- \p_1 u_1) \ dx +  \f12 \i \hs \p_2 (\p_3 (\f{\vro}{1+\vro})  |b|^2) \cdot \hs \p_3 u_2 \ dx \\
				\le &\; \nu \|(\hs \p_3 \du ,\hs \p_{13} u)\|_{L^2}^2 
				+C_{\nu} \mathcal{E}_{tan}^{m-1}(t)^{3} +
				\|\hs \p_3 u\|_{L^2} \mathcal{E}_{tan}^{m-1}(t)
				\md_{tan}^{m-1}(t)^{\frac12}.
			\end{align*}
			Similarly, it is easy to check that
			\begin{align*}
				VIII_5
				= &\;\f12 \int_{\mathbb{R}^3} \hs \nabla \p_3 (|\vro|^2)\cdot \hs \p_3 u \ dx = - \f12 \int_{\mathbb{R}^3} \hs \p_3 (|\vro|^2)\cdot \hs \p_3 \du \ dx\\
				\le &\; \nu
				\|\hs \p_3 \du  \|_{L^2}^2 +  C_{\nu} \mathcal{E}_{tan}^{m-1}(t)^2,\\
				VIII_7  = &\; \f12 \i \hs \p_3 \nabla ( \f{\vro}{1+\vro} b_2) \cdot \hs \p_3 u \ dx - \f12 \i \hs \p_3 (\nabla_h (\f{\vro}{1+\vro}) b_2) \cdot \hs \p_3 u_h \ dx \\
				&\;+ \f12 \i \hs  (\p_3 (\f{\vro}{1+\vro}) b_2) \cdot \hs \p_3^2 u_3 \ dx \\
				\le &\; \nu \|(\hs \p_3 \du ,\hs \p_{13} u)\|_{L^2}^2 
				+ C_{\nu} \mathcal{E}_{tan}^{m-1}(t)^2 +
				\|\hs \p_3 u\|_{L^2} \mathcal{E}_{tan}^{m-1}(t)^{\f12}
				\md_{tan}^{m-1}(t)^{\frac12},\\
				VIII_{10} = &\;  \var \i \hs  (\f{\vro}{1+\vro}(\p_2^2 u+\p_3^2 u )) \cdot \hs \p_3^2 u \ dx \\
				\le &\;\nu \var \| \hs \p_3^2 u\|_{L^2}^2 + C_{\nu} \mathcal{E}_{tan}^{m-1}(t)  \mathcal{D}_{tan}^{m-1}(t). 
			\end{align*}
			Substituting the estimates of terms for  $VIII_1$
			through $VIII_{10}$ into \eqref{3802} and using the smallness of $\nu$, we have
			\beq\label{3806}
			\begin{aligned}
				&\frac{1}{2}\|\hs  \p_3 (\vro, u, b)(t)\|_{L^2}^2
				+\int_0^t \| \hs \p_3 (\p_1 u, \du, \nabla_h  b)\|_{L^2}^2 d\tau + \var \int_0^t \|\hs \p_3 (\p_2 u,\p_3u,\p_3 b)\|_{L^2}^2 d\tau\\
				\lesssim
				&\frac{1}{2}\|\hs  \p_3 (\vro, u, b)(0)\|_{L^2}^2
				+\underset{0\le \tau \le t}{\sup}\|\hs  \p_3 (\vro, u, b)(\tau)\|_{L^2}
				\int_0^t \mathcal{E}_{tan}^{m-1}(\tau)^{\f12}
				\md_{tan}^{m-1} (\tau)^{\f12}
				d\tau\\
				&+\underset{0\le \tau \le t}{\sup}\|\hs  \p_3 (\vro, u, b)(\tau)\|_{L^2}
				\int_0^t
				(\mathcal{E}_{tan}^{m-1}(\tau)^{\f34}
				\md_{tan}^{m-1}(\tau)^{\f14} + \md^{m}(\tau) ^{\f1{12}} \mathcal{E}_{tan}^{m-1}(\tau) ^{\frac{1}{2}}
				\md_{tan}^{m-1}(\tau) ^{\frac{5}{12}} )d\tau\\
				&+\underset{0\le \tau \le t}{\sup}
				\|\hs  \p_3 (\vro, u, b)(\tau)\|_{L^2}
				\underset{0\le \tau \le t}{\sup}
				\|(\p_3^4 u, \p_3^4 b,\p_3^4 \vro)(\tau)\|_{L^2}^{\frac13}
				\int_0^t  \mathcal{E}_{tan}^{m-1}(\tau)^{\frac{7}{12}}
				\md_{tan}^{m-1}(\tau)^{\frac14} d \tau\\
				& 
				+ \int_0^t (\mathcal{E}_{tan}^{m-1}(\tau) \mathcal{D}_{tan}^{m-1}(\tau) +\mathcal{E}_{tan}^{m-1}(\tau)^{2} ) d\tau.
			\end{aligned}
			\deq
			Using the assumption \eqref{assumption}, it is easy to check that
			\begin{equation}\label{3807}
				\begin{aligned}
					\int_0^t \mathcal{E}_{tan}^{m-1}(\tau)^{\frac{7}{12}}
					\md_{tan}^{m-1}(\tau)^{\frac14} d\tau
					&\lesssim
					\bigg\{\int_0^t \md_{tan}^{m-1}(\tau)
					(1+\tau)^{\sigma}d\tau\bigg\}^{\frac{1}{4}}
					\bigg\{\int_0^t \mathcal{E}_{tan}^{m-1}(\tau)^{\frac{7}{9}}(1+\tau)^{-\frac13 \sigma}
					d\tau\bigg\}^{\frac{3}{4}}\\
					&\lesssim
					\delta^{\frac56}\bigg\{\int_0^t (1+\tau)^{-(\frac{7}{9}s+\frac13\sigma)}d\tau
					\bigg\}^{\frac{1}{2}}
					\lesssim \delta^{\frac56},
				\end{aligned}
			\end{equation}
			and 
			\begin{equation}\label{3808} \begin{aligned}
					&\;\int_0^t \md^{m}(\tau)^{\f1{12}}
					\mathcal{E}_{tan}^{m-1}(\tau)^{\frac{1}{2}}
					\md_{tan}^{m-1}(\tau)^{\frac{5}{12}} d\tau  \lesssim  \dl, \\
					&\;  \int_0^t (\mathcal{E}_{tan}^{m-1}(\tau) \mathcal{D}_{tan}^{m-1}(\tau) +\mathcal{E}_{tan}^{m-1}(\tau)^{2}) d\tau \lesssim  \dl^2,
			\end{aligned} \end{equation}
			where we have used the condition $\frac{9}{10}<\sigma<s<1$.
			Substituting \eqref{3807}, \eqref{3808} and \eqref{3704} into \eqref{3806}, then we have
			\beqq
			\begin{aligned}
				&\; \|\hs \p_3 (\vro, u, b)(t)\|_{L^2}^2
				+\int_0^t \| \hs \p_3 (\p_1 u, \du, \nabla_h  b)\|_{L^2}^2 d\tau + \var \int_0^t \|\hs \p_3 (\p_2 u,\p_3u,\p_3 b)\|_{L^2}^2 d\tau \\
				\lesssim &\; 
				\|\hs \p_3 (\vro, u, b)(0)\|_{L^2}^2+\delta^2.
			\end{aligned}
			\deqq
			Therefore, we complete the proof of this lemma.
		\end{proof}

		\subsection{Decay in time estimate}
		In this subsection, we will establish the decay rate estimates
		for density, velocity and magnetic field. This decay in time estimate
		will help us to close the energy estimate.

		\begin{lemm}
			Under the assumption \eqref{assumption},
			the smooth solution $(\vro ,u, b)$ of equation \eqref{eqr} has the estimate
			\beq\label{3901}
			(1+t)^s \mathcal{E}^{m-1}_{tan}(t) \lesssim C_0,
			\deq
			where the constant $C_0$ is defined in \eqref{co}.
		\end{lemm}
		\begin{proof}
			Under the assumption \eqref{assumption} and using the fact that the norm $ \|f\|_{H^m}$ is equivalent to the norms $\|f\|_{H_{tan}^m}$ and $\|  \p_3^m f\|_{L^2}$,  the combination of estimates \eqref{3101},
			\eqref{3201}, \eqref{3401} and \eqref{3501} yields directly
			\beq\label{39-vroub-m}
			\|(\vro, u,  b)(t)\|_{H^m}^2+\int_0^t \md^{m}(\tau)  d\tau
			\le C\|(\vro, u,  b)(0)\|_{H^m}^2+C\delta^{\frac32}.
			\deq
			The combination of estimates \eqref{3701} and \eqref{3801} yields directly
			\beqq \bal
			&\; \|(\hs( \vro, u, b), \hs \p_3 ( \vro, u, b)) (t)\|_{L^2}^2 \\
			\lesssim  &\; 
			\|(\hs( \vro, u, b), \hs \p_3 ( \vro, u, b))(0)\|_{L^2}^2 + \| (\vro,u)(0)\|_{H^m}^2 \|\hs u(0)\|_{L^2} +\delta^{\f32},
			\dal  \deqq
			which, together with estimate \eqref{39-vroub-m}, yields directly
			\beq\label{39-em}
			\mathcal{E}^m(t)
			\le CC_0,
			\deq
			where the constant $C_0$ is defined by
			\beq\label{co}
			C_0:=\|(\vro, u,  b)(0)\|_{H^m}^2
			+\|(\hs( \vro, u, b), \hs \p_3 ( \vro, u, b))(0)\|_{L^2}^2+ \|\hs u(0)\|_{L^2}\|(\vro, u)(0)\|_{H^m}^2+\delta^{\f32}.
			\deq
			From estimates \eqref{3101} and \eqref{3301}, we have
			\beq\label{3902} \bal
			&\; \frac{d}{dt}\Big(\| (\vro, u, \f{1}{\sqrt{1+\vro}} b)\|_{L^2}^2 + \sum\limits_{|\alpha_h| = m-1}\| (\p^{\ah} u, \p^{\ah}\vro ,\f{1}{\sqrt{1+\vro}}\p^{\al_h} b)\|_{L^2}^2 \Big)  \\
			&\; +  \|(\p_1 u, \du,\nabla_h b)\|_{H_{tan}^{m-1}}^2 
			+ \var \|(\p_2 u, \p_3 u,\p_3 b)\|_{H_{tan}^{m-1}}^2 \\
			\lesssim &\; \sqrt{\me^m(t)}\md_{tan}^{m-1}(t),
			\dal \deq
			and
			\beq\label{3903} \bal
			&\; \frac{d}{dt}\Big(\|(\p_3 \vro, \p_3 u,\f{1}{\sqrt{1+\vro}} \p_3 b )(t)\|_{L^2}^2 + \sum\limits_{ |\alpha_h|=m-2}\|( \p_3 \p^{\al_h} \vro, \p_3 \p^{\al_h} u, \f{1}{\sqrt{1+\vro}}\p^{\al_h} \p_3 b)\|_{L^2}^2 \Big)\\
			&\; 
			+\|\p_3 (\p_{1} u,  \du , \nabla_h b)\|_{H_{tan}^{m-2}}^2  
			+\var \|\p_3 (\p_{2} u, \p_3 u , \p_3 b )\|_{H_{tan}^{m-2}}^2  \\
			\lesssim &\; \sqrt{\me^m(t)}\md_{tan}^{m-1}(t).
			\dal \deq
			From the estimates \eqref{3401} and \eqref{3601}, we have
			\beq\label{3904}
			\begin{aligned}
				&\; \f{d}{dt} \i \Big(u_h \cdot \nabla_h  \vro +  \p_{2} b \cdot u +  \sum_{|\ah|=m-2}( \p^{\ah}  u_h \cdot \nabla_h \p^{\ah} \vro +  \p_{2} \p^{\ah} b \cdot  \p^{\ah} u ) \Big)\ dx + \|(\nabla_h  \vro, \p_{2} u)\|_{H_{tan}^{m-2}}^2 \\
				\lesssim &\;\var \| (\p_2^2 u, \p_3^2 u)\|_{H_{tan}^{m-2}}^2 + 
				\|(\p_1 u, \du, \nabla_h b)\|_{H^{m-1}_{tan}}^2
				+\sqrt{\me^m(t)}\md_{tan}^{m-1}(t),
			\end{aligned}
			\deq
			and
			\beq\label{3905}
			\begin{aligned}
				&\;\f{d}{dt} \i  \Big( \nabla_h u_3 \cdot \nabla_h  \p_3 \vro +\p_{23} b \cdot  \p_3 u + \sum_{|\ah|=m-3} (\p^{\ah} \nabla_h u_3 \cdot \nabla_h \p^{\ah} \p_3 \vro + \p^{\ah}\p_{23} b \cdot  \p^{\ah}\p_3 u) \Big) \ dx \\
				&\; +\|(\p_3  \nabla_h \vro, \p_{23}u)\|_{H_{tan}^{m-3}}^2\\
				\lesssim &\;  \var\| (\p_2^2 u, \p_3^2 u, \p_{23} u)\|_{H^{m-2}_{tan}}^2 +  \var\| \p_3^2 b\|_{H^{m-3}_{tan}}^2  +\|  (\nabla_h u, \p_3 \du, \p_1 \nabla u,\nabla_h \nabla b)\|_{H^{m-2}_{tan}}^2 
				+\sqrt{\me^m(t)}\md_{tan}^{m-1}(t).
			\end{aligned}
			\deq
			Then, for some small positive suitable constant $\kappa>0$,
			we can deduce from estimates \eqref{3902}-\eqref{3905} that
			\beqq
			\begin{aligned}
				&\frac{d}{dt}\widetilde{\mathcal{E}}^{m-1}_{tan}(t)
				+2\kappa \md_{tan}^{m-1}(t)
				\lesssim \sqrt{\me^m(t)}\md_{tan}^{m-1}(t),
			\end{aligned}
			\deqq
			where the quantity $\widetilde{\mathcal{E}}^{m-1}_{tan}(t)$ is defined as
			\begin{align*}
				\widetilde{\mathcal{E}}^{m-1}_{tan}(t)
				:=&\| (\vro, u, \f{1}{\sqrt{1+\vro}} b)\|_{L^2}^2 +\|(\p_3 \vro, \p_3 u,\f{1}{\sqrt{1+\vro}} \p_3 b )\|_{L^2}^2 + \sum\limits_{|\alpha_h| = m-1}\| (\p^{\ah} \vro, \p^{\ah} u ,\f{1}{\sqrt{1+\vro}}\p^{\al_h} b)\|_{L^2}^2 \\
				&\; + \sum\limits_{ |\alpha_h|=m-2}\|( \p_3 \p^{\al_h} \vro, \p_3 \p^{\al_h} u, \f{1}{\sqrt{1+\vro}}\p_3 \p^{\al_h} b)\|_{L^2}^2 + 2\kappa \i ( u_h \cdot \nabla_h  \vro +  \p_{2} b \cdot u )  \ dx\\
				&\; + 2\kappa   \sum_{|\ah|=m-2} \i ( \p^{\ah}  u_h \cdot \nabla_h \p^{\ah} \vro +  \p_{2} \p^{\ah} b \cdot  \p^{\ah} u ) \ dx + 2\kappa \i (\nabla_h u_3 \cdot \nabla_h  \p_3 \vro +\p_{23} b \cdot  \p_3 u ) \ dx \\
				&\; + 2\kappa \sum_{|\ah|=m-3} \i (\p^{\ah} \nabla_h u_3 \cdot \nabla_h \p^{\ah} \p_3 \vro + \p^{\ah}\p_{23} b \cdot  \p^{\ah}\p_3 u) \ dx. 
			\end{align*}
			Due to the smallness of $\kappa$, it is easy to check that
			$\widetilde{\mathcal{E}}^{m-1}_{tan}(t)$ is equivalent to $\mathcal{E}^{m-1}_{tan}(t)$.
			Thus, due to the a priori assumption \eqref{assumption}, we have
			\beq\label{3906}
			\frac{d}{dt}\widetilde{\mathcal{E}}^{m-1}_{tan}(t)
			+\kappa \md_{tan}^{m-1}(t)
			\le 0.
			\deq
			On the other hand, due to the inequality
			\beqq
			\|(\vro, u, b, \wr, \wu, \wb)\|_{L^2}
			\lesssim \|\hs(\vro, u, b, \wr, \wu, \wb)\|_{L^2}^{\frac{1}{1+s}}
			\|\nabla_h(\vro, u, b, \wr, \wu, \wb)\|_{L^2}^{\frac{s}{1+s}},
			\deqq
			it is easy to check that
			\beq\label{3907}
			\begin{aligned}
				\widetilde{\mathcal{E}}^{m-1}_{tan}(t)
				\lesssim
				\me_{tan}^{m-1}(t)
				\lesssim
				&(\|\hs(\vro, u, b, \wr, \wu, \wb)\|_{L^2}^2
				+\|\nabla_h(\vro ,u, b)\|_{H^{m-2}_{tan}}^2
				+\|\nabla_h( \wr, \wu, \wb)\|_{H^{m-3}_{tan}}^2)^{\frac{1}{1+s}}\\
				& \times(\|\nabla_h(\vro, u, b)\|_{H^{m-2}_{tan}}^2
				+\|\nabla_h(\wr, \wu, \wb)\|_{H^{m-3}_{tan}}^2)^{\frac{s}{1+s}}\\
				\lesssim
				&C_0^{\frac{1}{1+s}} \md_{tan}^{m-1}(t)^{\frac{s}{1+s}},
			\end{aligned}
			\deq
			where we have used the estimate \eqref{39-em} in the last inequality.
			The combination of \eqref{3906} and \eqref{3907} yields
			\beqq
			\frac{d}{dt}\widetilde{\mathcal{E}}^{m-1}_{tan}(t)
			+\kappa C_0^{-\frac{1}{s}}
			\widetilde{\mathcal{E}}^{m-1}_{tan}(t)^{1+\frac{1}{s}}\le 0,
			\deqq
			which yields directly the decay estimate
			\beq\label{3908}
			\mathcal{E}^{m-1}_{tan}(t)
			\lesssim
			\widetilde{\mathcal{E}}^{m-1}_{tan}(t)
			\lesssim C_0(1+t)^{-s}.
			\deq
			Therefore, we complete the proof of this lemma.
		\end{proof}
		
		Finally, we will establish the time integration
		of  $\md_{tan}^{m-1}(t)$ with the suitable weight $(1+\tau)^{\sigma}$.
		\begin{lemm}
			Under the assumption \eqref{assumption},
			the smooth solution $(\vro ,u, b)$ of equation \eqref{eqr} has the estimate
			\beq\label{31001}
			\begin{aligned}
				&(1+t)^{\sigma}{\mathcal{E}}^{m-1}_{tan}(t)
				+\kappa \int_0^t (1+\tau)^{\sigma} \md_{tan}^{m-1}(\tau)d\tau
				\lesssim C_0,
			\end{aligned}
			\deq
			where the constant $C_0$ is defined in \eqref{co}.
		\end{lemm}
		\begin{proof}
			For $0<\sigma<s$, multiplying \eqref{3906} by $(1+t)^{\sigma}$, we have
			\beqq
			\begin{aligned}
				&\frac{d}{dt}[(1+t)^{\sigma}\widetilde{\mathcal{E}}^{m-1}_{tan}(t)]
				+\kappa (1+t)^{\sigma} \md_{tan}^{m-1}(t)
				\le \sigma(1+t)^{\sigma-1}\widetilde{\mathcal{E}}^{m-1}_{tan}(t).
			\end{aligned}
			\deqq
			Integrating the above inequality over $[0, t]$
			and using the uniform estimate \eqref{3908}, we have
			\beqq
			\begin{aligned}
				&\; (1+t)^{\sigma}\widetilde{\mathcal{E}}^{m-1}_{tan}(t)
				+\kappa \int_0^t  (1+\tau)^{\sigma} \md_{tan}^{m-1}(\tau)d\tau\\
				\le
				&\; \widetilde{\mathcal{E}}^{m-1}_{tan}(0)
				+\sigma\int_0^t (1+\tau)^{\sigma-1}
				\widetilde{\mathcal{E}}^{m-1}_{tan}(\tau)d\tau\\
				\le
				&\; \widetilde{\mathcal{E}}^{m-1}_{tan}(0)
				+\sigma\underset{0\le \tau \le t}{\sup}
				\left[\widetilde{\mathcal{E}}^{m-1}_{tan}(\tau)(1+\tau)^s\right]
				\int_0^t (1+\tau)^{\sigma-1-s}d\tau\\
				\lesssim
				&\; C_0.
			\end{aligned}
			\deqq
			Therefore, we complete the proof of this lemma.
		\end{proof}
		
		\subsection{Global in time uniform regularity}
		
		In this subsection, we will give the proof of Proposition \ref{main_pro}.
		Indeed,  the combination of estimates \eqref{3901},\eqref{39-vroub-m} and \eqref{31001}
		yields directly
		\beqq \bal
		&\; \underset{0\le \tau \le t}{\sup}[(1+\tau)^{s}{\mathcal{E}}^{m-1}_{tan}(\tau)]
		+\int_0^t (1+\tau)^{\sigma} \md_{tan}^{m-1}(\tau)d\tau \\
		\le &\;  C(\|(\vro, u,  b)(0)\|_{H^m}^2
		+\|(\hs( \vro, u, b), \hs \p_3 ( \vro, u, b))(0)\|_{L^2}^2+ \|\hs u(0)\|_{L^2}\|(\vro, u)(0)\|_{H^m}^2+\delta^{\f32}),
		\dal \deqq
		and
		\beqq
		\underset{0\le \tau \le t}{\sup}\mathcal{E}^m(\tau)
        +\int_0^t \md^{m}(\tau)  d\tau
		\le C\mathcal{E}^m(0)+C\delta^{\frac32},
		\deqq
		where the constants $(s, \sigma)$ satisfy $\frac{9}{10}<\sigma<s<1$.
		Then, we can obtain the estimate
		\beqq
		\begin{aligned}
			&\;	\underset{0\le \tau \le t}{\sup}\mathcal{E}^m(\tau)
			+\underset{0\le \tau \le t}{\sup}[(1+\tau)^s\mathcal{E}_{tan}^{m-1}(\tau)]
			+\int_0^t (1+\tau)^{\sigma} \md_{tan}^{m-1}(\tau)d\tau
			+\int_0^t \md^{m}(\tau)  d\tau\\
			\le
			&\; 2C(\|(\vro, u,  b)(0)\|_{H^m}^2
			+\|(\hs( \vro, u, b), \hs \p_3 ( \vro, u, b))(0)\|_{L^2}^2+ \|\hs u(0)\|_{L^2}\|(\vro, u)(0)\|_{H^m}^2)
			+2C\delta^{\frac32}.
		\end{aligned}
		\deqq
		Now choose the small constant
		\beqq\label{choose_delta} \bal
		\delta:=&\;8C(\|(\vro, u,  b)(0)\|_{H^m}^2
		+\|(\hs( \vro, u, b), \hs \p_3 ( \vro, u, b))(0)\|_{L^2}^2+ \|\hs u(0)\|_{L^2}\|(\vro, u)(0)\|_{H^m}^2)\\
		\le &\; \min\{1,\frac{1}{64C^2}\},
		\dal \deqq
		then we have
		\beqq
		\begin{aligned}
			&\underset{0\le \tau \le t}{\sup}\mathcal{E}^m(\tau)
			+\underset{0\le \tau \le t}{\sup}[(1+\tau)^s\mathcal{E}_{tan}^{m-1}(\tau)]
			+\int_0^t (1+\tau)^{\sigma} \md_{tan}^{m-1}(\tau)d\tau
			+\int_0^t \md^{m}(\tau)  d\tau\\
			\le
			&2C(\|(\vro, u,  b)(0)\|_{H^m}^2
			+\|(\hs( \vro, u, b), \hs \p_3 ( \vro, u, b))(0)\|_{L^2}^2+ \|\hs u(0)\|_{L^2}\|(\vro, u)(0)\|_{H^m}^2)
			+2C\delta^{\frac32}\\
			\le
			&\frac{\delta}{4}+\frac{\delta}{4}=\frac{\delta}{2},
		\end{aligned}
		\deqq
		which implies the estimate \eqref{close_assumption}.
		Therefore, we complete the proof of Proposition \ref{main_pro}.

		\begin{proof}[\textbf{Proof of Theorem \ref{main_result_one}}]
			Suppose the assumptions in Theorem \ref{main_result_one} hold,
			one can establish the uniform local-in-time well-posedness for the equation \eqref{eqr}.
			Next, we use the standard continuity argument to show the global well-posedness.
			From the local existence result
			and smallness assumption of initial condition,
			it holds
			\beqq
			\underset{0\le \tau \le t}{\sup}\mathcal{E}^m(\tau)
			+\underset{0\le \tau \le t}{\sup}[(1+\tau)^s\mathcal{E}_{tan}^{m-1}(\tau)]
			+\int_0^t (1+\tau)^{\sigma} \md_{tan}^{m-1}(\tau)d\tau
			+\int_0^t \md^{m}(\tau)d\tau \le \delta,
			\deqq
			for all $t\in [0, T_0)$ and $$\delta=8C(\|(\vro, u,  b)(0)\|_{H^m}^2
			+\|(\hs( \vro, u, b), \hs \p_3 ( \vro, u, b))(0)\|_{L^2}^2+ \|\hs u(0)\|_{L^2}\|(\vro, u)(0)\|_{H^m}^2)$$
			and $C$ is a positive constant independent of time $t$ and
			parameter $\var$.
			Set
			\beq\label{criterion}
            \begin{aligned}
			T_1:=\underset{T_0}{\sup}
			&\left\{T_0~|
			\underset{0\le \tau \le t}{\sup}\mathcal{E}^m(\tau)
			+\underset{0\le \tau \le t}{\sup}[(1+\tau)^s\mathcal{E}_{tan}^{m-1}(\tau)]\right.\\
			&\quad\quad\quad\left.+\!\int_0^t \!(1+\tau)^{\sigma} \md_{tan}^{m-1}(\tau)d\tau
			+\!\int_0^t \! \md^{m}(\tau)d\tau \le \delta,
			\quad \forall ~ t\in [0, T_0)\right\},
            \end{aligned}
			\deq
			we claim that $T_1=+\infty$. Otherwise, applying the estimate
			\eqref{close_assumption}
			and the local-in-time existence result,
			there exists a positive constant $T_2$ such that $T_2>T_1$,
			it holds that for any $T\in [T_1, T_2)$,
			\beqq
			\underset{0\le \tau \le t}{\sup}\mathcal{E}^m(\tau)
			+\underset{0\le \tau \le t}{\sup}[(1+\tau)^s\mathcal{E}_{tan}^{m-1}(\tau)]
			+\int_0^t (1+\tau)^{\sigma} \md_{tan}^{m-1}(\tau)d\tau
			+\int_0^t \md^{m}(\tau)d\tau \le \delta,
			\quad \forall ~ t\in [0, T).
			\deqq
			This contradicts the definition of $T_1$ in \eqref{criterion}.
			Therefore, we can deduce that $T_1=+\infty$.
			Therefore, we complete the proof of Theorem \ref{main_result_one}.			
		\end{proof}
		
		\section{Convergence rate of solution}\label{asymptotic-behavior}
		In this section, we will establish the convergence rate of the solutions between equations \eqref{eqr}  and \eqref{eqr0}.
		Let the regularity index $m \geq 9$, Theorems \ref{main_result_one} and \ref{main_result_two} make sure that the solution $( \vro^{\var},u^{\var}, b^{\var})$ of equation \eqref{eqr} and  $( \vro^{0},u^{0}, b^{0})$ of equation \eqref{eqr0} exist
		globally in time. This will help us establish the convergence rate independent of time.
		Let us define $(\br,\bu, \bb):=(\vro^\var-\vro^0 , u^\var-u^0, b^\var-b^0)$
		and apply the equations \eqref{eqr} and \eqref{eqr0},
		then we can obtain
		\begin{equation}\label{401}
			\left\{
			\begin{array}{*{4}{ll}}
				\p_t \br + {\rm{div}} \bu = f_{\vro},\\
				\p_t \bu-\p_{1}^2 \bu- \nabla {\rm{div}} \bu + \nabla \br + \nabla \bb_2 -\p_2 \bb
				=\var \f{1}{1+\vro^{\var}} (\p_2^2 u^{\var} +\p_3^2 u^{\var} )+f_{u},\\
				\p_t \bb-\Delta_h \bb+e_2 {\rm{div}} \bu-\p_2 \bu
				=\var \p_3^2 b^{\var} + f_{b},
			\end{array}
			\right.
		\end{equation}
		where the source terms $f_{\vro}$, $f_{u}$ and $f_{b}$ are defined by
		\beqq
		\begin{aligned}
			f_{\vro}:=& -u^\var \cdot \nabla \br-\bu \cdot \nabla \vro^0-\vro^\var \, {\rm{div}} \bu-\br \, {\rm{div}} u^0,\\
			f_{u}:=&-u^\var \cdot \nabla \bu-\bu \cdot \nabla u^0-\vro^\var \, \nabla \br-\br \, \nabla \vro^0
			-\f{\vro^{\var}}{1+\vro^{\var}}(\p_1^2 \bu+\nabla {\rm{div}} \bu + \nabla \bb_2 -\p_2 \bb)\\
			&-\f{\br}{(1+\vro^{\var})(1+\vro^{0})}(\p_1^2 u^0+\nabla {\rm{div}} u^0 + \nabla b^0_2 -\p_2 b^0) + \f{1}{1+\vro^{\var}}(b^{\var} \cdot \nabla \bb -b^{\var} \, \nabla \bb) \\
			&+ \f{1}{1+\vro^{\var}}(\bb \cdot \nabla b^0 -\bb \, \nabla b^0) - \f{\br}{(1+\vro^{\var})(1+\vro^{0})}(b^0 \cdot \nabla b^0 -b^0 \, \nabla b^0) ,\\
			f_{b}:=&-u^\var \cdot \nabla \bb-\bu \cdot \nabla b^0-b^\var \, {\rm{div}} \bu-\bb \, {\rm{div}} u^0+b^\var \cdot \nabla \bu+\bb\cdot \nabla u^0.
		\end{aligned}
		\deqq
		Let us define the notations
		\beqq
		\begin{aligned}
			\mathcal{\overline{E}}(t)
			:=\|(\br,\bu, \bb)(t)\|_{H^1}^2,~~~~~
			\mathcal{\overline{D}}(t)
			:=\|(\p_1 \bu, {\rm{div}}\bu, \nabla_h \bb)(t)\|_{H^1}^2
			+\|(\p_2 \bu, \nabla_h \br)(t)\|_{L^2}^2,
		\end{aligned}
		\deqq
		and denote that $f(t):=(\vro^0, u^0, b^0, \vro^{\var},u^\var, b^\var)(t)$ and $g(t):=(\du^0,\du^{\var})(t)$, we define
		\beqq \bal
		\mathcal{{A}}(t)
		:= &\; \|f(t)\|_{H^1} \|\nabla_h f(t)\|_{H^1} + \|\nabla_h f(t)\|_{H^2} \|\nabla_h f(t)\|_{H^3}  + \|f(t)\|_{L^2} \|g(t)\|_{L^2} \\
		&\; + \|\p_3 f(t)\|_{L^2}^{\f14} \|\p_3^2 f(t)\|_{L^2}^{\f14} \|\p_3 \nabla_h f(t)\|_{H_{tan}^1}^{\f34}  \|\p_3^2 \nabla_h f(t)\|_{H_{tan}^1}^{\f34} + \|g(t)\|_{H^3}^2, \\
		\mathcal{{B}}(t)
		:= &\; \|\p_3^2 f(t)\|_{L^2}^{\f45}  \|\nabla_h \p_3^2 f(t)\|_{L^2}^{\f45} +  \|\nabla_h(f, \p_3 f)(t)\|_{H_{tan}^3}^{\f43} + \|\nabla_h(g, \p_3 g)(t)\|_{H_{tan}^2}^{\f43} \\
		&\; + \| \p_3 g(t)\|_{L^2}^{\f13} \| \p_3^2 g(t)\|_{L^2}^{\f13}\| \nabla_h \p_3 g(t)\|_{L^2}^{\f13}\| \nabla_h \p_3^2 g(t)\|_{L^2}^{\f13}.
		\dal \deqq
		First of all, let us establish the estimate for $(\br,\bu, \bb)$ in $L^2$ norm.
		\begin{lemm}\label{lemma41}
			Under the conditions of Theorem \ref{main_result_three},
			then it holds
			\beq\label{4201}
			\begin{aligned}
				\frac{d}{dt}\|(\br, \bu,\f{1}{1+\vro^{\var} } \bb)\|_{L^2}^2 +\|(\p_1 \bu, {\rm{div}} \bu, \nabla_h \bb)\|_{L^2}^2
				\lesssim
				\nu \mathcal{\overline{D}}(t)
				+ \mathcal{{A}}(t)
				\mathcal{\overline{E}}(t)
				+\var \|(\p_2^2u^\var, \p_3^2 u^\var, \p_3^2b^\var)\|_{L^2}
				\mathcal{\overline{E}}(t)^{\frac12},
			\end{aligned}
			\deq
			where $\nu$ is a small positive constant.
		\end{lemm}
		\begin{proof}
			The equation \eqref{401} yields that 
			\begin{align*}
				&\; \frac{d}{dt}\|(\bu, \br, \f{1}{1+\vro^{\var} } \bb)\|_{L^2}^2 +\|(\p_1 \bu, {\rm{div}} \bu, \nabla_h \bb)\|_{L^2}^2 \\
				= &\; \i f_\vro \cdot \br \ dx +  \i f_u \cdot \bu \ dx+ \i \f{1}{1+\vro^{\var}} f_b \cdot \bb \ dx
				+ \i \f{1}{2(1+\vro^{\var})^2} |\bb|^2  {\rm{div}}\bu \ dx \\
				&\;- \i \f{1}{2(1+\vro^{\var})^2} |\bb|^2  f_{\vro} \ dx-\i \f{\vro^{\var}}{1+\vro^{\var}} \Delta_h \bb \cdot \bb \ dx+\i \f{\vro^{\var}}{1+\vro^{\var}} (e_2 {\rm{div}}\bu -\p_2 \bu) \cdot \bb \ dx  \\
				&\;+\var \i \f{1}{1+\vro^{\var}} (\p_2^2 u^{\var} + \p_3^2 u^{\var} ) \cdot \bu \ dx + \var \i \f{1}{1+\vro^{\var}} \p_3^2 b^{\var} \cdot \bb \ dx 
				:=\; \sum_{i=1}^{9} IX_{i}.
			\end{align*}
			Using the anisotropic type inequality $\eqref{ie:Sobolev}_2$, it is easy to check that
			\beqq \bal
			IX_4 \lesssim &\; \| {\rm{div}}\bu \|_{L^2} \| \bb\|_{L^2}^{\f12} \| \p_1 \bb\|_{L^2}^{\f12} \| \bb\|_{L^2}^{\f14} \| \p_3 \bb\|_{L^2}^{\f14} \| \p_2 \bb\|_{L^2}^{\f14} \| \p_{23} \bb\|_{L^2}^{\f14} \\
			\lesssim &\; \nu \|({\rm{div}}\bu , \p_1 \bb)\|_{L^2}^2 + \| \bb\|_{L^2}^2 \| ( b^0, b^{\var}, \p_3 b^0, \p_3 b^{\var} )\|_{L^2} \| \p_2( b^0, b^{\var}, \p_3 b^0, \p_3 b^{\var} )\|_{L^2},\\
			IX_6 = &\;  \i \nabla_h \bb \cdot \left(\nabla_h (\f{\vro^{\var}}{1+\vro^{\var}})  \bb + \f{\vro^{\var}}{1+\vro^{\var}} \nabla_h \bb\right) \ dx\\
			\lesssim &\; \| \nabla_h \bb \|_{L^2} \| \bb\|_{L^2}^{\f12} \| \p_1 \bb\|_{L^2}^{\f12} \| \nabla_h \vro^{\var} \|_{L^2}^{\f14} \| \p_3 \nabla_h \vro^{\var}\|_{L^2}^{\f14} \| \p_2 \nabla_h \vro^{\var}\|_{L^2}^{\f14} \| \p_{23} \nabla_h \vro^{\var} \|_{L^2}^{\f14}  
			+ \| \vro^{\var} \|_{L^{\infty}} \| \nabla_h \bb \|_{L^2}^2 \\
			\lesssim &\; \nu \|\nabla_h \bb\|_{L^2}^2 + \| \bb\|_{L^2}^2 \| \nabla_h \vro^{\var}\|_{H_{tan}^1}^2  \| \p_3 \nabla_h \vro^{\var}\|_{H_{tan}^1}^2,\\
			IX_7 \lesssim &\;  \|( {\rm{div}}\bu, \p_2 \bu)\|_{L^2}  \| \bb\|_{L^2}^{\f12} \| \p_2 \bb\|_{L^2}^{\f12} \|  \vro^{\var} \|_{L^2}^{\f14} \| \p_3  \vro^{\var}\|_{L^2}^{\f14} \| \p_1 \vro^{\var}\|_{L^2}^{\f14} \| \p_{13} \vro^{\var} \|_{L^2}^{\f14}  \\
			\lesssim &\; \nu  \|( {\rm{div}}\bu, \p_2 \bu, \p_2 \bb)\|_{L^2}^2 + \| \bb\|_{L^2}^2 \| ( \vro^{\var}, \p_3  \vro^{\var} )\|_{L^2}^2  \| \p_1( \vro^{\var}, \p_3  \vro^{\var} )\|_{L^2}^2, 
			\dal \deqq
			and
			\beqq \bal
			IX_8 + IX_9 \lesssim &\; \var \| (\p_2^2 u^{\var}, \p_3^2 u^{\var})\|_{L^2} \| \bu\|_{L^2} +  \var \| \p_3^2 b^{\var}\|_{L^2} \| \bb\|_{L^2}.
			\dal \deqq 
			Now we estimate the term $IX_5$.  Using the anisotropic type inequality $\eqref{ie:Sobolev}_3$, we have
			\beqq \bal
			IX_5 = &\; \i \f{1}{2(1+\vro^{\var})^2} |\bb|^2  (u^\var \cdot \nabla \br+\bu \cdot \nabla \vro^0+\vro^\var \, {\rm{div}} \bu+\br \, {\rm{div}} u^0) \ dx \\
			\lesssim &\; \| \bb\|_{L^2} \| \nabla_h \bb\|_{L^2} \big( \| \nabla_h \br\|_{L^2}^{\f12} \| \p_3 \nabla_h \br\|_{L^2}^{\f12} \| u_h^{\var}\|_{L^{\infty}}
			+\| u_3^{\var} \|_{L^2}^{\f12} \| \p_3 u_3^{\var} \|_{L^2}^{\f12} \| \p_3 \br\|_{L^{\infty}} \\
			&\; + \| \bu_3 \|_{L^2}^{\f12} \| \p_3 \bu_3 \|_{L^2}^{\f12} \|\p_3 \vro^0\|_{L^{\infty}} + \| \nabla_h \vro^0 \|_{L^2}^{\f12} \| \p_3 \nabla_h \vro^0 \|_{L^2}^{\f12} \|\bu_h\|_{L^{\infty}} \\
			&\; + \| {\rm{div}} \bu\|_{L^2}^{\f12} \| \p_3 {\rm{div}} \bu \|_{L^2}^{\f12} \|\vro^{\var}\|_{L^{\infty}}  +  \| \du^0 \|_{L^2}^{\f12} \| \p_3 \du^0 \|_{L^2}^{\f12} \|\br\|_{L^{\infty}}
			\big)\\
			\lesssim &\; \nu\| \nabla_h \bb\|_{L^2}^2 + {\mathcal{A}}(t) \|\bb\|_{L^2}^2 . 
			\dal \deqq
			Next we estimate the term $IX_1$. Integrating by parts, we have
			\beqq \bal
			IX_1 = &\;- \i    (u^\var \cdot \nabla \br+\bu \cdot \nabla \vro^0+\vro^\var \, {\rm{div}} \bu+\br \, {\rm{div}} u^0) \cdot \br \ dx \\
			=&\;  \i ( \f12 \du^{\var} - \du^0)  \br^2  \ dx 
			-\i (\bu \cdot \nabla \vro^0+\vro^\var \, {\rm{div}} \bu) \cdot \br \ dx\\
			\lesssim &\; \| \br\|_{L^2} \| \nabla_h \br\|_{L^2} \|(\du^{\var},\du^0 )\|_{L^2}^{\f12} \|\p_3 (\du^{\var},\du^0 )\|_{L^2}^{\f12} \\
			&\; + \| \br\|_{L^2}^{\f12} \| \p_2 \br\|_{L^2}
			\big(\| \bu_h \|_{L^2}^{\f12} \| \p_1 \bu_h \|_{L^2}^{\f12} \| \nabla_h \vro^0 \|_{L^2}^{\f12} \| \p_3 \nabla_h \vro^0 \|_{L^2}^{\f12}+ \| \bu_3 \|_{L^2}^{\f12} \| \p_3 \bu_3 \|_{L^2}^{\f12} \|\p_3 \vro^0\|_{L^2}^{\f12} \|\p_{13} \vro^0\|_{L^2}^{\f12}   \\
			&\; 
			+ \| {\rm{div}} \bu\|_{L^2}^{\f12} \| \p_3 {\rm{div}} \bu \|_{L^2}  \|  \vro^{\var} \|_{L^2}^{\f14} \| \p_3  \vro^{\var}\|_{L^2}^{\f14} \| \p_1 \vro^{\var}\|_{L^2}^{\f14} \| \p_{13} \vro^{\var} \|_{L^2}^{\f14} 
			\big)\\
			\lesssim &\; \nu\| (\nabla_h \br,{\rm{div}} \bu , \nabla_h \bu)\|_{L^2}^2 + {\mathcal{A}}(t) \|(\br, \bu)\|_{L^2}^2. 
			\dal \deqq
			Then we estimate the term $IX_2$.
			\beqq \bal
			IX_2 = &\; \i ( -u^\var \cdot \nabla \bu-\bu \cdot \nabla u^0-\vro^\var \, \nabla \br-\br \, \nabla \vro^0) \cdot \bu \ dx\\
			&
			- \i \f{\vro^{\var}}{1+\vro^{\var}}(\p_1^2 \bu+\nabla {\rm{div}} \bu + \nabla \bb_2 -\p_2 \bb) \cdot \bu \ dx \\
			&- \i \f{\br}{(1+\vro^{\var})(1+\vro^{0})}(\p_1^2 u^0+\nabla {\rm{div}} u^0 + \nabla b^0_2 -\p_2 b^0) \cdot \bu \ dx \\
			&\; + \i \big(-\f{1}{1+\vro^{\var}}b^{\var} \, \nabla \bb + \f{1}{1+\vro^{\var}}(\bb \cdot \nabla b^0- \bb \, \nabla b^0) - \f{\br}{(1+\vro^{\var})(1+\vro^{0})}(b^0 \cdot \nabla b^0 -b^0 \, \nabla b^0)  \big)  \cdot \bu \ dx \\
			&\; +   \i  \f{1}{1+\vro^{\var}}b^{\var} \cdot \nabla \bb  \cdot \bu \ dx
			:= \sum_{i=1}^{5} IX_{2,i}.
			\dal \deqq
			Integrating by parts, we have
			\beqq \bal
			IX_{2,1} = &\;\f12 \i  \bu^2 \du^{\var}   \ dx + \i \br(  {\rm{div}} \bu \, \vro^{\var} + 
			\bu \cdot \nabla \vro^{\var}) \ dx  - \i ( \bu \cdot \nabla u^0+\br \, \nabla \vro^0) \cdot \bu \ dx \\
			\lesssim &\; \| \bu\|_{L^2} \| \nabla_h \bu\|_{L^2} \|\du^{\var} \|_{L^2}^{\f12} \|\p_3 \du^{\var}\|_{L^2}^{\f12} + \| \br\|_{L^2}^{\f12} \| \p_2 \br\|_{L^2}^{\f12} ( \| \bu_h \|_{L^2}^{\f12} \| \p_1 \bu_h \|_{L^2}^{\f12} \| \nabla_h \vro^{\var} \|_{L^2}^{\f12} \| \p_3 \nabla_h \vro^{\var} \|_{L^2}^{\f12}
			\\ &\; + \| \bu_3 \|_{L^2}^{\f12} \| \p_3 \bu_3 \|_{L^2}^{\f12} \|\p_3 \vro^{\var}\|_{L^2}^{\f12} \|\p_{13} \vro^{\var}\|_{L^2}^{\f12}  
			+\| {\rm{div}} \bu \|_{L^2} \|  \vro^{\var} \|_{L^2}^{\f14} \| \p_3  \vro^{\var}\|_{L^2}^{\f14} \| \p_1 \vro^{\var}\|_{L^2}^{\f14} \| \p_{13} \vro^{\var} \|_{L^2}^{\f14}) \\
			&\;+ \| \bu\|_{L^2}^{\f12} \| \p_1 \bu\|_{L^2}^{\f12} (\| \bu_h \|_{L^2}^{\f12} \| \p_2 \bu_h \|_{L^2}^{\f12} \| \nabla_h u^0 \|_{L^2}^{\f12} \| \p_3 \nabla_h u^0 \|_{L^2}^{\f12} + \| \bu_3 \|_{L^2}^{\f12} \| \p_3 \bu_3 \|_{L^2}^{\f12} \|\p_3 u^0\|_{L^2}^{\f12} \|\p_{23} u^0\|_{L^2}^{\f12} )\\
			&\; + \| \br\|_{L^2}^{\f12} \| \p_2 \br\|_{L^2}^{\f12} (\| \bu_h \|_{L^2}^{\f12} \| \p_1 \bu_h \|_{L^2}^{\f12} \| \nabla_h \vro^0 \|_{L^2}^{\f12} \| \p_3 \nabla_h  \vro^0 \|_{L^2}^{\f12} + \| \bu_3 \|_{L^2}^{\f12} \| \p_3 \bu_3 \|_{L^2}^{\f12} \|\p_3  \vro^0\|_{L^2}^{\f12} \|\p_{13}  \vro^0\|_{L^2}^{\f12} ) \\
			\lesssim  &\; \nu \|(\p_2 \br, {\rm{div}} \bu, \nabla_h \bu) \|_{L^2}^2 + {\mathcal{A}}(t) \|(\br, \bu)\|_{L^2}^2,
			\dal \deqq
			and similarly, we have
			\beqq \bal
			IX_{2,2}= &\; \i \p_1 \bu \cdot \p_1 ( \f{\vro^{\var}}{1+\vro^{\var}} \,\bu) \, dx + \i  {\rm{div}} \bu \,  {\rm{div}} ( \f{\vro^{\var}}{1+\vro^{\var}} \,\bu) \, dx\\
            &\; - \i   \bb_2 \, {\rm{div}} ( \f{\vro^{\var}}{1+\vro^{\var}} \,\bu)  \ dx 
			   + \i \f{\vro^{\var}}{1+\vro^{\var}} \p_2 \bb \cdot \bu \ dx \\
			\lesssim &\; \nu \|(\p_2 \bb, {\rm{div}} \bu, \nabla_h \bu) \|_{L^2}^2 + {\mathcal{A}}(t) \|(\bb, \bu)\|_{L^2}^2,\\
			IX_{2,3} \lesssim &\; \nu \|(\p_2 \br, {\rm{div}} \bu, \nabla_h \bu) \|_{L^2}^2 + {\mathcal{A}}(t) \|(\br, \bu)\|_{L^2}^2,\\
			IX_{2,4} = &\; \i \big( \f{1}{1+\vro^{\var}}(\bb \cdot \nabla b^0 - \bb \, \nabla b^0) - \f{\br}{(1+\vro^{\var})(1+\vro^{0})}(b^0 \cdot \nabla b^0 -b^0 \, \nabla b^0)  \big)  \cdot \bu \ dx \\
			&\; +\i   \bb \,{\rm{div}} ( \f{b^{\var}}{1+\vro^{\var}} \,\bu)  \ dx \\
			\lesssim  &\; \nu \|(\nabla_h \bb,\p_2 \br, {\rm{div}} \bu, \nabla_h \bu) \|_{L^2}^2 + {\mathcal{A}}(t) \|(\br, \bu,  \bb )\|_{L^2}^2.
			\dal \deqq
			It remains to deal with the term $IX_{2,5}$, which will be estimated together with $IX_3$ as follows:
			\beqq  \bal
			 IX_{2,5} + IX_3 
			= &\;  \i \f{1}{1+\vro^{\var}} (b^\var \cdot \nabla \bb \cdot \bu + b^\var \cdot \nabla \bu \cdot \bb) \ dx -   \i \f{1}{1+\vro^{\var}} u^\var \cdot \nabla \bb \cdot \bb\ dx\\
			&\; +   \i \f{1}{1+\vro^{\var}} (-\bu \cdot \nabla b^0-b^\var \, {\rm{div}} \bu-\bb \, {\rm{div}} u^0+\bb\cdot \nabla u^0) \cdot \bb \ dx\\
			=  &\; \i  \bb \cdot \bu \,  {\rm{div}} (\f{1}{1+\vro^{\var}} b^\var ) \ dx+ \f12 \i  |\bb|^2  {\rm{div}} (\f{1}{1+\vro^{\var}} u^\var ) \ dx\\
			&\; + \i \f{1}{1+\vro^{\var}} (-\bu \cdot \nabla b^0-b^\var \, {\rm{div}} \bu-\bb \, {\rm{div}} u^0+\bb\cdot \nabla u^0) \cdot \bb \ dx \\
			\lesssim &\; \nu \|(\nabla_h \bb,{\rm{div}} \bu, \nabla_h \bu) \|_{L^2}^2 + {\mathcal{A}}(t) \|(\bu,  \bb)\|_{L^2}^2,
			\dal \deqq
			which yields that
			\beqq
			IX_2 + IX_3 \lesssim \nu \|(\nabla_h \bb, \p_2 \br, {\rm{div}} \bu, \nabla_h \bu) \|_{L^2}^2 + {\mathcal{A}}(t) \|(\br, \bu,\bb)\|_{L^2}^2.
			\deqq
			Combining the estimates from $IX_1$ to $IX_9$, we finish the proof of this lemma.
		\end{proof}
		
		Next, we will establish the dissipation estimate for the quantity $(\p_2 \bu, \nabla_h \br)$ in $L^2-$norm.
		\begin{lemm}
			Under the conditions of Theorem \ref{main_result_three},
			then it holds
			\beq\label{4301}
			\begin{aligned}
				&\;\frac{d}{dt}\i  ( \p_2 \bb \cdot  \bu + \bu_h  \cdot   \nabla_h \br) \ dx+\|(\p_2 \bu,\nabla_h \br)\|_{L^2}^2 \\
				\lesssim
				&\;\|(\p_1 \bu, \p_1^2 \bu, {\rm{div}} \bu, \nabla {\rm{div}} \bu,\nabla_h \bb, \Delta_h \bb)\|_{L^2}^2
				+ \mathcal{{A}}(t)
				\mathcal{\overline{E}}(t)
				+\var \|(\p_2^2u^\var, \p_3^2 u^\var, \p_3^2 b^\var)\|_{L^2}
				\mathcal{\overline{E}}(t)^{\frac12}.
			\end{aligned}
			\deq
		\end{lemm}
		\begin{proof}
			Using the equation \eqref{401}, we have
			\begin{align*}
				&\;\frac{d}{dt}\i  ( \p_2 \bb \cdot  \bu + \bu_h  \cdot   \nabla_h \br) \ dx+\i (|\p_2 \bu|^2 +|\nabla_h \br|^2) dx\\
				= &\;  \i   \nabla_h \cdot  \bu_h \cdot ({\rm{div}}\bu-f_{\vro}) \ dx  - \i  f_b \cdot \p_{2} \bu   \ dx  + \i  ((f_u)_h  \cdot \nabla_h  \br +  f_u  \cdot 
				\p_{2} \bb )\ dx \\
				&\; + \i  (\p_1^2 \bu_h +\nabla_h {\rm{div}}\bu + \p_2 \bb_h - \nabla_h \bb_2) \cdot \nabla_h  \br\ dx + \i   \p_{2} \bu \cdot (e_2  {\rm{div}}\bu - \Delta_h \bb)  \ dx \\
				&\;+\i  ( \p_1^2 \bu + \nabla {\rm{div}}\bu + \p_2 \bb - \nabla \bb_2 -\nabla \br)  \cdot  \p_{2} \bb\ dx
				+ \var \i \p_3^2 b^{\var} \cdot \p_2  \bu  \ dx \\
				&\;  +\var \i \f{1}{1+ \vro^{\var}} \Big((\p_2^2 u_h^{\var} + \p_3^2 u_h^{\var}) \nabla_h  \br + (\p_2^2 u^{\var} + \p_3^2 u^{\var}) \p_2 \bb \Big)\ dx\\
				:= &\; \sum_{i=1}^8 X_{i}.
			\end{align*}
			It is easy to check that
			\beqq \bal
			X_4 \lesssim &\; \nu \| \nabla_h \br\|_{L^2}^2 + \| (\p_1^2 \bu_h , \nabla_h {\rm{div}}\bu, \nabla_h \bb)\|_{L^2}^2,\\
			X_5 \lesssim &\; \nu \| \p_2 \bu\|_{L^2}^2 + \| ({\rm{div}}\bu, \Delta_h \bb)\|_{L^2}^2,\\
			X_6 =&\;  \i  ( \p_1^2 \bu + \p_2 \bb  )\cdot  \p_{2} \bb  \ dx+ 
			\i\nabla(  {\rm{div}}\bu  - \bb_2 - \br)  \cdot  \p_{2} \bb\ dx 
			\lesssim  \| (\p_1^2 \bu , \p_2 \bb)\|_{L^2}^2,
			\dal \deqq
			and
			\beqq \bal
			X_7+ X_8 \lesssim \var \| ( \nabla_h \br, \p_2 \bb, \p_2 \bu)\|_{L^2} \|(\p_2^2 u^{\var},\p_3^2 b^{\var}, \p_3^2   u^{\var})\|_{L^2}.
			\dal \deqq
			Using the anisotropic type inequality $\eqref{ie:Sobolev}$ and H\"older inequality, we have\
			\beqq \bal
			X_1 = &\;  \i   \nabla_h \cdot  \bu_h \cdot ({\rm{div}}\bu+u^\var \cdot \nabla \br+\bu \cdot \nabla \vro^0+\vro^\var \, {\rm{div}} \bu+\br \, {\rm{div}} u^0) \ dx\\
			\lesssim &\; (1+ \| \vro^{\var}\|_{L^{\infty}} )\| \nabla_h \cdot \bu_h\|_{L^2} \| {\rm{div}}\bu\|_{L^2} +  \| \nabla_h  \cdot \bu_h\|_{L^2}  \| u^{\var} \|_{L^{\infty}} \| \nabla \br\|_{L^2} + \| \nabla_h  \cdot \bu_h\|_{L^2}  \\&\; \times \big(\| \bu_h \|_{L^2}^{\f12} \| \p_1 \bu_h \|_{L^2}^{\f12} \|\nabla_h \vro^0\|_{H_{tan}^1}^{\f12} \|\p_{3}  \nabla_h \vro^0\|_{H_{tan}^1}^{\f12} + \| \bu_3 \|_{L^2}^{\f12} \| \p_3 \bu_3 \|_{L^2}^{\f12} \|\p_3  \vro^0\|_{H_{tan}^2} \\
			&\; + \| \br\|_{L^2}^{\f12} \|\p_2 \br\|_{L^2}^{\f12} \| \du^0\|_{H_{tan}^1}^{\f12}\| \p_3 \du^0\|_{H_{tan}^1}^{\f12} \big)\\
			\lesssim &\; \nu \| \p_2 \bu\|_{L^2}^2+ \| (\p_1 \bu, {\rm{div}}\bu)\|_{L^2}^2 + {\mathcal{A}(t)} \| ( \br, \nabla \br, \bu )\|_{L^2}^2,
			\dal \deqq
			where we have used the inequality $\eqref{ie:Sobolev}_1$  for the term $\|u^{\var}\|_{L^{\infty}}$.
			Similarly, we can obtain that
			\beqq \bal
			X_2 + X_3 \lesssim &\;  \nu \| (\p_2 \bu, \nabla_h \br)\|_{L^2}^2+ \| (\p_1 \bu, {\rm{div}}\bu, \nabla_h \bb, \p_1^2 \bu, \nabla {\rm{div}}\bu)\|_{L^2}^2 + {\mathcal{A}(t)} \mathcal{\overline{E}}(t)^{\frac12}.
			\dal \deqq
			Combining the estimates from $X_1$ to $X_8$  and using the smallness of $\nu$, we finish the proof of this lemma.
		\end{proof}
		
		Next, we will establish the one-order derivative estimate for the quantity $(\br,\bu, \bb)$ in $L^2-$norm.
		\begin{lemm}
			Under the conditions of Theorem \ref{main_result_three},
			then it holds
			\beq\label{4401}
			\begin{aligned}
				&\; \frac{d}{dt}\Big(\|(\p_i \br,\p_i \bu,  \f{1}{1+\vro^{\var} }\p_i \bb)\|_{L^2}^2 + \i |\p_3 \br|^2 (\vro^0 +\vro^{\var}) \ dx +\i \bu_3 \, \p_3  \br \, \nabla_h \cdot u_h^{\var}\ dx \Big) +\|\p_i(\p_1 \bu, {\rm{div}} \bu,\nabla_h \bb)\|_{L^2}^2 \\
				\lesssim &\; 
				\nu \mathcal{\overline{D}}(t)
				+ (\mathcal{{A}}(t) + \mathcal{{B}}(t))
				\mathcal{\overline{E}}(t)+\var \|(\p_{2}^2 u^\var, \p_3^2 u^\var, \p_3^2b^\var)\|_{H^1}
				\mathcal{\overline{E}}(t)^{\frac12},
			\end{aligned}
			\deq
			where  $i=1,2,3$, and $\nu$ is a small positive constant.
		\end{lemm}
		\begin{proof}
			For $i=1,2,3$, the equation \eqref{401} yields that 
			\beqq
			\begin{aligned}
				&\; \frac{d}{dt}\|(\p_i \br,\p_i \bu,  \f{1}{1+\vro^{\var} }\p_i \bb)\|_{L^2}^2
				+\|\p_i(\p_1 \bu, {\rm{div}} \bu,\nabla_h \bb)\|_{L^2}^2  \\
				= &\; \i\p_i  f_\vro \cdot \p_i \br \ dx +  \i \p_if_u \cdot \p_i\bu \ dx+ \i \f{1}{1+\vro^{\var}} \p_i f_b \cdot \p_i \bb \ dx
				+ \i \f{1}{2(1+\vro^{\var})^2} |\p_i \bb|^2  {\rm{div}}\bu \ dx \\
				&\;- \i \f{1}{2(1+\vro^{\var})^2} |\p_i \bb|^2  f_{\vro} \ dx-\i \f{\vro^{\var}}{1+\vro^{\var}} \Delta_h \p_i \bb \cdot \p_i \bb \ dx+\i \f{\vro^{\var}}{1+\vro^{\var}} (e_2 \p_i{\rm{div}}\bu -\p_2 \p_i\bu) \cdot \p_i \bb \ dx  \\
				&\;+\var \i  \left( \p_i(\f{1}{1+\vro^{\var}}) (\p_2^2 u^{\var} + \p_3^2 u^{\var} ) +\f{1}{1+\vro^{\var}} \p_i(\p_2^2 u^{\var} + \p_3^2 u^{\var} ) \right)\cdot \p_i\bu \ dx + \var \i \f{1}{1+\vro^{\var}} \p_i \p_3^2 b^{\var} \cdot \p_i\bb \ dx \\
				:=&\; \sum_{j=1}^{9} XI_{j}.
			\end{aligned}
			\deqq
			Similar to the estimates in Lemma \ref{lemma41}, it is easy to check that
			\beqq \bal
			XI_4 \lesssim&\; \|  \p_i \bb\|_{L^2} \|\nabla_h \p_i \bb\|_{L^2} \| {\rm{div}}\bu\|_{L^2}^{\f12} \| \p_3 {\rm{div}}\bu\|_{L^2}^{\f12}  \lesssim \nu \| ({\rm{div}}\bu, \p_3 {\rm{div}}\bu, \nabla_h \p_i \bb)\|_{L^2}^2,  \\
			XI_5 = &\; \i \f{1}{2(1+\vro^{\var})^2} |\p_i \bb|^2  (u^\var \cdot \nabla \br+\bu \cdot \nabla \vro^0+\vro^\var \, {\rm{div}} \bu+\br \, {\rm{div}} u^0) \ dx \\
			\lesssim &\; \nu \| \nabla_h \p_i \bb\|_{L^2}^2 + \mathcal{{A}}(t) \| \p_i \bb \|_{L^2}^2,\\
			XI_6 = &\; \i \nabla_h \p_i \bb \cdot \left(\nabla_h (\f{\vro^{\var}}{1+\vro^{\var}}) \p_i \bb + \f{\vro^{\var}}{1+\vro^{\var}} \nabla_h \p_i \bb\right) \ dx 
			\lesssim \nu \| \nabla_h \p_i \bb\|_{L^2}^2 + \mathcal{{A}}(t) \| \p_i \bb \|_{L^2}^2,\\
			XI_7 = &\; \i \f{\vro^{\var}}{1+\vro^{\var}} e_2 \p_i{\rm{div}}\bu  \cdot \p_i \bb \ dx + \i \p_i\bu  \cdot \left(\p_2 (\f{\vro^{\var}}{1+\vro^{\var}})  \p_i \bb + \f{\vro^{\var}}{1+\vro^{\var}} \p_2 \p_i \bb \right) \ dx  \\
			\lesssim &\; \nu \| \p_i( {\rm{div}}\bu, \p_1 \bu, \p_2  \bb)\|_{L^2}^2 + \mathcal{{A}}(t) \| (\p_i \bb, \p_i \bu) \|_{L^2}^2,\\
			\dal \deqq 
			and 
			\beqq \bal
			XI_8 + XI_9 \lesssim &\; \var \| (\p_2^2 u^{\var}, \p_3^2 u^{\var})\|_{H^1} \| \p_i \bu\|_{L^2} +  \var \| \p_i \p_3^2 b^{\var}\|_{L^2} \| \p_i \bb\|_{L^2}.
			\dal \deqq 
			Thus, it remains to estimate the terms $XI_1$ to $XI_3$. We first split $XI_1$ as follows:
			\beqq \bal
			XI_1 = &\;- \i \p_i  (u^\var \cdot \nabla \br) \cdot \p_i \br \ dx  - \i \p_i  ( \bu \cdot \nabla \vro^0+\vro^\var \, {\rm{div}} \bu)  \cdot \p_i \br \ dx - \i \p_i  ( \br \, {\rm{div}} u^0) \cdot \p_i \br \ dx \\
			= &\;\i |\p_i \br|^2 (\f12 \du^{\var} - \du^0) \ dx - \i    \p_i u^\var_h \cdot \nabla_h \br  \cdot \p_i \br \ dx -\i \p_i u^\var_3 \cdot \p_3 \br   \cdot \p_i \br \ dx 
			\\
			&\; - \i \p_i  ( \bu \cdot \nabla \vro^0+\vro^\var \, {\rm{div}} \bu)  \cdot \p_i \br \ dx 
			- \i    \br \, \p_i {\rm{div}} u^0 \cdot \p_i \br \ dx \\
			:= &\; \sum_{j=1}^{5} XI_{1,j}.
			\dal \deqq
			When $i=1,2$, it is easy to check that
			\beqq
			XI_1 \lesssim \nu \| (  \nabla_h \br, \nabla_h \bu, {\rm{div}} \bu, \p_1 \nabla_h \bu, \nabla{\rm{div}} \bu)\|_{L^2}^2 + {\mathcal{A}}(t) \| (\br, \p_3 \br)\|_{L^2}^2.
			\deqq
			When $i=3$, we can obtain that 
			\begin{align*}
				XI_{1,2} \lesssim &\; \| \p_3 \br\|_{L^2}  \| \nabla_h \br\|_{L^2} \|\p_3 u^{\var}\|_{L^2}^{\f18} \|\p_3^2 u^{\var}\|_{L^2}^{\f18} \|\p_3 \nabla_h u^{\var}\|_{H_{tan}^1}^{\f38}  \|\p_3^2 \nabla_h u^{\var}\|_{H_{tan}^1}^{\f38} \\
				\lesssim &\;  \nu \|  \nabla_h \br\|_{L^2}^2 + {\mathcal{A}}(t)  \| \p_3 \br\|_{L^2}^2,\\
				XI_{1,4}  \lesssim &\; \| \p_3 \br\|_{L^2} \Big(\| \p_3 \bu \|_{L^2}^{\f12} \| \p_{13} \bu\|_{L^2}^{\f12} \|\nabla \vro^0\|_{L^2}^{\f14}  \| \p_3 \nabla \vro^0\|_{L^2}^{\f14} \|\p_2 \nabla \vro^0\|_{L^2}^{\f14} \|\p_{23} \nabla \vro^0\|_{L^2}^{\f14} \\
				&\; + \| \bu_h \|_{L^2}^{\f14} \| \nabla_h \bu_h\|_{L^2}^{\f12} \|\p_{12} \bu_h\|_{L^2}^{\f14}  \| \p_3 \nabla_h \vro^0\|_{L^2}^{\f12} \|\p_{3}^2 \nabla_h \vro^0\|_{L^2}^{\f12}\\
				&\; + \| \bu_3 \|_{L^2}^{\f14} \| \p_3 \bu_3\|_{L^2}^{\f14}  \| \p_1 \bu_3\|_{L^2}^{\f14}  \|\p_{13} \bu_3\|_{L^2}^{\f14}  \| \p_3^2 \vro^0\|_{L^2}^{\f12} \|\p_2 \p_{3}^2 \vro^0\|_{L^2}^{\f12} \\
				&\; + \| \p_3  {\rm{div}} \bu\|_{L^2} \|\vro^{\var}\|_{L^{\infty}} + \| {\rm{div}} \bu\|_{L^2}^{\f12} \| \p_3 {\rm{div}} \bu\|_{L^2}^{\f12}  \| \p_3 \vro^\var\|_{L^2}^{\f14} \| \nabla_h \p_3 \vro^\var\|_{L^2}^{\f12} \| \nabla_h^2 \p_3 \vro^\var\|_{L^2}^{\f14}\Big) \\
				\lesssim &\; \nu \| ( \nabla_h \bu, {\rm{div}} \bu, \p_1 \nabla \bu, \nabla{\rm{div}} \bu)\|_{L^2}^2 + ({\mathcal{A}}(t) +{\mathcal{B}}(t))  \|(\bu, \p_3 \bu, \p_3 \br)\|_{L^2}^2, 
			\end{align*}
			and
			\beqq \bal
			XI_{1,5} \lesssim  &\; 
			\| \p_3 \br\|_{L^2}  \|\br\|_{L^2}^{\f12} \|\p_2 \br\|_{L^2}^{\f12} \|\p_3 \du^0\|_{L^2}^{\f14} \|\p_3^2 \du^0\|_{L^2}^{\f14} \|\p_3 \nabla_h \du^0\|_{L^2}^{\f14}  \|\p_3^2 \nabla_h \du^0\|_{L^2}^{\f14}\\ 
			\lesssim &\; \nu \|  \p_2 \br\|_{L^2}^2 + {\mathcal{B}}(t)  \| (\br, \p_3 \br)\|_{L^2}^2.
			\dal \deqq
			Now we estimate the terms $XI_{1,1}$ and $XI_{1,3}$. 
			\beqq \bal
			XI_{1,1} + XI_{1,3} 
			= &\;  -  \f12 \i |\p_3 \br|^2 \du^{\var}\ dx   
			-\i |\p_3 \br|^2  \du^0  \ dx  +  \i |\p_3 \br|^2  \nabla_h \cdot u_h^{\var} \ dx.
			\dal \deqq
			First, we estimate the first term on the right hand, and the second term can be estimated similarly. Substituting the equation $\eqref{eqr}_1$ 
			\beqq
			\du^{\var} = - \p_t \vro^{\var} + F_1^{\var} = - \p_t \vro^{\var} -u^\var\cdot \nabla \varrho^\var-\varrho^\var {\rm div}u^\var   
			\deqq
			into the first term and using the density equation $\eqref{401}_1$, we have
			\beqq \bal
			&\; -\i |\p_3 \br|^2 \du^{\var}  \ dx \\
			= &\;  \f{d}{dt} \i |\p_3 \br|^2 \vro^{\var} \ dx - 2 \i \vro^{\var} \p_3 \br \, \p_3 \p_t \br\ dx +\i |\p_3 \br|^2(u^\var\cdot \nabla \varrho^\var+\varrho^\var {\rm div}u^\var) \ dx \\
			= &\; \f{d}{dt} \i |\p_3 \br|^2 \vro^{\var} \ dx + 2 \i \vro^{\var} \p_3 \br \, \p_3 {\rm div} \bu \ dx + 2 \i \vro^{\var} \p_3 \br \, \p_3 (\bu \cdot \nabla \vro^0+\vro^\var \, {\rm{div}} \bu+\br \, {\rm{div}} u^0) \ dx \\
			&\;+ 2 \i \vro^{\var} \p_3 \br \, \p_3 (u^\var \cdot \nabla \br) \ dx  +\i |\p_3 \br|^2(u^\var\cdot \nabla \varrho^\var+\varrho^\var {\rm div}u^\var) \ dx. \\
			\dal \deqq
			Integrating by parts, we have
			\beqq \bal
			&\; 2 \i \vro^{\var} \p_3 \br \, \p_3 (u^\var \cdot \nabla \br) \ dx  +\i |\p_3 \br|^2(u^\var\cdot \nabla \varrho^\var+\varrho^\var {\rm div}u^\var) \ dx\\
			= &\; 2 \i \vro^{\var} \p_3 \br \, \p_3 u^\var_h \cdot \nabla_h \br \ dx  + 2 \i  |\p_3 \br|^2 \vro^{\var} \, \p_3 u^\var_3  \ dx \\
			\lesssim &\; \| \p_3 \br\|_{L^2} \| \nabla_h \br\|_{L^2} \| \vro^{\var}\|_{L^{\infty}} \| \p_3 u_h^{\var}\|_{L^{\infty}} + \| \p_3 \br\|_{L^2}^2  \| \vro^{\var}\|_{L^{\infty}} (\| \du^{\var}\|_{L^{\infty}} + \| \nabla_h u^{\var}\|_{L^{\infty}})\\
			\lesssim &\; \nu \| \nabla_h \br\|_{L^2}^2  + {\mathcal{A}}(t) \| \p_3  \br\|_{L^2}^2.
			\dal \deqq
			Similarly, we can check that
			\beqq  \bal
			&\; \i \vro^{\var} \p_3 \br \, \p_3 {\rm div} \bu \ dx 
			\lesssim \nu \|\p_3 {\rm div} \bu\|_{L^2}^2 + {\mathcal{A}}(t) \| \p_3 \br\|_{L^2}^2 ,\\
			&\; \i \vro^{\var} \p_3 \br \, \p_3 (\bu \cdot \nabla \vro^0+\vro^\var \, {\rm{div}} \bu+\br \, {\rm{div}} u^0) \ dx 
			\lesssim   \nu \| (   \nabla_h \bu, {\rm{div}} \bu, \p_{13} \bu, \p_3 {\rm{div}} \bu)\|_{L^2}^2 + ({\mathcal{A}}(t) + {\mathcal{B}}(t)) \|(\bu, \p_3 \br)\|_{L^2}^2, 
			\dal \deqq
			which, together with the above estimate, yields that
			\beqq
			-\i |\p_3 \br|^2 \du^{\var}  \ dx \lesssim \f{d}{dt} \i |\p_3 \br|^2 \vro^{\var} \ dx + \nu \| (\nabla_h \br, \nabla_h \bu, {\rm{div}} \bu, \p_{13} \bu, \p_3 {\rm{div}} \bu)\|_{L^2}^2 + ({\mathcal{A}}(t) + {\mathcal{B}}(t)) \|(\bu, \p_3 \br)\|_{L^2}^2. 
			\deqq
			Next, we estimate the third term. Using the equation $\eqref{401}_2$
			\beqq
			\p_3 \br =  \var \f{1}{1+\vro^{\var}} (\p_2^2 u_3^{\var} +\p_3^2 u_3^{\var} )+(f_{u})_3 -\p_t \bu_3+\p_{1}^2 \bu_3+\p_3 {\rm{div}} \bu  - \p_3 \bb_2 +\p_2 \bb_3,
			\deqq
			we have
			\beqq \bal
			&\;  \i  \nabla_h \cdot u_h^{\var} |\p_3 \br|^2  \ dx \\
			= &\; - \f{d}{dt} \i  \nabla_h \cdot u_h^{\var} \, \p_3  \br \, \bu_3 \ dx   +  \i  \nabla_h \cdot u_h^{\var} \, \p_3  \p_t \br \, \bu_3 \ dx +\i  \nabla_h \cdot \p_t u_h^{\var} \, \p_3  \br \, \bu_3 \ dx 
			+ \i  \nabla_h \cdot u_h^{\var} \, \p_3  \br \,  (f_{u})_3  \ dx \\
			&\; + \i  \nabla_h \cdot u_h^{\var} \, \p_3  \br \, \left(\p_{1}^2 \bu_3+\p_3 {\rm{div}} \bu  - \p_3 \bb_2 +\p_2 \bb_3 + \var \f{1}{1+\vro^{\var}} (\p_2^2 u_3^{\var} +\p_3^2 u_3^{\var} )\right) \ dx. \\
			\dal \deqq
			Using the density equation $\eqref{401}_1$ and the velocity equation $\eqref{eqr}_2$, we have
			\beqq \bal
			&\;  \i  \nabla_h \cdot u_h^{\var} |\p_3 \br|^2  \ dx \\ 
			= &\;-  \f{d}{dt} \i  \bu_3 \, \p_3  \br \, \nabla_h \cdot u_h^{\var} \ dx +  \i  \nabla_h \cdot u_h^{\var} \, (\p_3 f_{\vro}-\p_3 {\rm{div}} \bu ) \, \bu_3 \ dx  \\
			&\; -\i  \nabla_h \cdot  (-\p_{1}^2 u_h^\var-\nabla_h {\rm div}u^\var
			+\nabla_h \varrho^\var+\nabla_h b_2^\var-\p_2 b^\var_h) \, \p_3  \br \, \bu_3 \ dx    + \i  \nabla_h \cdot (F_2^\var)_h \, \p_3  \br \, \bu_3 \ dx \\
			&\;  + \i  \nabla_h \cdot u_h^{\var} \, \p_3  \br \,  (f_{u})_3  \ dx  + \i  \nabla_h \cdot u_h^{\var} \, \p_3  \br \, \left(\p_{1}^2 \bu_3+\p_3 {\rm{div}} \bu  - \p_3 \bb_2 +\p_2 \bb_3  \right) \ dx  \\
			&\; +  \var \i  \nabla_h \cdot u_h^{\var} \, \p_3  \br \, \f{1}{1+\vro^{\var}} (\p_2^2 u_3^{\var} +\p_3^2 u_3^{\var} ) \ dx \\
			:= &\; - \f{d}{dt} \i \bu_3 \, \p_3  \br \, \nabla_h \cdot u_h^{\var} \ dx  + \sum_{j=1}^{6} J_{j}.
			\dal \deqq
			Integrating by parts  and using the anisotropic type inequality $\eqref{ie:Sobolev}_3$, we have
			\beqq \bal
			J_1 = &\; \i \nabla_h \cdot u_h^{\var} \, \p_3( u^\var \cdot \nabla \br+ \bu \cdot \nabla \vro^0 + (1+\vro^\var) \, {\rm{div}} \bu + \br \, {\rm{div}} u^0 ) \, \bu_3 \ dx \\
			=&\;  \i  (\p_3 \bu_3 \, \nabla_h \cdot u_h^{\var} +   \bu_3 \, \nabla_h \cdot \p_3 u_h^{\var}  )  ( u^\var \cdot \nabla \br+ \bu \cdot \nabla \vro^0 + (1+\vro^\var) \, {\rm{div}} \bu + \br \, {\rm{div}} u^0 ) \, \ dx \\
			\lesssim &\; \nu \|( \nabla_h \bu, {\rm{div}} \bu, \nabla_h \br ,\p_1 \nabla \bu, \p_2 {\rm{div}} \bu)\|_{L^2}^2 +  {\mathcal{A}}(t) \| ( \bu, \br, \p_3 \br)\|_{L^2}^2,
			\dal \deqq
			and 
			\beqq \bal
			J_2 \lesssim &\; \nu \|( \nabla_h \bu, {\rm{div}} \bu,\p_1 \nabla \bu)\|_{L^2}^2 +  {\mathcal{B}}(t)  \| ( \bu, \p_3 \br)\|_{L^2}^2,\\
			J_3 \lesssim &\; \nu \|( \nabla_h \bu, {\rm{div}} \bu, \p_1 \nabla \bu)\|_{L^2}^2+  {\mathcal{A}}(t) \| ( \bu, \p_3 \br)\|_{L^2}^2 + \var \| (\p_2^2 u^{\var}, \p_3^2 u^{\var})\|_{H^1} \| \p_3 \br\|_{L^2} ,  \\
			J_4 \lesssim &\; \nu \|( \nabla_h(\br,\bb,\bu),  \p_3  {\rm{div}} \bu, \p_1^2 \bu,\nabla_h \p_{3} \bb)\|_{L^2}^2+  ({\mathcal{A}}(t) +{\mathcal{B}}(t) )\| (\bu, \br, \bb, \p_3 \bb, \p_3 \br)\|_{L^2}^2 ,\\
			J_5 \lesssim &\; \nu \|( \p_2 \bb, \p_3  {\rm{div}} \bu, \p_1^2\bu, \nabla_h \p_{3} \bb)\|_{L^2}^2+  ({\mathcal{A}}(t) +{\mathcal{B}}(t) )\| ( \p_3 \bb, \p_3 \br)\|_{L^2}^2 ,\\
			J_6\lesssim &\; \var \| (\p_2^2 u^{\var}, \p_3^2 u^{\var})\|_{H^1} \| \p_3 \br\|_{L^2}.
			\dal \deqq
			Combining the estimates from $J_1$ to  $J_6$, we have
			\beqq \bal
			&\;  \i  \nabla_h \cdot u_h^{\var} |\p_3 \br|^2  \ dx \\ 
			\lesssim &\;  -\f{d}{dt} \i  \bu_3 \, \p_3  \br \, \nabla_h \cdot u_h^{\var}\ dx + \nu \mathcal{\overline{D}}(t)
			+ (\mathcal{{A}}(t) + \mathcal{{B}}(t))
			\mathcal{\overline{E}}(t) + 
			\var \| (\p_2^2 u^{\var}, \p_3^2 u^{\var})\|_{H^1} \| \p_3 \br\|_{L^2}.
			\dal \deqq
			Thus, we can obtain that 
			\beqq \bal
			XI_{1,1} +XI_{1,3} \lesssim &\; -\frac{d}{dt} \i |\p_3 \br|^2 (\vro^0 +\vro^{\var}) \ dx- \f{d}{dt} \i  \bu_3 \, \p_3  \br \, \nabla_h \cdot u_h^{\var}\ dx + \nu \mathcal{\overline{D}}(t) \\
			&\; + (\mathcal{{A}}(t) + \mathcal{{B}}(t))
			\mathcal{\overline{E}}(t) + 
			\var \| (\p_2^2 u^{\var}, \p_3^2 u^{\var})\|_{H^1} \| \p_3 \br\|_{L^2},
			\dal \deqq
			which, together with the term $XI_{1,2}$,  $XI_{1,4}$ and $XI_{1,5}$, yields that
			\beqq \bal
			XI_{1} \lesssim &\; -\frac{d}{dt} \i |\p_3 \br|^2 (\vro^0 +\vro^{\var}) \ dx- \f{d}{dt} \i  \bu_3 \, \p_3  \br \, \nabla_h \cdot u_h^{\var}\ dx + \nu \mathcal{\overline{D}}(t) \\
			&\; + (\mathcal{{A}}(t) + \mathcal{{B}}(t))
			\mathcal{\overline{E}}(t) + 
			\var \| (\p_2^2 u^{\var}, \p_3^2 u^{\var})\|_{H^1} \| \p_3 \br\|_{L^2}.
			\dal  \deqq
			Similarly, we can estimate the terms $XI_2$ and $XI_3$ as follows:
			\beqq
			XI_{2} + XI_{3} \lesssim  \nu \mathcal{\overline{D}}(t) + (\mathcal{{A}}(t) + \mathcal{{B}}(t))
			\mathcal{\overline{E}}(t).
			\deqq
			Combining the estimates from $XI_1$ to $XI_9$, we finish the proof of this lemma.
		\end{proof}
		
		\begin{proof}[\textbf{Proof of Theorem \ref{main_result_three}}]
			Let us define the energy norm
			\beqq \bal
			\mathcal{\overline{E}}_q(t):= &\;  \|( \br, \bu,\f{1}{1+\vro^{\var} } \bb)\|_{L^2}^2 + \|
			( \nabla  \br, \nabla \bu,\f{1}{1+\vro^{\var} }\nabla  \bb)\|_{L^2}^2 + \i |\p_3 \br|^2 (\vro^0+\vro^{\var} ) \ dx\\
			&\; +\i \bu_3 \, \p_3  \br \, \nabla_h \cdot u_h^{\var}\ dx +\kappa_1 \i (\bu_h \cdot \nabla_h \br  + \p_2  \bb \cdot \bu )\ dx,
			\dal \deqq
			where $\kappa_1 $ is a small constant.
			Then, it is easy to check that $\mathcal{\overline{E}}_q(t)$
			is equivalent to $\mathcal{\overline{E}}(t)$.
			Due to the smallness of $\nu$, the combination of estimates
			\eqref{4201}, \eqref{4301} and \eqref{4401}
			yields directly
			\beq \label{402}
			\begin{aligned}
				\frac{d}{dt}\mathcal{\overline{E}}_q(t)
				\le C(\mathcal{{A}}(t) + \mathcal{{B}}(t)) \mathcal{\overline{E}}_q(t) 
				+C\var \|(\p_2^2u^\var, \p_3^2 u^\var, \p_3^2b^\var)\|_{H^1}
				\mathcal{\overline{E}}(t)^{\frac12},
			\end{aligned}
			\deq
			which, together with the Gronwall inequality, yields directly
			\beq\label{403}
			\begin{aligned}
				\mathcal{\overline{E}}_q(t)
				\le
				&C \var \int_0^t \|(\p_2^2u^\var, \p_3^2 u^\var, \p_3^2b^\var)(\tau)\|_{H^1}\mathcal{\overline{E}}(\tau)^{\frac12} d\tau \exp\left\{\int_0^t (\mathcal{{A}}(\tau)+\mathcal{{B}}(\tau)
				) d\tau\right\}.
			\end{aligned}
			\deq
			With the help of the decay estimates
			in Theorems \ref{main_result_one} and \ref{main_result_two}, it is easy to check that for $\zeta\in(0,\f{1}{20})$, 
			\beq\label{404}
			\begin{aligned}
				&\int_0^t \mathcal{{A}}(\tau) d\tau\\
				\lesssim &\; 
				\underset{0\le \tau \le t}{\sup}[(1+\tau)^{\frac{1-\zeta}{2}} \| f(\tau)\|_{H^1}]
				\left\{\int_0^t (1+\tau)^{1-2\zeta} (\| \nabla_h f\|_{H^1}^2+\| g\|_{L^2}^2) d\tau\right\}^{\frac12}
				\left\{\int_0^t (1+\tau)^{-(2-3\zeta)} d\tau\right\}^{\frac12} \\
				&\; +\underset{0\le \tau \le t}{\sup}[(1+\tau)^{\frac{1-\zeta}{8}}\|\p_3 f(\tau)\|_{L^2}^{\f14} \|\p_3^2 f(\tau)\|_{L^2}^{\f14}  ]
				\left\{\int_0^t (1+\tau)^{1-2\zeta} \| \nabla_h \p_3 f\|_{H_{tan}^1}^2 d\tau\right\}^{\frac38}
				\left\{\int_0^t \| \nabla_h \p_3^2 f\|_{H_{tan}^1}^2 d\tau\right\}^{\frac38}\\
				&\; \times 
				\left\{\int_0^t (1+\tau)^{-\f12(4-7\zeta)} d\tau\right\}^{\frac14}  + \int_0^t (\|\nabla_h f\|_{H^2} \|\nabla_h f\|_{H^3} +\|g\|_{H^3}^2) d\tau
				\lesssim 1.
			\end{aligned}
			\deq
			Now we estimate the term $\int_0^t \mathcal{{B}}(\tau) d\tau$. First, we have 
			\beqq \bal 
			&\; \int_0^t \|\nabla_h(f, \p_3 f)(\tau)\|_{H_{tan}^3}^{\f43} d\tau
			\lesssim 
			\left\{\int_0^t (1+\tau)^{1-2\zeta} \|\nabla_h(f,\p_3 f)(\tau)\|_{H_{tan}^3}^2 d\tau\right\}^{\frac23}
			\left\{\int_0^t (1+\tau)^{2(1-2\zeta)} d\tau\right\}^{\frac13} 
			\lesssim  1,
			\dal \deqq
			and similarly, we  have
			\beqq
			\int_0^t \|\nabla_h(g, \p_3 g)(\tau)\|_{H_{tan}^2}^{\f43} d\tau
			\lesssim  1.
			\deqq
			Next, for $m \ge 9$, we can obtain
			\begin{align*}
				&\; \int_0^t \|\p_3^2 f(\tau)\|_{L^2}^{\f45}  \|\nabla_h \p_3^2 f(\tau)\|_{L^2}^{\f45}  d\tau \\
				=  &\; \int_0^t \|\p_3 f(\tau)\|_{L^2}^{\f{4(m-2)}{5(m-1)}} \| \p_3^m f(\tau) \|_{L^2}^{\f{4}{5(m-1)}} \|\nabla_h f(\tau)\|_{L^2}^{\f{4(m-3)}{5(m-1)}} \|\nabla_h \p_3^{m-1} f(\tau)\|_{L^2}^{\f{8}{5(m-1)}}  d\tau \\
				\lesssim &\; \underset{0\le \tau \le t}{\sup}[(1+\tau)^{1-\zeta}\|\p_3 f(\tau)\|_{L^2}^2 ]^{\f{2(m-2)}{5(m-1)}}
				\left\{\int_0^t (1+\tau)^{1-2\zeta} \| \nabla_h  f\|_{L^2}^2 d\tau\right\}^{\f{2(m-3)}{5(m-1)}}
				\left\{\int_0^t \| \nabla_h \p_3^{m-1} f\|_{L^2}^2 d\tau\right\}^{\f{4}{5(m-1)}}\\
				&\; \times 
				\left\{\int_0^t (1+\tau)^{-\f{2(m-2)}{3(m-1)}(1-\zeta) -\f{2(m-3)}{3(m-1)}(1-2\zeta)} d\tau\right\}^{\f35}\lesssim 1,
			\end{align*}
			where we have used the following interpolation inequality,
			\beqq \label{equ-interpolation}
			\| \p_3^2 f \|_{L^2} \lesssim  \|\p_3 f\|_{L^2}^{\f{m-2}{m-1}} \| \p_3^m f\|_{L^2}^{\f{1}{m-1}}, \| \nabla_h \p_3^2 f \|_{L^2} \lesssim  \|\nabla_h f\|_{L^2}^{\f{m-3}{m-1}} \|\nabla_h \p_3^{m-1} f\|_{L^2}^{\f{2}{m-1}},
			\deqq
			and similarly, we have
			\beqq \bal 
			&\; \int_0^t \| \p_3 g(\tau)\|_{L^2}^{\f13} \| \p_3^2 g(\tau)\|_{L^2}^{\f13}\| \nabla_h \p_3 g(\tau)\|_{L^2}^{\f13}\| \nabla_h \p_3^2 g(\tau)\|_{L^2}^{\f13}  d\tau \\
			=  &\; \int_0^t \|\p_3 g(\tau)\|_{L^2}^{\f{2m-3}{3(m-1)}} \| \p_3^m g \|_{L^2}^{\f{1}{3(m-1)}} \| \nabla_h \p_3  g)(\tau)\|_{L^2}^{\f13}  \|\nabla_h g(\tau)\|_{L^2}^{\f{m-3}{3(m-1)}} \|\nabla_h \p_3^{m-1} g(\tau)\|_{L^2}^{\f{2}{3(m-1)}}  d\tau \\
			\lesssim &\; 
			\left\{\int_0^t (1+\tau)^{1-2\zeta} \| (\p_3 g, \nabla_h  g,\nabla_h  \p_3 g)\|_{L^2}^2 d\tau\right\}^{\f{4m-7}{6(m-1)}}
			\left\{\int_0^t \| \nabla_h \p_3^{m-1} g\|_{L^2}^2 d\tau\right\}^{\f{1}{2(m-1)}}\\
			&\; \times 
			\left\{\int_0^t (1+\tau)^{ -\f{4m-7}{2(m-1)}(1-2\zeta)} d\tau\right\}^{\f13}\lesssim 1.
			\dal \deqq
			Thus, together with the estimate \eqref{404}, we have
			\beq \label{405}
			\int_0^t (\mathcal{{A}}(\tau)+\mathcal{{B}}(\tau)
			) d\tau \lesssim 1.
			\deq
			Using the interpolation inequality, we have
			\beq\label{406} \bal
			&\; \|(\p_2^2 u^{\var}, \p_3^2 u^{\var}, \p_3^2 b^{\var})\|_{H^1} \\
			\lesssim &\; \| \p_2^2 u^{\var}\|_{H^1} + 
			\|\p_3^2 (u^{\var}, b^{\var})\|_{L^2} + \|\nabla_h \p_3^2 (u^{\var}, b^{\var})\|_{L^2} + \|\p_3^3 (u^{\var}, b^{\var})\|_{L^2}\\
			\lesssim &\; \| \p_2^2 u^{\var}\|_{H^1}+ \|\p_3 (u^{\var}, b^{\var})\|_{L^2}^{\f{m-2}{m-1}}\|\p_3^m (u^{\var}, b^{\var})\|_{L^2}^{\f{1}{m-1}}+ \|\nabla_h (u^{\var}, b^{\var})\|_{L^2}^{\f{m-3}{m-1}}\|\nabla_h \p_3^{m-1} (u^{\var}, b^{\var})\|_{L^2}^{\f{2}{m-1}}\\
			&\; +  \|\p_3(u^{\var}, b^{\var})\|_{L^2}^{\f{m-3}{m-1}}\|\p_3^{m} (u^{\var}, b^{\var})\|_{L^2}^{\f{2}{m-1}}.
			\dal \deq
			Define $T^{*}:=\var^{-\f{m-1}{m(1-\zeta)}}$.  For $0 < t\le T^{*}$, substituting the estimates \eqref{405} and \eqref{406} into \eqref{403}, we have 
			\beqq \bal
			\mathcal{\overline{E}}(t)
			\lesssim \mathcal{\overline{E}}_q(t) 
			\lesssim &\; \var \int_0^t \me_{tan}^{m-1}(\tau)^{\f{m-3}{2(m-1)}} \me^m(\tau)^{\f{1}{m-1} }\mathcal{\overline{E}}(\tau)^{\frac12} d\tau \exp\left\{\int_0^t (\mathcal{{A}}(\tau)+\mathcal{{B}}(\tau)
			) d\tau\right\}\\
			\lesssim &\; \var \underset{0\le \tau \le t}{\sup}[(1+\tau)^{1-\zeta}\me_{tan}^{m-1}(\tau) ]^{\f{m-3}{2(m-1)}} \underset{0\le \tau \le t}{\sup} \me^m(\tau)^{\f{1}{m-1}} \underset{0\le \tau \le t}{\sup}[(1+\tau)^{1-\zeta} \mathcal{\overline{E}}(\tau) ]^{\f12} 
			\int_0^t (1+\tau)^{-\f{m-2}{m-1}(1-\zeta)} d\tau \\
			\lesssim &\; \var (1+t)^{\f{1+(m-2)\zeta}{m-1}} \lesssim \var^{1-\f{1+(m-2)\zeta}{m (1-\zeta)} }.
			\dal \deqq
			Finally we give the estimate for $T^* \le t < + \infty$, which means $(1+ T^*)^{-1} \le \var^{\f{m-1}{m(1-\zeta)}}$. Integrating the inequality \eqref{402} over $[T^*,t]$, we have
			\begin{align*}
				\mathcal{\overline{E}}(t)
				\le
				C\mathcal{\overline{E}}_q(t)  \lesssim
				&\; \left( \mathcal{\overline{E}}_q(T^*) + \var \int_{T^*}^{t} \|(\p_2^2 u^\var, \p_3^2 u^\var, \p_3^2 b^\var )(\tau)\|_{H^1}\mathcal{\overline{E}}(\tau)^{\frac12} d\tau \right) \exp\left\{\int_{T^*}^{t} (\mathcal{{A}}(\tau)+\mathcal{{B}}(\tau)
				) d\tau\right\} \\
				\lesssim
				&\;  \mathcal{\overline{E}}_q(T^*) + \var^{1-\f{m-3}{2(m-2)}} \int_{T^*}^{t} \md_{tan}^{m-1}(\tau)^{\f{m-3}{2(m-2)}} \me^m(\tau)^{\f{1}{2(m-2)}} \mathcal{\overline{E}}(\tau)^{\frac12} d\tau  \\
				\lesssim &\;\var^{1-\f{1+(m-2)\zeta}{m (1-\zeta)} } + \var^{1-\f{m-3}{2(m-2)}} \underset{T^*\le \tau \le t}{\sup}\me^m(\tau)^{\f{1}{2(m-2)}}  \left\{\int_{T^*}^{t} (1+\tau)^{1-2\zeta} \md_{tan}^{m-1}(\tau) d\tau \right\}^{\f{m-3}{2(m-2)}}  \\ 
				&\; \times   \underset{T^*\le \tau \le t}{\sup}[(1+\tau)^{(1-\zeta)} \mathcal{\overline{E}}(t) ]^{\f12}  \left\{\int_{T^*}^{t} (1+\tau)^{-\f{(m-3)(1-2\zeta) + (m-2)(1-\zeta)}{m-1}} d\tau\right\}^{\f{m-1}{2(m-2)}}\\
				\lesssim &\;  \var^{1-\f{1+(m-2)\zeta}{m (1-\zeta)} } + \var^{1-\f{m-3}{2(m-2)}} (1+ T^*)^{ -\f{m-4 -(3m-8)\zeta}{2(m-2)}}  \\
				\lesssim &\;  \var^{1-\f{1+(m-2)\zeta}{m (1-\zeta)} } + \var^{1-\f{m-3}{2(m-2)}+ \f{m-4 -(3m-8)\zeta}{2(m-2)} \cdot \f{m-1}{m(1-\zeta)} }  
				\lesssim   \var^{1-\f{1+(m-2)\zeta}{m (1-\zeta)} },
			\end{align*}
			where we have used the following inequality
			\beqq \bal
			\|(\p_2^2 u^{\var}, \p_3^2 u^{\var}, \p_3^2 b^{\var})\|_{H^1} 
			\lesssim &\;  \|\p_2^2 u^{\var}\|_{L^2} + \|\p_3^2 (u^{\var}, b^{\var})\|_{H_{tan}^1} +  \|\p_3^3 (u^{\var}, b^{\var})\|_{L^2}\\
			\lesssim &\;  \|\p_2^2 u^{\var}\|_{L^2} + \|\p_3^2 (u^{\var}, b^{\var})\|_{H_{tan}^1} +  \|\p_3^2 (u^{\var}, b^{\var})\|_{L^2}^{\f{m-3}{m-2}} \|\p_3^m (u^{\var}, b^{\var})\|_{L^2}^{\f{1}{m-2}}.
			\dal  \deqq
			Therefore, we complete the proof of Theorem \ref{main_result_three}.
		\end{proof}
		
		\section*{Acknowledgments}
		Jincheng Gao was partially supported by the National Key Research and Development Program of China(2021YFA1002100), Guangdong Special Support Project (2023TQ07A961) and Guangzhou Science and Technology Program (2024A04J6410).
		Xianpeng Hu's research was partially supported by the RFS Grant and GRF Grants from the Research Grants Council (Nos. PolyU 11302021, 11310822 and 11302523).
		Lianyun Peng was partially supported by a fellowship award (PolyU RFS2122-1S05).
		Jiahong Wu was partially supported by the
		National Science Foundation of the United States (DMS 2104682, DMS 2309748).

		\begin{appendices}
			\section{Some useful inequalities}\label{usefull-inequality}	
			Now let us state some  anisotropic Sobolev inequalities used frequently in our paper.
			\begin{lemm}\label{lemm:sobolev-ie}
				For any suitable functions $(f(x), g(x), h(x))$ defined on $\mathbb{R}^3$ and different numbers $i,j,k \in \{1, 2, 3\}$,
				the following estimates hold
				\beq \label{ie:Sobolev}
				\bal
				\|f\|_{L^\infty} &\; \lesssim \|f\|_{L^2}^{\f18}\|\p_1 f\|_{L^2}^{\f18}\|\p_2 f\|_{L^2}^{\f18}\|\p_{12} f\|_{L^2}^{\f18}
				\|\p_3 f\|_{L^2}^{\f18}\|\p_{13} f\|_{L^2}^{\f18}\|\p_{23} f\|_{L^2}^{\f18}\|\p_{123} f\|_{L^2}^{\f18},\\
				\int_{\mathbb{R}^3} |f g h |\, dx
				&\; \lesssim \|f\|_{L^2}
				\|g\|_{L^2}^{\frac12}
				\|\p_i g\|_{L^2}^{\frac12}
				\|h\|_{L^2}^{\frac14}
				\|\p_j h\|_{L^2}^{\frac14}
				\|\p_{k} h\|_{L^2}^{\frac14}
				\|\p_{jk} h\|_{L^2}^{\frac14},\\
				\int_{\mathbb{R}^3} |f g h |\, dx
				&\; \lesssim \|f\|_{L^2}^{\frac12}
				\|\p_1 f\|_{L^2}^{\frac12}
				\|g\|_{L^2}^{\frac12}
				\|\p_2 g\|_{L^2}^{\frac12}
				\|h\|_{L^2}^{\frac12}
				\|\p_3 h\|_{L^2}^{\frac12},\\
				\left\|\|f\|_{L^\infty(\mathbb{R})}\right\|_{L^{\frac2s}(\mathbb{R}^2)}
				&\lesssim \left(\|f\|_{L^2}\|\partial_2 f\|_{L^2}
				+\|\partial_1 f\|_{L^2}\|\partial_{12}f\|_{L^2}\right)^{\frac{1-s}{2}}
				\|f\|_{L^2}^{\frac{2s-1}{2}}
				\|\partial_3 f\|_{L^2}^{\frac12}.
				\dal
				\deq
			\end{lemm}
			
			\begin{proof}
				First of all, one can follow the idea in \cite[Lemma 1.2]{Wu2021Advance} to establish the inequalities $\eqref{ie:Sobolev}_1$-$\eqref{ie:Sobolev}_3$.
				Let us give the proof of inequality $\eqref{ie:Sobolev}_4$.
				Indeed, it holds
				\begin{equation}\label{a1}
					\begin{aligned}
						\left\|\|f\|_{L^\infty(\mathbb{R})}\right\|_{L^{\frac2s}(\mathbb{R}^2)}^{\frac2s}
						\lesssim \left\|\|f\|_{L^2(\mathbb{R})}^{\frac12}
						\|\partial_3 f\|_{L^2(\mathbb{R})}^{\frac12}\right\|_{L^{\frac2s}(\mathbb{R}^2)}^{\frac2s}
						\lesssim \|\partial_3 f\|_{L^2}^{\frac1s}
						\left\|\|f\|_{L^2(\mathbb{R})}\right\|_{L^{\frac{2}{2s-1}}
							(\mathbb{R}^2)}^{\frac1s},
					\end{aligned}
				\end{equation}
				and
				\begin{equation}\label{a2}
					\begin{aligned}
						\left\|\|f\|_{L^2(\mathbb{R})}\right\|_{L^{\frac{2}{2s-1}}(\mathbb{R}^2)}
						^{\frac{2}{2s-1}}
						&\lesssim \left\|\|f\|_{L^\infty(\mathbb{R}^2)}\right\|_{L^2(\mathbb{R})}
						^{\frac{4-4s}{2s-1}}\|f\|_{L^2}^2\\
						&\lesssim \left(\|f\|_{L^2}^{\frac12}\|\partial_2 f\|_{L^2}^{\frac12}
						+\|\partial_1 f\|_{L^2}^{\frac12}\|\partial_{12}f\|_{L^2}^{\frac12}\right)
						^{\frac{4-4s}{2s-1}}\|f\|_{L^2}^2.
					\end{aligned}
				\end{equation}
				The combination of estimates \eqref{a1} and \eqref{a2} gives
				\begin{equation*}
					\begin{aligned}
						\left\|\|f\|_{L^\infty(\mathbb{R})}\right\|_{L^{\frac2s}(\mathbb{R}^2)}
						&\lesssim \|\partial_3 f\|_{L^2}^{\frac12}
						\left\|\|f\|_{L^2(\mathbb{R})}\right\|_{L^{\frac{2}{2s-1}}
							(\mathbb{R}^2)}^{\frac12}\\
						&\lesssim \left(\|f\|_{L^2}\|\partial_2 f\|_{L^2}
						+\|\partial_1 f\|_{L^2}\|\partial_{12}f\|_{L^2}\right)^{\frac{1-s}{2}}
						\|f\|_{L^2}^{\frac{2s-1}{2}}
						\|\partial_3 f\|_{L^2}^{\frac12}.
					\end{aligned}
				\end{equation*}
				Therefore, we complete the proof of this lemma.
			\end{proof}
			
			Next, we introduce the Hardy-Littlewood-Sobolev inequality
			in \cite[pp. 119, Theorem 1]{Stein1970} as follow.
			\begin{lemm}\label{H-L}
				Let $0<\alpha<2 , 1<p<q<\infty, \frac{1}{q}+\frac{\alpha}{2}=\frac{1}{p}$, then
				\begin{equation}\label{a3}
					\|\Lambda_h^{-\alpha}f\|_{L^q(\mathbb{R}^2)}\lesssim \|f\|_{L^p(\mathbb{R}^2)}.
				\end{equation}
			\end{lemm}
			In this paper, taking $q=2$ in \eqref{a3}, then
			$
			p=\frac{1}{\frac12+\frac{\alpha}{2}}=\frac{2}{1+\alpha}
			$
			should satisfy the condition
			$$
			1<p=\frac{2}{1+\alpha} <2.
			$$
			This implies the index $\alpha \in (0, 1)$.
		\end{appendices}
		
		\section*{Data Availability}
		Data sharing is not applicable to this article as no new data were created or analysed in this study.

		\section*{Conflict of interest}
		The authors declared that they have no Conflict of interest to this work.

		\phantomsection

	\end{sloppypar}
\end{document}